\documentclass[11pt,a4paper,twoside,dvips]{article}
\let\textit\textsl
\let\emph\textsl
\let\itshape\slshape
\let\em\slshape

\usepackage{xspace}

\newcommand\indexor{index set\xspace}
\newcommand\indexors{index sets\xspace}
\newcommand\Indexors{Index sets\xspace}

\usepackage{titling}
\usepackage{titletoc}
\usepackage{url}

\usepackage[utf8]{inputenc}
\usepackage[T1]{fontenc}
\usepackage{amsmath}
\usepackage{amsfonts}
\usepackage{amssymb}
\usepackage{amsthm}
\usepackage{color}

\usepackage{mathrsfs}
\usepackage{mathtools}
\newtagform{bold}[\textbf]{\textbf{(}}{\textbf{)}}
\mathtoolsset{mathic}

\usepackage{enumitem}

\mathchardef\ccts="017F
\newcommand\cct{\mathbin{\mkern-4mu{}\ccts{}}}

\usepackage{stmaryrd}
\usepackage[all]{xy}
\usepackage{makeidx}

\usepackage{natbib}
\bibpunct{(}{)}{,}{a}{}{,}
\usepackage[english,french]{babel}\frenchsetup{GlobalLayoutFrench=false}
\usepackage{etoolbox}

\newcommand\s[2]{\suc(#1)_{#2}}
\newcommand \sucr[1]{\s{#1_i}{i\in\In_{#1}}}

\newcommand \Ord {\mathbf{Ord}}
\newcommand \ord {\mathbf{ord}}
\newcommand \LEM {\textbf{LEM}}
\newcommand \LPO {\textbf{LPO}}
\newcommand \LLPO {\textbf{LLPO}}

\newcommand \Fam {\mathrm{Fam}}

\newcommand \Tree {\mathrm{Tree}}
\newcommand \In {\mathrm{In}}

\newcommand \axsuccunaire{{\tt ax8}}
\newcommand \axsuccunair{{\tt ax9}}
\newcommand \axquattre{{\tt ax4}}
\newcommand \axsupsucfini{{\tt ax10}}
\newcommand \axdouze{{\tt ax11}}
\newcommand \axtreize{{\tt ax12}}
\newcommand \ctn{{\tt contraction}}
\newcommand \transu{{\tt trans1}}
\newcommand \transd{{\tt trans2}}
\newcommand \transt{{\tt trans3}}
\newcommand \rfl {{\tt rfl}}
\newcommand \irfl {{\tt irfl}}
\newcommand \sucdef{{\tt sdef}}
\newcommand \succz{{\tt s0}}
\newcommand \succu{{\tt s1}}
\newcommand \supdef{{\tt supdef}}
\newcommand \supz{{\tt sup0}}
\newcommand \wkn{{\tt weakening}}
\newcommand \ept[1]{\(#1\)-order}
\newcommand \epts[1]{\(#1\)-orders}
\newcommand{\sibrouillon}[1]{}
\newcommand\hum[1]{{\sibrouillon{\sf #1}}}
\newcommand\gui[1]{``#1''}

\newcommand \CAdre[2]{%
\begin{center}
\begin{tabular}%
{|p{#1\textwidth}|}
\hline
\vspace{-1.5mm}
#2 
\vspace{1mm}\\ %
\hline
\end{tabular}
\end{center} 
}

\usepackage{framed}
\newcommand\ndsp{\textstyle}

\newcommand\sta{^\star}
\let\sta\etl

\newcommand\eqdefi{\buildrel{\rm def}\over {\;=\;}}
\newcommand\formule[1]{{\left\{ {\arraycolsep2pt\begin{array}{lll} #1 \end{array}}\right.}}

\newcommand \so[1] {\left\{{#1}\right\}} 
\newcommand \sotq[2] {\so{\,#1\,|\,#2\,}} 

\newcommand \und[1] {\underline{#1}}

\newcommand \lora {\longrightarrow}
\newcommand \mt {\mapsto}
\renewcommand \leq{\leqslant}
\renewcommand \preceq{\preccurlyeq}
\renewcommand \geq{\geqslant}

\newcommand \som {\sum\nolimits}
\newcommand\lrb[1] {\llbracket #1 \rrbracket}
\newcommand\lrbn {\lrb{1..n}}
\newcommand\lrbm {\lrb{1..m}}
\newcommand\lrbr {\lrb{1..r}}
\newcommand \dar[1] {\MA{\downarrow \!#1}}
\DeclareRobustCommand{\eoe}{\leavevmode\unskip\penalty9999\hbox{}\nobreak\hfill\quad\(\diamond\)}
\newcommand \NN{\mathbb {N}} 
\newcommand\MA[1]{\mathop{#1}\nolimits}
\newcommand\fF{\mathfrak{F}}

\newcommand \Ack {\MA{\mathrm{Ack}}}
\newcommand \Acko {\underline{\Ack}}
\newcommand \Lst {\MA{\mathrm{Lst}}}

\newcommand \suc {\MA{\mathrm{\vphantom ts}}}

\makeatletter\def\th@plain{\slshape}\makeatother
\makeatletter\patchcmd{\th@remark}{\itshape}{\slshape}{}{}\makeatother

\newcounter{bidon}

\newcommand{\rdb}{\refstepcounter{bidon}}

\usepackage[bookmarksopen=false,breaklinks=true,%
      backref=page,pagebackref=true,plainpages=false,%
      hyperindex=true,pdfstartview=FitH,colorlinks=true,%
      pdfpagelabels=true,linkcolor=blue,%
      citecolor=red,urlcolor=red,
      ]%
   {hyperref}
\usepackage{breakurl}
\usepackage{doi}
\usepackage[capitalise]{cleveref}
\newcommand{\creflastconjunction}{, and }
\crefname{subsection}{Section}{Sections}
\crefname{propdef}{Proposition and Definition}{}
\crefname{property}{Property}{Properties}
\crefname{inequality}{Inequality}{Inequalities}
\crefname{fact}{Fact}{Facts}
\crefname{axiom}{Axiom}{Axioms}
\crefname{equation}{}{}
\creflabelformat{inequality}{(#2#1#3)}

\crefname{fproperty}{propriété}{propriétés}
\crefname{fequation}{}{}
\creflabelformat{finequality}{(#2#1#3)}
\crefname{ffact}{fait}{faits}
\crefname{faxiom}{axiome}{axiomes}
\crefname{ftheorem}{théorème}{théorèmes}
\crefname{flemma}{lemme}{lemmes}
\crefname{fdefinition}{définition}{définitions}
\crefname{fproposition}{proposition}{propositions}
\crefname{fcorollary}{corolaire}{corolaires}
\crefname{fpropdef}{proposition et définition}{propositions et définitions}
\crefname{fconventions}{conventions}{conventions}
\crefname{fremark}{remarque}{remarques}
\crefname{fcomments}{commentaires}{commentaires}
\crefname{fexample}{exemple}{exemples}

\renewcommand\cpageref[1]{page~\pageref{#1}}

\begin{document}
\selectlanguage{english}

\thispagestyle{empty}
~ 
\vspace{3cm}

\noindent In this file you find the English version, starting on the page numbered \pageref{beginenglish}:

\medskip \noindent  {\Large \bf Constructive theory of ordinals}

\medskip This paper has appeared in the book \emph{Mathematics for Computation – M4C} edited by Marco Benini, Olaf Beyersdorff, Michael Rathjen and Peter Michael Schuster (2023). Singapore: World Scientific.

\bigskip \noindent  
Then the French version begins on the page numbered \pageref{beginfrench}:

\medskip\noindent   {\Large \bf Une théorie constructive des ordinaux}

\smallskip \noindent Le lecteur ou la lectrice sera sans doute surprise de l'alternance des sexes ainsi que de l'orthographe du mot \raise.3ex\hbox{$\scriptscriptstyle\langle\!\langle\,$}corolaire\raise.3ex\hbox{$\scriptscriptstyle\,\rangle\!\rangle$}, avec d'autres innovations auxquelles elle n'est pas habituée. En fait, nous avons essayé de suivre au plus près les préconisations de l'orthographe nouvelle recommandée, telle qu'elle est enseignée aujourd'hui dans les écoles en France.  

\bigskip\noindent   {\large \bf Authors}  

\smallskip \noindent Thierry Coquand, Computer Science and Engineering Department, University of Gothenburg, Sweden\\
email: {\tt thierry.coquand@cse.gu.se}

\smallskip \noindent Henri Lombardi, Université de Franche-Comté, CNRS, UMR 6623, LmB, 25000 Besançon, France\\
email: {\tt henri.lombardi@univ-fcomte.fr}

\smallskip \noindent Stefan Neuwirth, Université de Franche-Comté, CNRS, UMR 6623, LmB, 25000 Besançon, France\\
email: {\tt stefan.neuwirth@univ-fcomte.fr}

\thispagestyle{empty}

\pagestyle{headings}
\patchcmd{\sectionmark}{\MakeUppercase}{}{}{}
\setcounter{page}{0}
\renewcommand\thepage{E\arabic{page}}



\begingroup

\theoremstyle{plain}
\newtheorem{theorem}{Theorem}[section]
\newtheorem{lemma}[theorem]{Lemma}
\newtheorem{corollary}[theorem]{Corollary}
\newtheorem{proposition}[theorem]{Proposition}
\newtheorem{propdef}[theorem]{Proposition and definition}
\newtheorem{fact}[theorem]{Fact}
\newtheorem*{propertiesofindexors}{Properties of the set \(\fF\) of \indexors}
\newtheorem*{axiomsfororders}{Axioms for \epts{\fF}}

\theoremstyle{definition}
\newtheorem{definition}[theorem]{Definition}
\newtheorem*{conventions}{Conventions}

\theoremstyle{remark}
\newtheorem{remark}[theorem]{Remark}
\newtheorem*{comments}{Comments}
\newtheorem{example}[theorem]{Example}

\newcommand\Affilfont{\footnotesize}

\title{Constructive theory of ordinals}

\def\proofname{\textsl{Proof}}

\author{Thierry Coquand, Henri Lombardi, Stefan Neuwirth}
\maketitle

\startcontents[english]

\rdb
\label{beginenglish}

\begin{abstract}
\citet{PML} describes recursively constructed ordinals. 
He gives a constructively acceptable version of Kleene's computable ordinals.
In fact, the Turing definition of computable functions is not needed from a constructive point of view. 
We give in this paper a constructive theory of ordinals that is similar to
Martin-Löf's theory, but based only on the two relations \gui{\(x\leq y\)} and  \gui{\(x< y\)}, i.e.\ without considering sequents whose intuitive meaning is a classical disjunction.  In our setting, the operation \gui{supremum of ordinals}   plays an important rôle
through its interactions with the relations \gui{\(x\leq y\)} and  \gui{\(x< y\)}. This allows us to approach as much as we may the notion of linear order when the  property \gui{\(\alpha\leq \beta\) or \(\beta\leq \alpha\)} is provable only within classical logic. 
Our aim is to give a formal definition corresponding to intuition and to prove that our constructive ordinals satisfy constructively all desirable properties. 
\end{abstract}



\setcounter{tocdepth}{4}
\markboth{Contents}{Contents}
\small

\printcontents[english]{}{1}{}
\normalsize

\newpage

\section{Introduction}

This paper is written in the framework of informal constructive mathematics.
We use Bishop's constructive set theory enriched with generalised inductive definitions
(Bishop used this kind of constructions for measure theory, Borel sets, and Lebesgue integration).

In classical mathematics, a natural definition for an ordinal is to be an order type of a well-ordered set (see e.g.\ \citealp[III.2.Ex.14]{bourbaki68}).
Nevertheless it is more convenient to use von Neumann ordinals, for which many results can be proved without using choice (see e.g.\ 
\citealp[Chapitre~2]{Kri} and \citealp[Chapitre~II]{dehornoy17}).

Let us now propose a constructive approach.
A binary relation \(<\) on a set \(X\) is said to be \emph{well-founded} if for any family of sets \((E_x)_{x\in X}\) indexed by \(X\) it is possible to construct elements of \(\prod_{x\in X}E_x\) by \(<\)-induction. Precisely, each time a construction \(\gamma\) is given which from an element \(a\in X\)
and an element \(\varphi\in\prod_{x\in X,x<a}E_x\) constructs an element \(\gamma(a,\varphi)\in E_a\), there exists a unique \(\Phi\in \prod_{x\in X}E_x\) such that for all \(a\in X\) we have \(\Phi(a)=\gamma(a,\Phi|_{x\in X,x<a})\). 
This notion has a clear constructive meaning.

In particular, let us consider a property for elements in \(X\). If the property is \(<\)-hereditary, i.e.\ if it is true for \(a\in X\) as soon as it is true for all \(x\in X\) with \(x<a\), then this property is true for all elements in~\(X\).

In constructive mathematics, \citet*[Section~I.6]{MRR} spell out
well-foundedness in a different but equivalent way and define an ordinal as a linearly ordered set for which the order relation is well-founded. So all
subsets of \(\NN\) are ordinals even if we don't know whether they have a
smallest element.

The \citet[Section 10.3]{hottbook} considers ``Grayson ordinals'' (see
\citealp*[Exercise~I.6.12]{MRR}) in the framework of univalent homotopy type
theory; the ordinals of a given universe turn out to form a set (and not a
groupoid). This theory of ordinals differs from ours with respect to
\cref{axsuccunaire,axsuccunaire2} for \epts{\fF} in the
following.

Among other constructive points of view there are descriptions of countable ordinals constructed by induction
in the works \citealt{brouwer26}, \citealt{gentzen36}, \citealt{church38}, \citealt{kleene38}, \citealt{heyting61}, and \citealt[Chapter 3]{PML}. 

A constructive treatment of von Neumann ordinals based on transfinite recursion is given by \citet[Section 9.4]{aczelrathjen10}.

Brouwer proposes an inductive construction based on the idea that when ordinals  \(\alpha_n\) are defined for all \(n\in\NN\) and are linearly ordered well-founded sets, then we can describe the ordinal \(\alpha\) corresponding intuitively to \(\alpha_1\) followed by \(\alpha_2\) followed by \(\alpha_3\) followed by~\dots. The ordered set \(\alpha\) defined by Brouwer will again be a linearly ordered well-founded set. And if the order relation on each \(\alpha_i\) is decidable, the same is true for \(\alpha\).

Two Brouwer ordinals are in general not comparable (within intuitionistic logic): there is no general criterion allowing us to decide whether two ordinals have the same order type, and, when this is not the case, which is isomorphic to an initial segment of the other.

The paper \citealt*{krausnordvallxu21} compares three distinct constructive approaches to constructive ordinals, denoted by \textsf{Cnf}, \textsf{Brw} and \textsf{Ord}, which are available  in the framework of univalent homotopy type theory. The approach \textsf{Brw} is directly inspired by Brouwer ordinals.

Martin-Löf describes recursively constructed ordinals. 
He gives a constructively acceptable version of Kleene's computable ordinals.
Intuitively, an ordinal à la Martin-Löf is inductively defined 
using the following two basic constructions:
\begin{itemize}
\item there is  a minimum ordinal \(\und 0\);
\item if \((\alpha_n)\) is an explicit sequence of ordinals
(indexed by \(\NN\) or by an \(\NN_k=\sotq{n\in\NN}{n<k}\)), the supremum of the successors of the \(\alpha_n\)'s is an ordinal.\footnote{Martin-Löf denotes this supremum 
by \(\sup(\alpha_n)\). In his setting, \(\und 0\) is in fact the supremum of the empty sequence. Except for this case, his \(\sup(\alpha_n)\) is the supremum of the successors of the \(\alpha_n\)'s; we shall prefer the notation \(\s{\alpha_n}{}\).}  
\end{itemize}
To say that the definition is inductive is to say that every ordinal is constructed using the indicated rules.

In a constructive framework, we can drop Turing machines and replace Turing computability by intuitive (undefined) computability. 
In this case, the main difference between Brouwer and Martin-Löf ordinals
is that Martin-Löf ordinals, being defined in a \gui{parallel} way rather
than in a \gui{sequential} way, are more general: it is possible for any sequence of well-defined ordinals \((\alpha_n)\) to construct the supremum of the successors of the ordinals \(\alpha_n\). A drawback is that there is no way to associate to a Martin-Löf ordinal a linearly ordered well-founded set with the same order type.
For example, if the \(\alpha_n\) are all equal to \(\und 0\) or \(\und 1\), it is a priori impossible to decide whether the supremum of the successors of the \(\alpha_n\)'s equals \(\und 1\) or \(\und 2\).  

\subsection*{Ordinals as trees}

Martin-Löf proposes to visualise an ordinal~\(\alpha\) as a well-founded tree with finite or countable branchings. The ordinal~\(\alpha\) is given with an index set denoted by~\(\In_\alpha\); in the sequel, it will be an element of the set \(\fF_2\) of index sets consisting of~\(\NN\) and its finite subsets~\(\NN_k\).
\begin{itemize}
\item The tree with only its root represents \(\und 0\).
\item    If \((t_i)_{i\in \In_{\alpha}}\) is a family of ordinal trees 
for a family of ordinals \((\alpha_i)_{i\in \In_{\alpha}}\), the supremum \(\alpha=\s{\alpha_i}{i\in \In_{\alpha}}\)  of the successors of the \(\alpha_i\)'s is given by the ordinal tree for which there are \(\#\In_\alpha\) branches above the root and a copy of \(t_i\) is attached to the branch indexed by \(i\in\In_\alpha\).
\end{itemize}

Consider the trees in Figure~\ref{efigure}.

If \(n\in\NN\), the ordinal \(\und n\) can be represented by the tree with \(n\)~successive unary branchings at \(n\)~nodes, so that it has \(n+1\)~nodes. 

The first infinite ordinal \(\omega\) can be represented by the tree that has a countable branching above the root, the branches bearing the preceding trees (representing \(\und n\), \(n\in\NN\)).

Its successor, denoted by \(\omega+\und1\), can be represented by the tree with unary branching above the root, the branch bearing the preceding tree.

The ordinal \(\omega+\und2\) can be represented by the tree with unary branching above the root, the branch bearing the preceding tree.

The ordinal \(\omega+\omega\) can be represented by the tree that has a countable branching above the root, the branches bearing the trees representing \(\omega+\und n\), \(n\in\NN\).
\begin{figure}[!h]
  \centering
  \includegraphics[width=\textwidth,trim=4 4 6 5
  ]{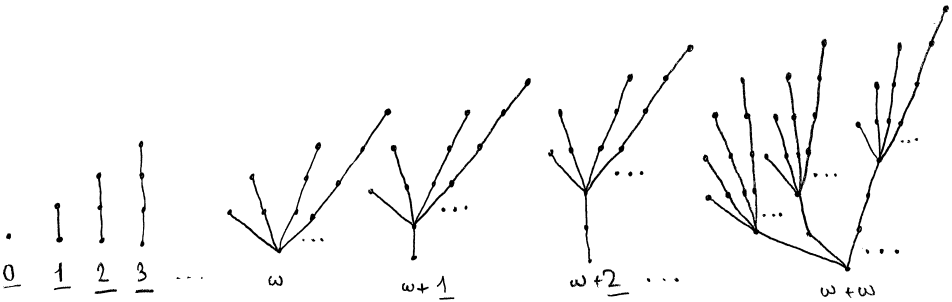}
  \caption{Ordinal trees.}
  \label{efigure}
\end{figure}

More formally, such a tree can be defined as the set of its nodes, or branching points, suitably named. We may consider the set \(\Lst(\NN)\) of finite lists  of elements of \(\NN\). 
Let \(n\in\NN\) and \(\ell\), \(\ell'\in\Lst(\NN)\). We denote by \(n\cct \ell\) the list \([n,\ell_1,\dots,\ell_k]\), where \(\ell=[\ell_1,\dots,\ell_k]\), by \(\ell\cct n\) the list \([\ell_1,\dots,\ell_k,n]\), and by \(\ell\cct \ell'\) the concatenation of the lists \(\ell\) and \(\ell'\).

We remark that \(\Lst(\NN)\) can be enumerated in a natural way\footnote{For example, for \(\ell=[\ell_1,\dots,\ell_k]\in\Lst(\NN)\), we let \(\mu(\ell)=\sum_{i=1}^{k}(\ell_i+1)\) and we enumerate the lists by increasing \(\mu(\ell)\).} and that the notion of an \(\NN\)-indexed family in  \(\Lst(\NN)\) corresponds, via such an enumeration, to the basic (undefined) notion of map from~\(\NN\) to~\(\NN\).

A well-founded tree with finite or countable branchings can then be described as a detachable subset~\(T\) of \(\Lst(\NN)\) which is inductively constructed according to the previously indicated process. 
\(T\) is closed by initial segments: if \(\ell\in\Lst(\NN)\), \(p\in\NN\), and \(\ell\cct p\in T\), then  \(\ell\in T\). 
Thus, to each ordinal \(\alpha\), we are associating a tree, defined as a suitable subset of \(\Lst(\NN)\), denoted by \(\Tree(\alpha)\).

If \(n\in\NN\), the ordinal \(\und n\) can be described by the finite sequence of \(n+1\) lists \([\,]\), \([0]\), \([0,0]\), \dots, \([0,\dots,0]\). 

The first infinite ordinal \(\omega\) can be described by the subset of \(\Lst(\NN)\) enumerated by the infinite sequence \([\,]\), \([0]\), \([1]\), \([1,0]\), \([2]\), \([2,0]\), \([2,0,0]\), \([3]\), \([3,0]\), \([3,0,0]\), \([3,0,0,0]\), etc.

The ordinal \(\omega+\und1\) can be described by the infinite sequence \([\,]\), \([0]\), \([0,0]\), \([0,1]\), \([0,1,0]\), \([0,2]\), \([0,2,0]\), \([0,2,0,0]\), \([0,3]\), \([0,3,0]\), \([0,3,0,0]\), \([0,3,0,0,0]\), etc.

The ordinal \(\omega+\und2\) can be described by the infinite sequence \([\,]\), \([0]\), \([0,0]\), \([0,0,0]\), \([0,0,1]\), \([0,0,1,0]\), \([0,0,2]\), \([0,0,2,0]\), \([0,0,2,0,0]\), \([0,0,3]\), \([0,0,3,0]\), \([0,0,3,0,0]\), \([0,0,3,0,0,0]\), etc.

\begin{sloppypar}
The ordinal \(\omega+\omega\) can be described by the doubly infinite sequence  \([\,]\), \([0]\), \([0,0]\), \([0,1]\), \([0,1,0]\), \([0,2]\), \([0,2,0]\), \([0,2,0,0]\), \([0,3]\), \([0,3,0]\), \([0,3,0,0]\), \([0,3,0,0,0]\), etc., \([1]\), \([1,0]\), \([1,0,0]\), \([1,0,1]\), \([1,0,1,0]\), \([1,0,2]\), \([1,0,2,0]\), \([1,0,2,0,0]\), \([1,0,3]\), \([1,0,3,0]\), \([1,0,3,0,0]\), \([1,0,3,0,0,0]\),  etc., \([2]\), \([2,0]\), \([2,0,0]\), \([2,0,0,0]\), \([2,0,0,1]\), \([2,0,0,1,0]\), \([2,0,0,2]\), \([2,0,0,2,0]\), \([2,0,0,2,0,0]\), \([2,0,0,3]\), \([2,0,0,3,0]\), \([2,0,0,3,0,0]\), \([2,0,0,3,0,0,0]\), etc., etc.
\end{sloppypar}

These trees, seen as subsets of~$\Lst(\NN)$ defined by induction, form a well-defined set in the context of intuitive constructive mathematics. It can be denoted by~$\ord_2$ (see \cref{deford}). 
Note that from a constructive point of view, $\ord_2$ is a discrete set if, and only if, the Markov principle is valid.
This set $\ord_2$ is a \gui{set of ordinal names} in \citealt{PML}. And the set of Martin-Löf ordinals, \(\Ord_2^\mathrm{ML}\), is a quotient of $\ord_2$ by a correctly proved equivalence relation. The set \(\Ord_2^\mathrm{ML}\) and our set \(\Ord_2\) are discrete if the little principle of omniscience \LPO\ is valid.  See \cref{subsecOML} for more details.

\smallskip\centerline{*\ *\ *}\smallskip


We give in this paper a constructive theory of ordinals that is similar to Martin-Löf's theory, but based only on the two relations \gui{\(x\leq y\)} and  \gui{\(x< y\)}, i.e.\ without considering sequents whose intuitive meaning is a classical disjunction.

In our setting, the operation \gui{supremum of ordinals}   plays an important rôle through its interactions with the relations \gui{\(x\leq y\)} and  \gui{\(x< y\)}. This allows us to approach as much as we may the notion of linear order when the  property \gui{\(\alpha\leq \beta\) or \(\beta\leq \alpha\)} is provable only within classical logic.
In the same way, the impossibility of constructively proving  linear order for real numbers is circumvented by the introduction of \(x< y\), \(x\leq y\) and \(\sup(x,y)\), which are all three indispensable. 

Our problem is to give a formal definition corresponding to intuition and to prove that our constructive ordinals satisfy constructively all desirable properties. 

\smallskip\centerline{*\ *\ *}\smallskip

The first step in \cref{secEpto} is to describe these desirable properties.  

\hum{À COMPLÉTER}

\section{Linear orders associated to a set of \indexors}\label{secEpto}
We define in this section the structure of linear orders associated to a set~\(\fF\) of \indexors, \(\fF\)-orders for short.

\subsection{\Indexors} 
First we need a set \(\fF\) of \indexors. An \indexor will be denoted by \(I\), \(J\), \(K\),  \(I'\), \(I''\), \(J'\),  \(I_a\), \(I_b\), etc. 

An \emph{\indexor} is simply a set that will be used as a set of indices
for the families we shall consider. In the sequel, a finitely enumerated subset of \(A\) is always a subset of \(A\) defined à la Bishop by a map \(\NN_k\to A\). If \(A\) is discrete, a finitely enumerated subset of \(A\)
is a detachable subset.

\begin{propertiesofindexors}
We will assume that 
\begin{itemize} 
\item \(\NN\) and the finite sets \(\NN_k=\sotq{n\in\NN}{n<k}\) (\(k\geq 0\))   are elements of \(\fF\);  
\item 
any finitely enumerated subset\footnote{By definition this is a subobject  given by a function \(\NN_k\to \fF\).} of an element of \(\fF\) is isomorphic\footnote{In the category of sets.} to an element of \(\fF\);
\item if \(J\in\fF\), the set of finitely enumerated subsets of \(J\) is isomorphic to an element of \(\fF\);  
\item \(\fF\) is closed by disjoint unions indexed by \(\fF\): we will denote by
 \(I+J\) a disjoint union of \(I\) and \(J\), and by \(\sum_{i\in I} J_i\) a disjoint union of the family \((J_i)_{i\in I}\). 
\end{itemize}
\end{propertiesofindexors}

Disjoint unions are to be understood as direct sums in the category of sets.
The disjoint union \(J=\sum_{i\in I} J_i\) comes with a family 
 \(\iota_\ell\colon J_\ell \to J\) of injective maps realising
 \(J\) as the direct sum of the \(J_i\)'s in the category of sets.

If we restrict ourselves to countable ordinals, we can take for \(\fF\)
 the set  
\[\fbox{\(\fF_2=\sotq{\NN_k}{k\in\NN, k\geq0} \cup \so\NN\)}\]
with convenient operations  for the set of finite subsets of an \(I\in \fF\) and for disjoint unions of elements of~\(\fF\) indexed by an element of~\(\fF\).
Any other set~\(\fF\) of \indexors will contain \(\fF_2\).    

An \emph{\(\fF\)-indexed family  of elements of \(E\)} is a family \((x_i)_{i\in I}\), where \(I\in\fF\) and the \(x_i\)'s \(\in E\). The set of \(\fF\)-indexed families  of elements of \(E\) is denoted by \(\Fam(\fF,E)\). 

We shall restrict the use of subscripts for ordinal variables to this meaning, and use superscripts for all other uses.

\subsection{Axioms} 
A  structure of \emph{\ept{\fF}} on a set~\((E,=)\) is given as \((E,<,\leq,0_E,\sup,\suc)\), where 
\begin{itemize}
\item \(<\) and \(\leq\) are binary relations defined on \((E,=)\);
\item  \(0_E\) is an element of \(E\) and we let \(E\sta=\sotq{\alpha\in E}{0_E<\alpha}\);
\item  \(\sup\) is a map from \(\Fam(\fF,E\sta)\) to \(E\sta\): taking as input  an element \((\alpha_i)_{i\in I}\) of \(\Fam(\fF,E\sta)\), it constructs an element of \(E\sta\) denoted by \(\alpha=\sup(\alpha_i)_{i\in I}\);
\item   \(\suc\) is a unary map from  \(E\) to \(E\sta\):  taking as input  an element \(\beta\in E\), it constructs an element of \(E\sta\) denoted by \(\suc(\beta)\).
\end{itemize}

\begin{definition} \label{defibinarysup}
In order to write axioms with finite \(\sup\)'s,  we define \(\sup(\alpha,\beta)\) for \(\alpha,\beta\in E\) in the following way (using implicitly \cref{ax15}):   \(\sup(0_E,\alpha)=\alpha=\sup(\alpha,0_E)\); if \(\alpha,\beta\in E\sta\), 
\(\sup(\alpha,\beta)\) is already defined.
\end{definition}

These data are to satisfy the following axioms. 

%
\begin{axiomsfororders}
\begin{enumerate}[label=\textit{\arabic*.},ref=\textit{\arabic*}]
\item[]
\item \label[axiom]{refantisym} \(\alpha=\beta\) if and only if \(\alpha\leq \beta\) and \(\beta\leq \alpha\) \ (reflexivity and antisymmetry);
\item \label[axiom]{zero}\(0_E\leq \alpha\);
\item \label[axiom]{ax3} if \(\alpha<\alpha\) then \(0_E=\beta\) \ (irreflexivity);  
\item \label[axiom]{axltle} if \(\alpha<\beta\) then \(\alpha\leq \beta\);  
\item\label[axiom]{transun} if \(\alpha\leq \beta\) and \(\beta\leq \gamma\), then \(\alpha\leq \gamma\) \ (transitivity 1); 
\item\label[axiom]{transdeux} if \(\alpha<\beta\) and \(\beta\leq \gamma\), then \(\alpha<\gamma\) \ (transitivity 2);
\item\label[axiom]{transtrois} if \(\alpha\leq \beta\) and \(\beta< \gamma\), then \(\alpha< \gamma\) \ (transitivity 3);
\item \label[axiom]{axsuccunaire} \(\alpha<\suc(\beta)\) if and only if 
\(\alpha\leq \beta\) \ (using \cref{refantisym} this gives \(\alpha<\suc(\alpha)\));
\item \label[axiom]{axsuccunaire2} \(\suc(\beta)\leq \alpha\) if and only if \(\beta<\alpha\);
\item \label[axiom]{supsucfini} if \(\alpha<\gamma\) and \(\beta<\gamma\), then \(\sup(\alpha,\beta)<\gamma\);
%
\item \label[axiom]{ax10} if \(\alpha<\sup(\alpha,\beta)\) then \(\alpha<\beta\);
\item  \label[axiom]{ax11}  if \(\gamma<\alpha\) and \(\alpha\leq\sup(\beta,\gamma)\), then \(\alpha\leq \beta\);
\item \label[axiom]{axsup} for  \((\alpha_i)_{i\in I}\in \Fam(\fF,E\sta)\) and   \(\beta\in E\),  we have
\[
  \alpha_i\leq \beta\text{ for all }i\in I\ \text{ if and only if }\ \sup(\alpha_i)_{i\in I}\leq \beta
\]
(characteristic property of \(\sup\));
\item \label[axiom]{ax14} if \(\gamma<\beta\) for all \(\gamma<\alpha\), then \(\alpha\leq \beta\);
\item  \label[axiom]{ax15} either \(\alpha\leq 0_E\) or \(0_E<\alpha\).
\end{enumerate}
\end{axiomsfororders}

The \emph{category of \epts{\fF}}  is defined by its morphisms 
\[
  \ndsp
(E,<_E,\leq_E,0_E,\sup_E,\suc_E) \lora (F,<_F,\leq_F,0_F,\sup_F,\suc_F)\text,
\]  
which are maps from \(E\) to \(F\) preserving the structure (in the usual meaning).

\begin{comments}
  \begin{enumerate}[label=\arabic*)]
  \item Let \(\gamma\in E\sta\) and \((\alpha_n)_{n\in \NN}\) such that \(\alpha_n=\gamma\) or \(\alpha_n=\suc(\gamma)\) for each~\(n\). The element \(\sup(\alpha_n)_{n\in \NN}\) hesitates between \(\gamma\) and \(\suc(\gamma)\). Thus there is no hope that
the disjunction \gui{\(\alpha\leq \beta\) \hbox{or \(\beta< \alpha\)}}
be constructive for arbitrary elements \(\alpha,\beta\ne0_E\). Consequently, we have introduced the \(\sup\) map together with its axioms in order to best describe in what sense the order can be thought of as linear.
Perhaps this is not optimal (reasonable axioms, satisfied for the set~\(\Ord_2\) of ordinals of the second class constructed in \cref{ConstOrd},  might be missing).

\item The irreflexivity is given a form that, instead of stating a negation, allows \(E\) to reduce to a singleton. This happens if and only if \(0_E=\suc(0_E)\), which implies \(0_E<0_E\)  using \cref{axsuccunaire}. 

\item \cref{ax15} expresses that \(\so{0_E}\) is detachable. 
This contrasts with the fact that  
 elements other than~\(0_E\) do not define detachable singletons. We have defined \(\sup\) on \(E\sta\) rather than on \(E\) in order to satisfy constructively the disjunction of \cref{ax15}. 

\item The characteristic property of \(\sup\) shows that this law  satisfies idempotence as well as generalised associativity and commutativity.
\eoe
\end{enumerate}
\end{comments}

\subsection{Some properties}

\begin{propdef}[generalising \cref{defibinarysup}] \label{propdefsupfini}\leavevmode\\*  
For \(\alpha^1,\dots,\alpha^r\in E\) we let
\[
  \sup(\alpha^1,\dots,\alpha^r)\eqdefi
\formule{0_E \hbox{ if } \alpha^1=\dots=\alpha^r=0_E\\[.5em]
\hbox{the \(\sup\) of the \(\alpha^k\ne0_E\) otherwise.}}
\] 
The characteristic property of \(\sup\) is satisfied:
\[
  \alpha^1\leq \beta\hbox{ and }\dots\hbox{ and }\alpha^r\leq \beta\ \text{ if and only if }\ \sup(\alpha^1,\dots,\alpha^r)\leq \beta\text.
\] 
\end{propdef}

\begin{fact} \label{factsuciso} Let \(\alpha,\beta\) be elements of \(E\).
\begin{itemize}
\item \(\suc(\alpha)<\suc(\beta)\) if and only if \(\alpha<\beta\).
\item \(\suc(\alpha)\leq \suc(\beta)\) if and only if \(\alpha\leq \beta\).
\end{itemize}
\end{fact}
\begin{proof}
Use \cref{axsuccunaire,axsuccunaire2}.  
\end{proof}
%

\begin{fact} \label{factreverse1013}
\Cref{supsucfini,ax10,ax11,ax14} are in fact equivalences:
\begin{itemize}
\item [\ref{supsucfini}.] \(\alpha<\gamma\) and \(\beta<\gamma\) hold simultaneously if and only if \(\sup(\alpha,\beta)<\gamma\); 
\item [\ref{ax10}.] \(\alpha<\sup(\alpha,\beta)\) if and only if \(\alpha<\beta\);
\item [\ref{ax11}.]   if \(\gamma<\alpha\), then \(\alpha\leq\sup(\beta,\gamma)\) holds if and only if \(\alpha\leq \beta\);
\item  [\ref{ax14}.]  \(\alpha\leq \beta\) if and only if \(\gamma<\beta\) for all \(\gamma<\alpha\).
\end{itemize}
\end{fact}
\begin{proof}
Use the transitivities  and the characteristic property of \(\sup\).
\end{proof}
%

\begin{fact}[\(\suc\) commutes with finite \(\sup\)'s, notation as in {\cref{propdefsupfini}}] \label{factsupsuc}\leavevmode
We have 
\(\sup(\suc(\alpha),\suc(\beta))=\suc(\sup(\alpha,\beta))\) and
more generally  \(\sup(\suc(\alpha^1),\allowbreak\dots,\suc(\alpha^r))=\suc(\sup(\alpha^1,\dots,\alpha^r))\).
\\
In particular, if \(\alpha^1<\gamma\), \dots, \(\alpha^r<\gamma\), then \(\sup(\alpha^1,\dots,\alpha^r)<\gamma\).
\end{fact}
\begin{proof} It suffices to prove \(\suc(\sup(\alpha,\beta))=\sup(\suc(\alpha),\suc(\beta))\). We have the following chain of equivalences: \(\suc(\sup(\alpha,\beta))\leq\gamma \iff\)
\(\sup(\alpha,\beta)<\gamma \iff\)
\((\alpha<\gamma\) and \({\beta<\gamma})\iff\)
 \((\suc(\alpha)\leq\gamma\)  and \(\suc(\beta)\leq\gamma)\iff\) \(\sup(\suc(\alpha),\suc(\beta))\leq \gamma\).
\end{proof}

\begin{propdef}[definition of infinitary \(\suc\) and its characteristic property] \label{factsupsucinfinis} 
For any  \((\alpha_i)_{i\in J}\in\Fam(\fF,E)\), we define 
\(\s{\alpha_i}{i\in J}=\sup(\suc(\alpha_i))_{i\in J}\).
Then we get the following equivalence:
\[\alpha_i<\beta\text{ for all }i\in J\ \text{ if and only if }\ \s{\alpha_i}{i\in J}\leq \beta \text. 
\]
\end{propdef}
\begin{proof} Use \cref{axsup,axsuccunaire2}.
\end{proof}
%


We write \fbox{\(F\subseteq_f I\)} in order to express that \(F\) is a finitely enumerated 
subset 
of~\(I\). 

\begin{fact} \label{factltle1} Let \(\alpha,\beta^1,\dots,\beta^m\in E\).%
\begin{enumerate}[label=\textit{\arabic*.},ref=\textit{\arabic*}]
\item\label{factlte1-1} Assume that \(\alpha=\s{\alpha_i}{i\in J}\) with \((\alpha_i)_{i\in J}\in\Fam(\fF,E)\) and that \(\alpha_i<\sup(\beta^1,\allowbreak\dots,\beta^m)\) for all \(i\in J\). Then \(\alpha\leq \sup(\beta^1,\dots,\beta^m)\).
\item\label{factlte1-2} Assume that \(\beta^k=\suc ((\beta^k)_i)_{i\in J_k}\) with \(((\beta^k)_i)_{i\in J_k}\in\Fam(\fF,E)\) for \(k\in\lrbm\).
\\
Let \hbox{\(F_1\subseteq_f J_1, \dots, F_m\subseteq_f J_m\)} not all be empty.  If 
\[
  \alpha\leq \sup((\beta^k)_j)_{k\in\lrbm,\;j\in F_k},
\]  
then \(\alpha< \sup(\beta^1,\dots,\beta^m)\).
\end{enumerate}
\end{fact}
\begin{proof}
\ref{factlte1-1}. This is \cref{factsupsucinfinis}.
\\
\ref{factlte1-2}. Suppose e.g.\ that \(F_1\) is nonempty. Then
\(\alpha\leq \sup((\beta^1)_j)_{j \in F_1}<\beta^1\leq \sup(\beta^1,\allowbreak\dots,\allowbreak\beta^m)\). The strict inequality comes from \cref{factsupsuc} because all \((\beta^1)_j\)'s are \(<\beta^1\) by \cref{factsupsucinfinis}.
\end{proof}
%

\section{Inductive construction of ordinals}\label{ConstOrd}

\CAdre{.66} {In \cref{ConstOrd,PropfondOrd}, the set \(\fF\) of \indexors is fixed but often implied.}
 
We shall define a set of ordinals \(\Ord\) (more precisely \(\Ord_\fF\)) 
and we shall prove that it is an initial object in the category of \epts{\fF}.

First we define a set \(\ord\) of names for \(\fF\)-indexed ordinals by an inductive definition.
The simplest inductive definition of an infinite set is that of \(\NN\): it admits an element \(0\) and a successor map \(x\mt s(x)\colon \NN\to \NN\).
The inductive definition of \(\ord\) is very similar to that of \(\NN\). 
In \(\NN\), each element is either  \(0\) or an \(s(x)\) for an~\(x\in\NN\). 
Similarly, in \(\ord\), 
each element is either \(\und0\) or the \(\suc\) of an \(\fF\)-indexed family in \(\ord\); we denote by \(\ord\sta\) the set of elements of this second type. 

\begin{definition} \label{deford}
The set \(\ord\)  (more precisely \(\ord_\fF\)) is defined in an inductive way: it is to admit a distinguished  element \(\und0\) and  a map 
\[\suc\colon\Fam(\fF,\ord)\to\ord\text.\]
\end{definition}
\noindent N.B.: The only constraint in this inductive definition is that \(\suc\) be indeed a map from~\(\Fam(\fF,\ord)\) to \(\ord\). 

An element of \(\ord\)
will be called \emph{[name of an] ordinal} in the sequel.

When \(\fF=\fF_2\), we get  the set of names of countable ordinals, denoted by \(\ord_2\).
%
\begin{remark} \label{remdeford}
Each  element \(\alpha\in\ord\sta\) 
is given with two data: 
\begin{itemize}
\item the \indexor used in the definition of \(\alpha\): it will be denoted by~\(\In_\alpha\);
\item the family \(\chi_\ord(\alpha,i)_{i\in\In_\alpha}\) of its \emph{definitional subordinals}, i.e.\ the element of \(\Fam(\fF,\ord)\) such that \(\alpha=\s{\chi_\ord(\alpha,i)}{i\in\In_\alpha}\).
\end{itemize}
Thus the inductive  definition of \(\ord\) implies the existence 
of a map \(\alpha\mt\In_\alpha\colon\ord\sta\to\fF\) and 
the existence of a dependent
family \((\alpha,i)\mt\chi_\ord(\alpha,i)\) which is defined for \(\alpha\in\ord\sta\) and \(i\in\In_\alpha\). 
 In order to make the text more readable, we shall perform a slight abuse of notation: we shall not mention the construction of the dependent family \(\chi_\ord\), and the notation \(\alpha_i\) will be an abbreviation for
 \(\chi_\ord(\alpha,i)\). 
 With these conventions we may write \fbox{\,\(\alpha=\s{\alpha_i}{i\in\In_{\alpha}}\)}. \eoe
 
\end{remark}

For   \(\alpha^1,\dots,\alpha^r\in\ord\) we define  \fbox{\(\suc(\alpha^1,\dots,\alpha^r)=\s{\alpha^i}{i\in\lrbr}\)}.

In particular, if  \(\alpha\in\ord\),  its \emph{immediate successor}
  \hbox{\(\suc(\alpha)\)} is the element 
\(\beta=\s{\beta_i}{i\in\In_{\beta}}\), where \(\In_\beta=\NN_1=\so 0\) and \(\beta_0=\alpha\).
The sequence \((\und m)_{m\in\NN}\)  in \(\ord\) is defined inductively by
  \(\und{m+1}=\suc(\und m)\). 
Then we can define \fbox{\(\omega=\s{\und{n}}{n\in\NN}\)}.

In order to prove a property for \(\alpha=\s{\alpha_i}{i\in\In_\alpha}\), it is sufficient to prove the property for each~\(\alpha_i\). 
In a similar way we can construct inductively a map whose domain is \(\ord\), or define inductively a predicate on~\(\ord\).
This is stated precisely in \cref{factsousord} and done e.g.\ in \cref{defsousord,defsupord} and more generally throughout the rest of this article.

\subsection{Subordinals} 
Here is a correct inductive definition.

\begin{definition} \label{defsousord} 
Let \(\alpha=\s{\alpha_i}{i\in\In_\alpha}\in\ord\sta\). An element \(\beta\) of \(\ord\) is a \emph{definitional subordinal} of \(\alpha\) if \(\beta=\alpha_i\) for an \(i\in \In_\alpha\): we write this \(\beta\lessdot_1 \alpha\). An element \(\gamma\) is a \emph{subordinal of \(\alpha\)} if it is a definitional subordinal of~\(\alpha\) or a subordinal of a definitional subordinal of \(\alpha\).
We write this \(\gamma\lessdot \alpha\).
\end{definition}

Thus \(\und0\) is the only element of \(\ord\) which has no subordinal.

The following fact acknowledges that the definition of the relations $\cdot\lessdot_1\cdot$ and $\cdot\lessdot\cdot$ is a correct inductive definition on $\ord$.


\begin{fact} \label{factsousord}
The relations \(\lessdot_1\) and \(\lessdot\)  on \(\ord\) are well-founded.
\end{fact}

Consequently there is no infinite branch in the tree of subordinals
of an element of \(\ord\), in the following sense.

\begin{fact} \label{factSousord}
A sequence \((\alpha^j)_{j=1,2,\dots}\) in \(\ord\), where each \(\alpha^{j+1}\) is a subordinal of \(\alpha^{j}\), reaches in a finite number of steps \(\alpha^r=\und0\).
\end{fact}

Remark that in order to perform a construction (or a proof) by \(\lessdot_1\)-induction or by  \(\lessdot\)-induction, the case \(\und0\) has to be dealt with separately since it has no subordinal.
Nevertheless, we shall be able to avoid this case distinction until considering ordinal arithmetic on \cpageref{subsec-arit-ord}.

\subsection{Definition of the \texorpdfstring{\(\sup\)}{sup} law }

\begin{definition} \label{defsupord}
\begin{enumerate}
\item 
The law   \(\sup\colon\Fam(\fF,\ord\sta)\to\ord\sta\)  is defined in the following way.
%
Let \((\alpha^j)_{j\in J}\) be a family in \(\ord\sta\) with \(J\in\fF\).
If \(\alpha^j=\s{(\alpha^j)_i}{i \in I_{j}}\), then \fbox{\(\sup(\alpha^j)_{j\in J}\)} is the element \(\varepsilon=\s{\varepsilon_k}{k \in K}\), where 
\begin{itemize}
\item \(K\) is the disjoint union of the \(I_{j}
  \)'s;
\item \((\varepsilon_{k})_{k \in K}\) is the family defined by  
\(\varepsilon_k=(\alpha^j)_i\) if \(\iota_j(i)=k\) 

\end{itemize}
(here \(\iota_j\colon I_{j}
\to K\) is the injective map from \(I_{j}
\) to the disjoint union of the \(I_{j}
\)'s). 
We shall write \(\sup(\alpha^j)_{j\in\lrbr}=\sup(\alpha^1,\dots,\alpha^r)\). 

\item The \(\sup\) of a finite family in \(\ord\) is defined in the following way.
\[
  \sup(\alpha^1,\dots,\alpha^r)\eqdefi
\formule{\und0 \;\hbox{ if } \alpha^1=\dots=\alpha^r=\und0
\\[.5em]
\hbox{the \(\sup\) of the \(\alpha^k\in\ord\sta\) otherwise.}}
\] 
\end{enumerate}
\end{definition}

We note that Item 2 is formally included in Item 1 if we adopt the convention \(\In_{\und0}=\NN_0\). However, this convention would not allow us to define an arbitrary \(\fF\)-indexed \(\sup\) in  \(\ord\).

\subsection{\texorpdfstring{Definition of \(\leq \) and of \(<\)}{Definition of <= and of <}} 
The main job remains to be done, i.e.\ to define two binary relations \(\leq \) and \(<\) on \(\ord\) with the required properties, viz.
\begin{itemize}
\item the relation \gui{\(\alpha\leq \beta\) and \(\beta\leq \alpha\)} has to be an equivalence relation  (we shall denote by \(\Ord\) the quotient set);
\item the relations \(\leq \) and \(<\) and the maps \(\sup\) and \(\suc\) have to descend to the quotient (we shall not change their names), i.e.\ they have to be compatible with the equivalence relation;
\item with these maps and relations, \(\Ord\) has to be an \ept{\fF}.
\end{itemize}

Moreover, since the map \(\suc\colon\Fam(\fF,\ord)\to\ord\sta\) is defined before the map \(\sup\colon\Fam(\fF,\ord\sta)\to\ord\sta\), we have to verify in our construction that \cref{axsup} is satisfied in~\(\Ord\). This will be a consequence of \cref{factsucc} in the following.

For our job, we define inductively two asymmetric relations
between, on the left side, an element of \(\ord\) and, on the right side, \emph{a nonempty finitely enumerated set} 
of elements of \(\ord\), written as a list: 
\CAdre{.7}{\centerline{
\(\alpha\leq \beta^1,\dots,\beta^m\)\quad \hbox{  and  }\quad \(\alpha<\beta^1,\dots,\beta^m\)\quad (\(m\geq 1\)).}
\vspace{-.8em}
}

\begin{conventions}
  \begin{itemize}
  \item The letters \(\alpha\), \(\beta\), \(\gamma\),
    \(\varepsilon\), possibly with exponents, indices or primes, are
    used for elements of \(\ord\).
  \item If \(\alpha\) is an element of \(\ord\) and if \(F\) is a finite list, possibly empty, in \(\In_\alpha\),
we denote by \(\alpha_F\) the list of the \(\alpha_i\)'s with \(i\) in \(F\).
\end{itemize}
\end{conventions}

The two relations are defined by simultaneous induction in the following way.
\begin{framed}%
\noindent Particular cases involving \(\und0\) are avoided by using the convention \(\In_{\und0}=\NN_0\). Let \(m\) be an integer \(\geq 1\). \smallskip

\noindent\(\alpha\leq \beta^1,\dots,\beta^m\)\quad
is defined as\quad\(\alpha_i<\beta^1,\dots,\beta^m\) for all \(i\in \In_\alpha\).

\noindent\(\alpha< \beta^1,\dots,\beta^m\)\quad
is defined as\quad
\begin{tabular}[t]{@{}l@{}}
  there are \(F_1\subseteq_f \In_{\beta^1},\dots,F_m\subseteq_f \In_{\beta^m}\)\\
  not all empty with \(\alpha\leq \beta^1_{F_1}, \dots,\beta^m_{F_m}\).
\end{tabular}
\end{framed}
This definition is correct since elements of \(\ord\)
are inductively defined and the pair of clauses is inductive. 

Without the convention that \(\In_{\und0}=\NN_0\), we would have had to include \cref{factOr00}
below in the definition. This convention is a little miracle allowing us to avoid a case-by-case reasoning with respect to the disjunction  \gui{\(\alpha= \und0\) or \(\alpha\in\ord\sta\)} in the proofs.
 
The meaning of the two relations is \(\alpha\leq \sup(\beta^1,\dots,\beta^m)\)
and \(\alpha< \sup(\beta^1,\allowbreak\dots,\beta^m)\).

\begin{lemma} \label{lemsupetlist}
We have \(\alpha<\beta^1,\dots,\beta^m\) if and only if \(\alpha< \sup(\beta^1,\dots,\beta^m)\). 
\\
Similarly, we have \(\alpha\leq \beta^1,\dots,\beta^m\) if and only if \(\alpha\leq \sup(\beta^1,\dots,\beta^m)\).
\end{lemma}
\begin{proof}
Let us write 
\[
  \begin{aligned}
    \alpha\prec \beta^1,\dots,\beta^m&\text{ for }\alpha< \sup(\beta^1,\dots,\beta^m)\text,\\
    \alpha\preceq \beta^1,\dots,\beta^m&\text{ for }\alpha\leq \sup(\beta^1,\dots,\beta^m)\text.
\end{aligned}
\]
Let \(\varepsilon=\sup(\beta^1,\dots,\beta^m)\). Then \(\alpha\prec \beta^1,\dots,\beta^m\) if and only if \(\alpha\leq\varepsilon_F\) with \(F\) a nonempty finitely enumerated subset of the disjoint union~\(K\) of the \(\In_{\beta^j}\)'s and \(\varepsilon_k=(\beta^j)_i\) if \(k\) is the image of~\(i\) in \(K\); letting \(F_j=F\cap\In_{\beta^j}\), not all~\(F_j\)'s are empty and this may be rewritten as \(\alpha\leq \beta^1_{F_1}, \dots,\beta^m_{F_m}\). This holds if and only if \(\alpha< \beta^1,\dots,\beta^m\).

We have \(\alpha\preceq \beta^1,\dots,\beta^m\) if and only if, for all \(i\in \In_\alpha\), \(\alpha_i<\varepsilon\), i.e.\ \(\alpha_i\prec\beta^1,\dots,\beta^m\), i.e.\ \(\alpha_i<\beta^1,\dots,\beta^m\); this holds if and only if \(\alpha\leq \beta^1,\dots,\beta^m\).
\end{proof}

The relation \(\alpha=_\Ord\beta\) is defined as meaning \gui{\(\alpha\leq \beta\) and \(\beta\leq \alpha\)}. 

We shall show in \cref{PropfondOrd} that the relation \(\cdot=_\Ord\cdot\) is an equivalence relation and we shall define the set \(\Ord\) as the quotient of \(\ord\) by this relation.

Let us note that until \cref{thOrd1}, the symbol \(=\) between two elements of \(\ord\) is the equality in \(\ord\) and has not the meaning of  \(=_\Ord\). Nevertheless, after having shown that the
relations and the laws of \(\ord\) descend to the quotient \(\Ord\), the statements with the symbol \(=\) will also work for the symbol \(=_\Ord\).

\subsection{Finite ordinals, bounded ordinals}

We start with a few properties of~\(\und0\).

\begin{fact} \label{factOr00}  Let \(m\) be an integer \(\geq 1\), \(\alpha,\beta^1,\dots,\beta^m \in\ord\), and \(\gamma\in\ord\sta\). We have
\begin{enumerate}[label=\textit{\arabic*.},ref=\textit{\arabic*}]
\item\label{factOr001}  \(\und0\leq \beta^1,\dots,\beta^m\);
\item  \(\und0< \gamma,\beta^2,\dots,\beta^m\);
\item \(\alpha<\underbrace{\und0,\dots,\und0}_{m\text{ times}}\) is impossible.
\end{enumerate}
\end{fact}

\begin{proof}
This is straightforward from the definitions.
\end{proof}

\begin{remark} \label{rem-ax15}
  \cref{ax15} 
will be valid  in \(\Ord\) because every  element of \(\ord\) is given either as~\(\und0\) or as an element \(\gamma\in\ord\sta\), so that always \(\und0<\gamma\) by Item \emph{2} of \cref{factOr00}.\eoe
\end{remark}

\begin{fact} \label{factnatord}
Let \(m,n\in\NN\). Then
\begin{enumerate}[label=\textit{\arabic*.},ref=\textit{\arabic*}]
\item \(m\leq n\) if and only if \(\und m\leq \und n\);
\item \(m< n\) if and only if \(\und m< \und n\); 
\item \(\und m\leq \und n\) and \(\und n< \und m\) are incompatible.
\end{enumerate}
\end{fact}
\begin{proof}
Concerning the direct implications in \emph{1} and \emph{2}, we write \(n=m+r\) and we do an induction on~\(r\). 
For the reverse implications, cases \(m=0\) and \(n=0\) are already known. 
Next, we see that \(\und{m+1}\leq \und{n+1}\) implies 
\(\und{m}\leq \und{n}\), and that \(\und{m+1}< \und{n+1}\) implies 
\(\und{m}< \und{n}\). This allows us to conclude by induction on \(m\).
\\
Item \emph{3} follows from Items \emph{1} and \emph{2}. 
\end{proof}

An element \(\alpha\in\ord\) is said to be \emph{finite} if \(\alpha=_\Ord\und m\)  for an \(m\in\NN\),   \emph{bounded} if \(\alpha\leq \und m\) for an \(m\in\NN\). Bounded ordinals are much more complicated than finite ordinals (see \cref{exaLLPO,exa123}).

In \cref{in-class-math}, we shall discuss what the relations \(\leq\)~and~\(<\) on the set~\(\ord_\fF\) become in classical mathematics.

\subsection{First consequences} 

The following fact shows that the \(\suc\) law will satisfy the characteristic property given in \cref{axsup} when we shall know that it descends to the quotient \(\Ord\).
\begin{fact}[\sucdef] \label{factsucc}
We have \(\alpha\leq \beta\) if and only if \(\alpha_i<\beta\) for all \(i\in \In_\alpha\).
\end{fact}
\begin{proof}
This property is tautological: this is the definition of \(\alpha\leq \beta\). 
\end{proof}

Similarly, the following fact shows that the \(\sup\) law will satisfy the  characteristic property given in \cref{axsup} when we shall know that it descends to the quotient \(\Ord\). 

\begin{fact}[\supdef] \label{factsup}
Let \((\alpha^j)_{j\in J}\) be a family in \(\ord\sta\) with \(J\in\fF\), \(\gamma=\sup(\alpha^j)_{j\in J}\), and \(\beta\in\ord\).\\
We have \(\gamma\leq \beta\) if and only if \(\alpha^j\leq \beta\) for all \(j\in J\).
In particular, \(\sup(\alpha,\beta)\leq \beta \) if and only if \(\alpha\leq \beta\).
\end{fact}

\noindent N.B.: The result is equally true for the \(\sup\) of a finite family in \(\ord\).

\begin{proof} This is another linguistic tautology.
We have \(\alpha^j=\s{(\alpha^j)_i}{i \in I_{j}}\) for an \(I_j\in\fF\).
By the definition of~\(\gamma\) and of~\(\leq\), the inequality \(\gamma\leq \beta\) means that for each \(j\in J\) and each \(i\in I_j\)  we have 
\((\alpha^j)_i<\beta\), i.e.\ that for each \(j\in J\) we have \(\alpha^j\leq \beta\). 
\end{proof}
The following fact shows that \cref{axsuccunaire,axsuccunaire2}  will be valid when we shall descend to the quotient \(\Ord\).
\begin{fact} \label{factOr01} 
\begin{enumerate}[label=\textit{\arabic*.},ref=\textit{\arabic*}]
\item\label{factOr01-1} \axsuccunaire.
We have \(\alpha<\suc(\beta)\) if and only if \(\alpha\leq \beta\).
\item\label{factOr01-2} \axsuccunair.
We have \(\beta<\alpha\) if and only if \(\suc(\beta)\leq \alpha\).
\end{enumerate}
\end{fact}
\begin{proof}
Recall that the element \(\gamma=\suc(\beta)\) is defined by \(\In_\gamma=\so 0\) and \(\gamma_0=\beta\). 

\noindent\ref{factOr01-1}. By definition, \(\alpha<\gamma\) means that \(\alpha\leq \gamma_F\) for a nonempty list \(F\subseteq_f\so0\). This forces \(F=[0]\) and \(\gamma_F=\beta\).

\noindent\ref{factOr01-2}. By definition, \(\gamma\leq \alpha\) means that \(\gamma_0<\alpha\), i.e.\ \(\beta<\alpha\).

\noindent Thus, better than equivalences, these are tautologies.
\end{proof}

The following fact will allow us to shorten certain proofs by induction.

\begin{fact} \label{factOr1}
\begin{enumerate}
%
\item [a.] We have an inequality \(\; \alpha\leq \beta\) if and only if for each \(i\in \In_\alpha\), 
there exists a nonempty  \(F_i\subseteq_f \In_\beta\)  such that \(\;\alpha_i\leq \beta_{F_i}\).
\item [b.] We have an inequality \(\; \alpha< \beta\) if and only if there exists 
a nonempty \(F\subseteq_f \In_\beta\)  such that  for each \(i\in \In_\alpha\) we have  \(\;\alpha_i< \beta_{F}\).
\end{enumerate}
\end{fact}
\begin{proof}
Straightforward from the definitions.
\end{proof}

Now we leave behind tautological proofs and turn to inductive proofs.

\penalty-2500
\begin{fact} \label{factOr0}
\begin{itemize}
%
\item \wkn. If \(\alpha\leq \beta^1,\dots,\beta^m\), then for each \(\beta\)
we have  \(\alpha\leq \beta,\beta^1,\dots,\beta^m\).
\item \ctn. If \(\alpha\leq \beta^1,\beta^1,\beta^2,\dots,\beta^m\) then \(\alpha\leq \beta^1,\beta^2,\dots,\beta^m\).
\item The same properties hold with \(<\) instead of \(\leq \).
\end{itemize}
\end{fact}
\begin{proof}
Use induction applying the definitions.
\end{proof}

The following lemma is a  corollary of \cref{factOr1}.  Item~\ref{corfactOr1-succz} (resp.~\ref{corfactOr1-supz}) 
will imply that the \(\suc\) (resp.~\(\sup\))
map descends to the quotient in \(\Ord\) (resp.\ 
\(\Ord\sta\)).
Item~\ref{corfactOr1-rfl} will imply that the relations \(\leq \) and~\(=\) are reflexive in \(\Ord\); Items~\ref{corfactOr1-irfl} and~\ref{corfactOr1-ax14} will imply \cref{ax3,ax14} for \(\Ord\).

\begin{lemma} \label{corfactOr1}
\begin{enumerate}[label=\textit{\arabic*.},ref=\textit{\arabic*}]
%
\item\label{corfactOr1-succz} \succz. Let \(\alpha,\beta\in\ord\) with \(\In_\alpha=\In_\beta\) and \(\alpha_i\leq \beta_i\) for all \(i\in \In_\alpha\). Then \(\alpha\leq \beta\).

\item\label{corfactOr1-supz} \supz. Let \(\alpha,\beta\in\ord\sta\) with \(\In_\alpha=\In_\beta\) and \(\alpha_i\leq \beta_i\) for all \(i\in \In_\alpha\). Then
\[\sup(\alpha_i)_{i\in\In_\alpha}\leq \sup(\beta_i)_{i\in\In_\beta}\text.\] 
The result works also for the \(\sup\) of a finite family in \(\ord\).
\item\label{corfactOr1-rfl} \rfl. For all \(\alpha\in\ord\), we have \(\alpha\leq \alpha \).
A fortiori, \(\alpha\leq \alpha,\beta^1,\dots,\beta^m\).
\item\label{corfactOr1-succu} \succu. For all \(\alpha\in\ord\sta\) and all \(i\in \In_\alpha\), we have \(\alpha_i<\alpha\). A fortiori, \(\alpha_i< \alpha,\beta^1,\dots,\beta^m\).
\item\label{corfactOr1-irfl} \irfl. For all \(\alpha\in\ord\), \(\alpha<\alpha\) is impossible.
\item \label{corfactOr1-37} \(\alpha < \suc(\alpha)\).
\item\label{corfactOr1-ax14} {\tt ax14}.  If \(\gamma<\beta\) for all \(\gamma<\alpha\), then \(\alpha\leq \beta\).
\end{enumerate}
\end{lemma}
\begin{proof}

\noindent \ref{corfactOr1-succz}. Straightforward from \cref{factOr1}\emph{a}.
We take \(F=\so{i}\).

\noindent
\ref{corfactOr1-supz}. Let \(\gamma=\sup(\alpha_i)_{i\in\In_\alpha}\)
and \(\epsilon=\sup(\beta_i)_{i\in\In_\beta}\). By \cref{factOr1}\emph{a}, for each~\(j\in\In_{\alpha_i}\) there exists a nonempty  \(F_{i,j}\subseteq_f \In_{\beta_i}\) with \((\alpha_i)_j\leq (\beta_i)_{F_{i,j}}\); \(F_{i,j}\) is a fortiori in the disjoint union of the \(\In_{\beta_i}\)'s, so that \((\alpha_i)_j<\epsilon\) by definition of~\(\epsilon\). By definition, \(\alpha_i\leq\epsilon\), so that by \cref{factsup} \(\gamma\leq\epsilon\).

\noindent
\ref{corfactOr1-rfl}. By induction: we use \cref{factOr1}\emph{a}, we take \(F=\so{i}\) and \(\alpha\leq\alpha\) reduces to \(\alpha_i\leq \alpha_i\).

\noindent
\ref{corfactOr1-succu}. By induction: we use \cref{factOr1}\emph{b}, we take \(F=\so{i}\) and \(\alpha_i<\alpha\) reduces to \((\alpha_i)_j< \alpha_i\).

\noindent
\ref{corfactOr1-irfl}. By induction: we use \cref{factOr1}\emph{b}, we take \(F=\so{i}\) and \gui{\(\alpha< \alpha\) is impossible} reduces to: \gui{\(\alpha_i< \alpha_i\) is impossible}.

\noindent
\ref{corfactOr1-37}.
Apply  \succu\ to \(\beta=\suc(\alpha)\).

\noindent
\ref{corfactOr1-ax14}.
If \(\alpha=\und0\), the conclusion is clear. If \(\alpha=\s{\alpha_i}{i\in I}\), as \(\alpha_i<\alpha\) for each \(i\in\In_\alpha\) (Item~\ref{corfactOr1-succu}), the hypothesis that \(\gamma<\beta\) for all \(\gamma<\alpha\) shows that \(\alpha_i<\beta\) for all \(i\in \In_\alpha\).
We conclude by~\cref{factsucc} that \(\alpha\leq \beta\).
\end{proof}

\subsection{Ordinals and limited principles of omniscience}
\label{sec:princ-omnisc}

\begin{example} \label{exaLLPO}
Let \((v_n)_{n\in \NN}\) be a sequence in \(\so{0,1}\) which takes at most once the value \(1\). The lesser limited principle of omniscience \LLPO\ says that
we have 
\[
  \exists k\in \so{0,1}\ \forall n\ \ (v_n=1 \Rightarrow \;n\equiv k\mod 2).\eqno(*)
\] 
From such a sequence \((v_n)_{n\in \NN}\) let us define \(\varepsilon\), \(\varepsilon^1\) and \(\varepsilon^2\in\ord\) in the following way:
\[
  \varepsilon=\s{\und{v_n}}{n\in\NN},\quad\varepsilon^1=\s{\und{v_{2m}}}{m\in\NN},\quad\varepsilon^2=\s{\und{v_{2m+1}}}{m\in\NN}.
\] 
Then we have \(\varepsilon\leq \sup(\varepsilon^1,\varepsilon^2)\). But 
\(\varepsilon\leq\varepsilon^1\) gives \(k=0\) in \((*)\) and \(\varepsilon\leq\varepsilon^2\) gives \(k=1\) in \((*)\). Thus, the disjunction \(\varepsilon\leq\varepsilon^1\) or \(\varepsilon\leq\varepsilon^2\) has no constructive proof: assuming the disjunction for an arbitrary \((v_n)\) would imply \LLPO.
\end{example}

\begin{example} \label{exa123}
Let \((u_n)_{n\in \NN}\) be a nondecreasing sequence in \(\so{0,1}\). The limited principle of omniscience \LPO\ says that such a sequence is eventually constant: 
\[
  \exists n\in\NN\ \forall m\in\NN\ u_m\leq u_n. \eqno(*)
\]   
From such a sequence \((u_n)_{n\in \NN}\) let us define \(\alpha\) and \(\beta\in\ord\) in the following way: 
\[
  \alpha=\s{\und{u_n}}{n\in\NN},\qquad    \beta=\s{\und{u_n+1}}{n\in\NN}.
\]  
We note that the strict inequality \(\alpha<\beta\) is equivalent (using \cref{factnatord,factOr1,lemsupetlist}) to
\[
  \exists n\in\NN\ \forall m\in\NN\ u_m< u_n+1,
\] 
which amounts to \((*)\).
In fact, \(\alpha\) hesitates between \(1\) and \(2\), \(\beta\) hesitates between \(2\) and~\(3\), and the inequality \(\alpha<\beta\)
is valid if we assume \LPO\@. But asserting \(\alpha<\beta\) for all sequences \((u_n)\) implies  \LPO\ in constructive mathematics.
Here we see that hesitating between \(1\) and \(2\) for an infinite sequence has the same flavour as
hesitating (in a classical setting) between bounded and unbounded for an infinite sequence of natural numbers: adding \(1\) to each term of the sequence increases strictly the \(\sup\) only if the sequence is bounded.   
\end{example}

\subsection{In classical mathematics}\label{in-class-math}

\Cref{propOrdclass} shows that the law of excluded middle (\LEM) simplifies and/or obscures dramatically the structure of the set \(\ord_\fF\) with respect to the relations~\(<\) and~\(\leq\). 
\begin{proposition} \label{propOrdclass}
Assume \LEM\@. Then for \(\alpha,\beta\in\ord\), we have  \(\alpha\leq \beta\) or \(\beta<\alpha\). Moreover, if \(\beta<\alpha\), there exists an \(i\in \In_\alpha\) such that \(\beta\leq \alpha_i\).
\end{proposition}
\begin{proof}
We prove by simultaneous induction the two following properties. 

\smallskip \centerline{\gui{\(\alpha\leq \beta\) or \(\beta<\alpha\)}\quad  and \quad \gui{\(\beta\leq \alpha\) or \(\alpha<\beta\)}.}

\smallskip \noindent By induction hypothesis, we have  for all \(i\in \In_\alpha\) and all \(j\in \In_\beta\),
\gui{\(\alpha\leq \beta_j\) or \(\beta_j<\alpha\)}, and also  \gui{\(\beta\leq \alpha_i\) or \(\alpha_i<\beta\)}. 
\\
The first disjunction implies by \LEM\ that either \(\beta_j<\alpha\) for all \(j\in \In_\beta\) or there is \(j\in \In_\beta\) such that \(\alpha\leq\beta_j \). In the first case, we have  \(\beta\leq \alpha\) by definition of \(\cdot\leq \cdots\). In the second case, we have  \(\alpha<\beta\) by definition of \(\cdot< \cdots\), with for \(F\subseteq_f \In_\beta\) the list \([j]\).
\\
The symmetric reasoning yields the second disjunction.
\end{proof}
\noindent N.B.: For countable ordinals, the limited principle of omniscience (\LPO) suffices to prove the proposition.

\begin{corollary} \label{cor1propOrdclass}
Assume \LEM\@. Any ordinal \(\alpha\ne\und0\) is either an immediate successor or the \(\sup\) of the ordinals \(\gamma<\alpha\).
\end{corollary}
\begin{proof}
Consider \(\alpha=\s{\alpha_i}{i\in\In_\alpha}\) and compare \(\alpha\) with \(\sup(\alpha_i)_{i\in\In_\alpha}\). The details are left to the reader.
\end{proof}
%

\begin{corollary} \label{corpropOrdclass}
Assume \LEM\@. Any bounded ordinal is finite.
\end{corollary}
\begin{proof}
Left to the reader: use \cref{factOr01}.
\end{proof}
%

\section{Fundamental results}\label{PropfondOrd}

\subsection{\texorpdfstring{\(\Ord_\fF\) is an initial object in the category of  \epts{\fF}}{Ord\_F is an initial object in the category of  F-orders}}
%
\begin{lemma} \label{lemsupsucc} For \(\alpha^1,\dots,\alpha^r\) in \(\ord\) (\(r\geq 1\)), we have \[\fbox{\(\sup(\alpha^j)_{j\in\lrbr}<\allowbreak\s{\alpha^j}{j\in\lrbr}\).}\] 
\end{lemma}
\begin{proof} Let us show e.g.\ that \(\epsilon=\sup(\alpha,\beta)<\gamma=\suc(\alpha,\beta)\). We have \(\In_{\epsilon}=\In_\alpha+\In_\beta\), with \(\epsilon_k=\alpha_i\) if \(\iota_1(i)=k\), and \(\epsilon_k=\beta_j\) if \(\iota_2(j)=k\). We have \(\In_\gamma=\so{1,2}\)
with \(\gamma_1=\alpha\) and \(\gamma_2=\beta\). We apply \cref{factOr1}\emph{b} with \(F=\so{1,2}\). For an arbitrary  \(k\) in \(\In_\epsilon\), we have \(\epsilon_k<\alpha,\beta\) since \(\epsilon_k\) is \(\alpha_i\) or \(\beta_j\) and, by \succu, we have \(\alpha_i<\alpha\) (a fortiori \(\alpha_i<\alpha,\beta\)) and \(\beta_j<\beta\) (a fortiori \(\beta_j<\alpha,\beta\)). 
\end{proof}
Let us note that the preceding proof relies on the fact that the definitions of \(\leq \) and \(<\) have been given with lists on the right-hand side.

\begin{lemma}[transitivities] \label{lemtrans}
\begin{enumerate}
\item \transu. If \(\alpha\leq \beta^1,\dots,\beta^m\) and, for each
\(j\in\lrbm\), \(\beta^j\leq \gamma^1,\dots,\gamma^r\), then  
\(\alpha\leq \gamma^1,\dots,\gamma^r\).
\item \transd. If \(\alpha< \beta^1,\dots,\beta^m\) and, for each
\(j\in\lrbm\), \(\beta^j\leq \gamma^1,\dots,\gamma^r\), then  
\(\alpha< \gamma^1,\dots,\gamma^r\).
\item \transt. If \(\alpha\leq \beta^1,\dots,\beta^m\) and, for each
\(j\in\lrbm\), \(\beta^j< \gamma^1,\dots,\gamma^r\), then  
\(\alpha< \gamma^1,\dots,\gamma^r\).
\end{enumerate}
\end{lemma}
As particular cases,  \cref{transun,transdeux,transtrois}  will be valid when we shall descend to the quotient \(\Ord\):
\begin{itemize}
\item if \(\alpha\leq \beta\) and \(\beta\leq \gamma\)
then \(\alpha\leq \gamma\);
\item   if \(\alpha< \beta\) and \(\beta\leq \gamma\)
then \(\alpha< \gamma\);
\item   if \(\alpha\leq \beta\) and \(\beta< \gamma\)
then \(\alpha< \gamma\).
\end{itemize}
\begin{proof} The three transitivities are being proved by simultaneous induction.

  In order to prove \transu, we note that the hypothesis means that we have \(\alpha_i<\beta^1,\dots,\beta^m\) for all \(i\in\In_\alpha\). Let us fix such an \(i\). We use  \transd\ with this \(\alpha_i\) instead of \(\alpha\) and we get \(\alpha_i<\gamma^1,\dots,\gamma^r\). Since this works for all \(i\in\In_\alpha\), this gives the desired conclusion \(\alpha\leq \gamma^1,\dots,\gamma^r\).

In order to prove \transd, we note that the hypothesis implies that there are \(G_j\subseteq_f \In_{\beta^j}\) not all empty such that 
\(\alpha\leq \beta^1_{G_1},\dots,\beta^m_{G_m}\). We have also for \(j\in\lrbm\) and for all
\(h\in\In_{\beta^j}\), \hbox{\(\beta^j_h<\gamma^1,\dots,\gamma^r\)}.
A fortiori, this is true for the \(h\)'s \(\in G_j\). We use \transt\ with these \(\beta^j_h\)'s instead of the~\(\beta^j\)'s.
This gives the desired conclusion \(\alpha< \gamma^1,\dots,\gamma^r\).  

In order to prove \transt, we note that the hypothesis implies
(by weakening) that there are \(F_k\subseteq_f \In_{\gamma^k}\) not all empty such that 
\(\beta^j\leq \gamma^1_{F_1},\dots,\gamma^r_{F_r}\) for \(j\in\lrbm\).
This time we use \transu\ with the \(\gamma^k_\ell\)'s instead of the~\(\gamma^k\)'s and we deduce that \(\alpha\leq \gamma^1_{F_1},\dots,\gamma^r_{F_r}\), which implies \(\alpha< \gamma^1,\dots,\gamma^r\).
\end{proof}

The following lemma shows that when descending to the quotient, 
  \cref{axltle} will be valid in \(\Ord\).

\hum{Étonnant que ce ne soit pas arrivé bien plus vite !}
\begin{lemma}[\axquattre] \label{lemltle}
  Let \(\alpha,\beta^1,\dots,\beta^m\in\ord\). If \(\alpha<\beta^1,\dots,\beta^m\), then \(\alpha\leq \beta^1,\dots,\beta^m\).
\end{lemma}
\begin{proof} Proof by induction on \(\alpha\).
We have \(\alpha < \beta^1,\dots,\beta^m\) if and only if we can find \(F_k\subseteq_f \In_{\beta^k}\) not all empty such that, for each \(i\in \In_\alpha\), we have  \(\alpha_i\leq  \beta^1_{F_1},\dots, \beta^m_{F_m}\). 
Let us fix an \(i\in\In_\alpha\). For \(j\in F_k\), we have  \(\beta^k_j<\beta^k\), and by weakening \(\beta^k_j<\beta^1,\dots,\beta^m\). By \transt,  we get \(\alpha_i<\beta^1,\dots,\beta^m\). Finally, since this is true for all \hbox{\(i\in\In_\alpha\)}, we have  \(\alpha\leq \beta^1,\dots,\beta^m\).
\end{proof}

The following fact shows that \cref{supsucfini} will be valid when we shall descend to the quotient \(\Ord\).  
\begin{lemma}[\axsupsucfini] \label{lemaxsupsucfini}
  If \(\alpha<\gamma\) and \(\beta<\gamma\), then \(\sup(\alpha,\beta)<\gamma\).
\end{lemma}
\begin{proof}
By definition, we have  \(\suc(\alpha,\beta)\leq \gamma\). \Cref{lemsupsucc}
gives \(\sup(\alpha,\beta)<\suc(\alpha,\beta)\). By transitivity, we get \(\sup(\alpha,\beta)<\gamma\).
\end{proof}
%

\begin{lemma} \label{lemContradic}
 Let \(n\) be a positive integer and \(\alpha^1,\dots,\alpha^n\in\ord\).
It  is impossible that, for each \(i\in\lrbn\), we have  \(\alpha^i < \alpha^1,\dots,\alpha^n\).
\end{lemma}
\begin{proof} By induction. Using weakening, the hypothesis to be proven impossible gives finite lists 
\[
  \text{\(F_1\subseteq_f\In_{\alpha_1}\), \dots, \(F_n\subseteq_f\In_{\alpha_n}\),}
\]  
not all empty, such that
\[
  \alpha^i \leq    \alpha^1_{F_1},\dots,\alpha^n_{F_n} \hbox{  for } i \in\lrbm.
\] 
In particular, for \(j\in F_i\) (if \(F_i\) is nonempty), we have 
\[
  \alpha^i_j <    \alpha^1_{F_1},\dots,\alpha^n_{F_n}.
\] 
This reduces to the hypothesis with the nonempty list \(\alpha^1_{F_1},\dots,\alpha^n_{F_n}\) instead of the list \(\alpha^1,\dots,\alpha^n\).
\end{proof}
%

\begin{lemma} \label{lema<avirguleb} Let \(\alpha^1,\dots,\alpha^n,\beta^1,\dots,\beta^m\in\ord\) \((n,m\geq 1)\).
\begin{enumerate}
\item  If   \(\alpha^i < \alpha^1,\dots,\alpha^n,\beta^1,\dots,\beta^m\) for \(i\in\lrbn\),
then  \(\alpha^i < \beta^1,\dots,\beta^m\) for each \(i\).
\item Let \(F_1\subseteq_f\In_{\alpha_1}\), \dots, \(F_n\subseteq_f\In_{\alpha_n}\). If \(\alpha^i \leq  \alpha^1_{F_1},\dots,\alpha^n_{F_n},\beta^1,\dots,\beta^m\) for \(i\in\lrbn\),
then  \(\alpha^i \leq  \beta^1,\dots,\beta^m\) for each \(i\).
\end{enumerate}
\end{lemma}
\begin{proof} 
\emph{1}. The hypothesis yields finite lists
\[
  \text{\(F_1\subseteq_f\In_{\alpha^1}\), \dots, \(F_n\subseteq_f\In_{\alpha^n}\), \(G_1\subseteq_f\In_{\beta^1}\), \dots, \(G_m\subseteq_f\In_{\beta^m}\),}
\]
not all empty, such that
\[
  \alpha^i \leq    \alpha^1_{F_1},\dots,\alpha^n_{F_n}, \beta^1_{G_1},\dots,\beta^m_{G_m} \hbox{ for } i \in\lrbn.
\eqno(*)
\] 
Thus we have for  \(i\in\lrbn\) and \(j\in\In_{\alpha^i}\) 
\[ 
\begin{array}{ccc} 
\alpha^i_j  <  \alpha^1_{F_1},\dots,\alpha^i_{F_i},\dots,\alpha^n_{F_n}, \beta^1_{G_1},\dots,\beta^m_{G_m}\text.
\end{array}
\]
Let us fix \(i\)~and~\(j\): a fortiori, with \(F'_i=F_i\cup\so j\)
\[ 
\begin{array}{ccc} 
\alpha^i_j  <  \alpha^1_{F_1},\dots,\alpha^i_{F'_i},\dots,\alpha^n_{F_n}, \beta^1_{G_1},\dots,\beta^m_{G_m}.   
\end{array}
\]
We have also by weakening, for \(k\in\lrbn\) and \(\ell\in F_k\)
\[ 
\begin{array}{ccc} 
\alpha^k_\ell  <  \alpha^1_{F_1},\dots,\alpha^i_{F'_i},\dots,\alpha^n_{F_n}, \beta^1_{G_1},\dots,\beta^m_{G_m}.   
\end{array}
\]
Thus by induction
\(\alpha^i_j < \beta^1_{G_1},\dots,\beta^m_{G_m}\).
Since \(j\) is arbitrary, we get \(\alpha^i \leq  \beta^1_{G_1},\allowbreak\dots,\beta^m_{G_m}\).
This gives the desired conclusion, \(\alpha^i<\beta^1,\dots,\beta^m\), if at least one list~\(G_k\) is nonempty, for an arbitrary~\(i\in\lrbn\).
If this is not the case, \((*)\)  yields  \(\alpha^i \leq    \alpha^1_{F_1},\dots,\alpha^n_{F_n}\)   for \(i \in\lrbn\), with  lists \(F_i\) not all empty. By definition, this implies \(\alpha^i <    \alpha^1,\dots,\alpha^n\)   for \(i \in\lrbn\), which is impossible by \cref{lemContradic}.

\noindent \emph{2}.
We have  for  \(i\in\lrbn\) and \(j\in\In_{\alpha^i}\) 
\[ 
\begin{array}{ccc} 
\alpha^i_j  <  \alpha^1_{F_1},\dots,\alpha^i_{F_i},\dots,\alpha^n_{F_n}, \beta^1,\dots,\beta^m\text.
\end{array}
\]
Let us fix \(i\)~and~\(j\): a fortiori, with \(F'_i=F_i\cup\so j\),
\[ 
\begin{array}{ccc} 
\alpha^i_j  <  \alpha^1_{F_1},\dots,\alpha^i_{F'_i},\dots,\alpha^n_{F_n}, \beta^1,\dots,\beta^m.   
\end{array}
\]
We have also by weakening, for \(k\in\lrbn\) and \(\ell\in F_k\)
\[ 
\begin{array}{ccc} 
\alpha^k_\ell  <  \alpha^1_{F_1},\dots,\alpha^i_{F'_i},\dots,\alpha^n_{F_n}, \beta^1,\dots,\beta^m.   
\end{array}
\]
Item \emph{1} then yields
 \(\alpha^i_j  <   \beta^1,\dots,\beta^m\).   
As  \(j\) is arbitrary, we get what we want: \hbox{\(\alpha^i \leq  \beta^1,\dots,\beta^m\)} for an arbitrary~\(i\in\lrbn\). 
\end{proof}

The following fact shows that \cref{ax10,ax11} will be valid when we shall descend to the quotient~\(\Ord\).

\begin{lemma} \label{lem-supleftright}
\begin{enumerate}[label=\textit{\arabic*.},ref=\textit{\arabic*}]
%
\item \axdouze.  If \(\alpha<\sup(\alpha,\beta)\), then \(\alpha<\beta\);
\item \axtreize. If \(\gamma<\alpha\) and \(\alpha\leq\sup(\beta,\gamma)\), then \(\alpha\leq \beta\).
\end{enumerate} 
\end{lemma}
\begin{proof} \emph{1}. Assume \(\alpha<\sup(\alpha,\beta)\).
\Cref{lemsupetlist} gives \(\alpha<\alpha,\beta\). Item \emph{1} of \cref{lema<avirguleb} gives \(\alpha<\beta\).
 
\noindent \emph{2}. Assume \(\gamma<\alpha\) and \(\alpha\leq\sup(\beta,\gamma)\). The first hypothesis gives \(\gamma\leq \alpha_F\) for a nonempty \(F\subseteq_f\In_\alpha\). 
The second hypothesis gives   \(\alpha\leq \gamma,\beta\) (by \cref{lemsupetlist}). 
By transitivity we have   \(\alpha\leq  \alpha_F,\beta\).
 Item~\emph{2} of \cref{lema<avirguleb} gives \(\alpha\leq \beta\).
\end{proof}
%

\begin{theorem} \label{thOrd1}
We have constructed \(\Ord\) as an \ept{\fF}.
\end{theorem}
\begin{proof} Using \rfl\ and \transu, we first show that the equality is indeed an equivalence relation, and then that the relation \(\leq\) descends to the quotient in \(\Ord\).

Similarly, \transd\ and \transt\ imply that  the relation \(<\) descends to the quotient in \(\Ord\).

The \(\sup\) map descends to the quotient by \cref{corfactOr1}, Item~\ref{corfactOr1-supz}.


The unary \(\suc\) map descends to the quotient by \cref{factOr01}. 



It remains to note that \cref{refantisym,zero,ax3,axltle,transun,transdeux,transtrois,axsuccunaire,axsuccunaire2,supsucfini,ax10,ax11,axsup,ax14,ax15} of \epts{\fF} have been proved above. See, respectively: \cref{corfactOr1} (Item~\ref{corfactOr1-rfl}); \cref{factOr00} (Item~\ref{factOr001}); \cref{corfactOr1} (Item~\ref{corfactOr1-irfl}); \cref{lemltle}; \cref{lemtrans}; \cref{factOr01}; \cref{lemaxsupsucfini}; \cref{lem-supleftright}; \cref{factsup}; \cref{corfactOr1} (Item~\ref{corfactOr1-ax14}); \cref{rem-ax15}.
\end{proof}

The following theorem generalises \cref{factnatord}.
\begin{theorem} \label{thOrd2}
The set \(\Ord\) is not reduced to a point. More precisely:
\begin{itemize}
\item for all \(\alpha,\beta\in\Ord\),  \(\beta\leq \alpha\) and \(\alpha<\beta\) are incompatible;
\item  the map \(n\mapsto \und n\colon\NN\to\Ord\) 
is injective (\(m<n\) if and only if \(\und m <\und n\));
\item for all \(\alpha\in\Ord\) and  \(n>m\) in \(\NN\),  it is impossible that \(\suc^{(n)}(\alpha)=_\Ord\suc^{(m)}(\alpha)\).
\end{itemize}
\end{theorem}
\begin{proof} 
The first item is a consequence of \irfl\ and of \transd. The rest follows.
\end{proof}
%

\begin{theorem} \label{thOrd3}
\(\Ord\) is an initial object in the category of \epts{\fF}.
\end{theorem}
\begin{proof}[Sketch of proof]
 The structure is purely algebraic and in order to construct \(\Ord\), we have only used the axioms of the structure.

\noindent  In fact, let us consider an object \((E,<_E,\leq_E,0_E,\sup_E,\suc_E)\) in the category. Elements of \(\ord\) do have their copies in \(E\). Furthermore, the relations \(\cdot<\cdots\) and \(\cdot\leq \cdots\) defined in \(\ord\) are valid in \(E\) by \cref{factltle1} if interpreted in \(E\) with finite \(\sup\)'s on the right-hand side (as we may by \cref{lemsupetlist}).
This implies that there is a unique morphism from \(\Ord\) to \(E\)
in the category.
\end{proof}
%

\subsection{More properties}
%

\begin{proposition} \label{propordbienfonde}
The binary relation \(<\) on \(\Ord\) is well-founded.
\end{proposition}
\begin{proof}
This is a direct consequence of \cref{factsousord}.
\end{proof}
%

\begin{lemma}[weak forms of the disjunction ``\(\alpha\leq\beta\) or \(\beta<\alpha\)'']
  \label{lemcut} 
Let \(r\geq 1\) and \(\alpha,\beta^1,\dots,\beta^r,\gamma\in\ord\). 
\begin{enumerate}
\item If \(\alpha\leq \beta\) and \(\beta<\alpha,\gamma^1,\dots,\gamma^r\), then  \(\beta<\gamma^1,\dots,\gamma^r\).
\item If \(\beta< \alpha\) and \(\alpha\leq \beta,\gamma^1,\dots,\gamma^r\), then  \(\alpha\leq \gamma^1,\dots,\gamma^r\). 
\end{enumerate}
\end{lemma}
\begin{proof}
Introduce \(\gamma=\sup(\gamma^1,\dots,\gamma^r)\). Using \cref{lemsupetlist}, both items reduce to already established properties.
\end{proof}
%

\begin{definition} \label{defiordfiltrant}
An element \(\beta=\sucr\beta\in\ord\) is said to be \emph{filtering} if for each  \(F\subseteq_f\In_\beta\) there exists \(j\in \In_\beta\) such that \(\sup(\beta_i)_{i\in F}\leq\beta_j\).
\end{definition}

\begin{lemma} \label{filtrant}
For each \(\alpha\in\ord\), there exists \(\beta\in\ord\) such that
\(\alpha=_\Ord  \beta\) and \(\beta\) is filtering.
\end{lemma}
\begin{proof}
If \(\alpha=\s{\alpha_i}{i\in J}\), we let \(K\) be the set of finitely enumerated subsets of \(J\), and for \(F\subseteq_f J\) we let \(\beta_F=\sup(\alpha_j)_{j\in F}\). Finally we let \(\beta=\s{\beta_F}{F\in K}\).
\end{proof}

\hum{ 
\begin{remark} \label{thOrdComplet} 
(Complétude) ??? 
 On aimerait bien un résultat du genre suivant.
\emph{Pour \(\beta\in\Ord\) on note \(\dar \beta=\sotq{\alpha\in\Ord}{\alpha\leq \beta}\). 
Soit \(\beta\in\Ord\) and \((\gamma^{(\alpha)})_{\alpha\in\dar\beta}\) 
une famille dans \(\Ord\) indexée par~\(\dar\beta\): alors, cette famille admet une borne supérieure  in \(\Ord\).}
Cela serait  peut-être pratique pour la suite, mais semble hors d'atteinte: il faudrait parcourir tous les noms d'ordinaux \(\leq \beta\), mais ceci ne semble pas former un ensemble. En tout cas, même si c'est un ensemble, il semble que l'on n'a aucun accès raisonnable à la relation d'ordre.
\end{remark}
}

\subsection{Elementary ordinal arithmetic}\label{subsec-arit-ord}

\subsubsection*{(Sequential) addition}
The sequential addition  \(\alpha+\beta\) (\(\alpha\) followed by \(\beta\): addition is not  commutative) is defined by induction on~\(\beta\): 
\[
  \alpha+\und0=\alpha\quad \hbox{and}\quad 
\alpha+\beta=\s{\alpha+\beta_j}{j\in\In_\beta} \;\hbox{ if }\;\beta=\s{\beta_j}{j\in\In_\beta} \in\ord\sta.
\]  

The formula for \(\alpha+\beta\) works only in the case \(\In_\beta\neq \NN_0\) (it would yield \(\alpha+\und0=\und0\)).
%
We also have \(\alpha+\beta=\sup((\alpha+\beta_j)+\und1)_{j\in\In_\beta}\) if \(\In_\beta\neq \NN_0\).

The following properties can be proved by induction:
\begin{itemize}
\item if \(\alpha\leq \alpha'\) and \(\beta\leq \beta'\), then 
\(\alpha+\beta\leq \alpha'+\beta'\);
\item \((\alpha+\beta)+\gamma=\alpha+(\beta+\gamma)\);
\item \(\alpha+\und0=\und0+\alpha=\alpha\); 
\item  \(\alpha+\beta\leq \alpha+\gamma\) if and only if \(\beta\leq\gamma\); 
\item  \(\alpha+\beta< \alpha+\gamma\) if and only if \(\beta< \gamma\);
\item  \(\alpha=\und1+\alpha\) if and only if \(\omega\leq  \alpha\);
\item if \(\alpha\leq \gamma\), then there is \(\beta\) such that \(\gamma=\alpha+\beta\);
\item if \(\alpha< \gamma\), then there is \(\beta\ne\und0\) such that \(\gamma=\alpha+\beta\).
\end{itemize}

\subsubsection*{Sequential sum}
Let \(J\in\fF\) and consider a well-founded linear order relation~\(\prec\) on \(J\)
with a detachable minimal element \(0_J\). 
Let \((\beta^j)_{j\in J}\) be an element of \(\Fam(J,\ord)\). The \(\prec\)-indexed sequential sum \(\sum_{j\prec \ell}\beta^j\) is defined by induction on \(\ell\in{(J,\prec)}\): 
\[
  \som_{j\prec 0_J}\beta^j=0_J\quad \hbox{and}\quad 
\som_{j\prec \ell}\beta^j=
\sup \left(\big(\som_{j\prec k}\beta^j\big)+
\beta^{k}\right)_{k\prec\ell} \hbox{ if }0_J\prec\ell.
\] 
We show by induction on \(\prec\) that, given two families \((\beta^j)_{j\in J}\) 
and \((\gamma^j)_{j\in J}\) such that \(\beta^j\leq\gamma^j \) for all \(j\in J\), we have \(\som_{j\prec \ell}\beta^j\leq \som_{j\prec \ell}\gamma^j\) for all \(\ell\in J\). This construction descends therefore to the quotient~\(\Ord\).

\begin{remark} \label{remsomindexée}  
This construction allows us to define a map \(\ord_2^\mathrm{Br}\to\ord_2\), where \(\ord_2^\mathrm{Br}\) is the set of names of Brouwer ordinals. See \cite
{troelstra69} and \cite
{brouwer18,brouwer26}. Troelstra only treats countable Brouwer ordinals.
\end{remark}

\subsubsection*{Multiplication}

We define  \(\alpha\cdot\beta\) by induction on \(\beta\in\ord\): 
\[
  \alpha\cdot\und0=\und0\quad \hbox{and}\quad 
\alpha\cdot\beta=\sup (\alpha\cdot\beta_j +\alpha)_{j\in\In_\beta} \;\hbox{ if }\;\beta=\s{\beta_j}{j\in\In_\beta} \in\ord\sta.
\]  

The following properties can be proved by induction:
\begin{itemize}
\item  if \(\alpha\leq \alpha'\) and \(\beta\leq \beta'\), then
\(\alpha\cdot\beta\leq \alpha'\cdot\beta'\);
\item \((\alpha\cdot\beta)\cdot\gamma=\alpha\cdot(\beta\cdot\gamma)\);
\item \(\alpha\cdot\und1=\und1\cdot\alpha=\alpha\); 
\item  
\(\alpha\cdot(\beta+\gamma)=(\alpha\cdot\beta)+(\alpha\cdot\gamma)\);
\item if \(\und1\leq \alpha\), then \(\alpha\cdot\beta\leq \alpha\cdot\gamma\) if and only if \(\beta\leq\gamma\); 
\item
if \(\und1\leq \alpha\), then  \(\alpha\cdot\beta< \alpha\cdot\gamma\) if and only if \(\beta< \gamma\).
%
\end{itemize}

\subsubsection*{Exponentiation}
We define  \(\alpha^\beta\) by induction on \(\beta\in\ord\), as follows: 
\[
  \alpha^{\und0}=\und1\quad \hbox{and}\quad 
\alpha^\beta=\sup ( \alpha^{\beta_j} \cdot \alpha)_{j\in\In_\beta} \;\hbox{ if }\;\beta=\s{\beta_j}{j\in\In_\beta} \in\ord\sta.
\]  

\subsubsection*{Ackermann}

It is possible to continue this elementary arithmetic à la Ackermann as in \citealt{finsler51}. We define by induction an ordinal \(\Acko(\alpha,\beta,\gamma)\) that we get by iterating \(\gamma\) times the  preceding  map, initialised at \(\alpha\), i.e.\ more precisely

\[ 
\begin{array}{rcl} 
\Acko(\alpha,\beta,\und0)  & =  & \alpha+\beta  \\[.3em] 
\Acko(\alpha,\und0,\gamma)  & =  &  \alpha \qquad \hbox{ if }\;\gamma\in\ord\sta\\[.3em] 
\Acko(\alpha,\beta,\gamma)  & =  & \sup \bigl(\sup
  (\Acko(\Acko(\alpha,\beta_j,\gamma),\alpha,\gamma_k))_{j\in\In_\beta}\bigr)_{k\in\In_\gamma} \\[.3em]&& \qquad\quad \hbox{ if }\;\beta=\s{\beta_j}{j\in\In_\beta}\;\hbox{ and }\;\gamma=\s{\gamma_k}{k\in\In_\gamma}.
 \end{array}
\]

In particular, \(\varepsilon_0=\Acko({\omega,\omega,\und4})\).




\section{Countable ordinals}

\subsection{First steps}
As previously indicated, we get countable ordinals when we choose as set of \indexors  
\[\fF_2=\sotq{\NN_k}{k\in\NN, k\geq0} \cup \so\NN\]
with convenient operations for the set of finite subsets of an \(I\in \fF_2\) and for disjoint unions.
We write \(\ord_2\) and \(\Ord_2\) for \(\ord_{\fF_2}\) and \(\Ord_{\fF_2}\).
Thus \(\Ord_2\) is the set of ordinals of the second class
and \(\ord_2\) is a set of names for elements of \(\Ord_2\).

\begin{lemma} \label{lemsucccroissant}
Any countable ordinal is the \(\suc\) of a nondecreasing sequence of countable ordinals.
\end{lemma}
\begin{proof}
This is \cref{filtrant}.
\end{proof}
%

\begin{proposition} \label{propOrdLPO1} 
Assume \LPO\@. 
Then, for \(\alpha,\beta\in\Ord\), we have  \(\alpha\leq \beta\) or \(\beta<\alpha\).
\end{proposition}
\begin{proof} Proceed as for \cref{propOrdclass}, in the countable case.
\end{proof}
%

%
%

\subsection{Comparison with Martin-Löf ordinals}\label{subsecOML}

We present a variation of the theory of ordinals in the book
\emph{Notes on Constructive Mathematics} \citep[Chapter 3]{PML}. We write ``variation'' since Martin-Löf's
theory is formulated in the framework of Markov's recursive mathematics,
while we take as primitive intuitionistic logic with generalised inductive definitions,
as does the work \citealt{heyting61} (the fact that this setting can provide
a more elegant treatment than the one in recursive mathematics is stressed in \citeauthor{kreisel63}'s
review (\citeyear{kreisel63}) of this work).

\subsubsection{Martin-Löf's formal system}

In this system, ordinals are described inductively: if we have a finite or infinite sequence
of ordinals \(\sigma = \sigma_0,\dots,\sigma_n,\dots\) (maybe empty), then \(\suc(\sigma)\) is an ordinal.

The (classical) semantics of this operation is the following: to \((\sigma_n)\) sequence of ordinals
we associate the supremum of the sequence of the successors of the~\(\sigma_n\)'s.

In particular, \(\und 0\) is defined as \(\suc(\sigma)\), where \(\sigma\) is the empty sequence.

 We write simply \(\suc(\alpha)\) for \(\suc(\sigma)\), where \(\sigma\) is the sequence with one element \(\sigma_0 = \alpha\).

In constructive mathematics, the set of all such ordinals is an example of a nondiscrete set.

As stated in the introduction, to any ordinal \(\alpha\) we associate, by induction on~\(\alpha\), a tree \(\Tree(\alpha)\)
:
\(\Tree(\alpha)\) always contains the empty sequence, and \(\Tree(\suc(\sigma))\) contains \(n\cct \ell\) if
\(\ell\) is in \(\Tree(\sigma_n)\).

This set \(\Tree(\alpha)\) does not contain any infinite branch: if \(f\) is a numerical function,
we can always find \(n\) such that \([f(0),\dots,f(n-1)]\) is not in \(\Tree(\alpha)\). This is proved
directly by induction on \(\alpha\). In other words, the tree \(\Tree(\alpha)\) is well-founded.
The fact that we get in this way all well-founded trees is the content of Brouwer's
bar theorem, which holds neither in Bishop's set theory nor in dependent type theory.
This follows from the fact that both
systems have an interpretation in recursive mathematics, while the bar theorem does
not hold in recursive mathematics, as shown by an example due to Kleene \citep[see][]{kleenevesley65}. 

We define next what an atomic formula is: a formula of the form \(\alpha<\beta\) or \(\alpha\leqslant \beta\).

Finally, we can define when a sequent \(\Gamma\) is provable, where \(\Gamma\) is a finite \emph{set}
of atomic formulae. The formulation is quite elegant!
\[
  \frac{\Gamma,\alpha\leqslant \sigma_n}{\Gamma,\alpha<\suc(\sigma)} ~~~~~~~~
\frac{\cdots\ \Gamma,\sigma_n<\beta\ \cdots}{\Gamma,\suc(\sigma)\leqslant \beta}
\] 
Note that there is a direct proof of  \(\und0\leqslant \beta\) by the second rule.

The intuitive meaning of a sequent is the classical disjunction of the atomic formulae it contains.

Martin-Löf then defines an equivalence relation $\alpha=_\mathrm{ML} \beta$ on $\ord_2$ as expressing the fact that the sequents $\alpha\leq \beta$ and $\beta\leq \alpha$ are valid. The set of Martin-Löf ordinals, denoted by~$\Ord_2^\mathrm{ML}$, is the quotient of $\ord_2$ by this equivalence relation.

Martin-Löf proves for instance the sequent \(\alpha<\beta,\beta\leqslant \alpha\) by induction on~\(\beta\) and~\(\alpha\).
He also shows that the following rule is admissible by induction on \(\alpha\):
\[
  \frac{\Gamma,\alpha<\alpha}{\Gamma}\text,
\] 
which implies in particular that \(\alpha<\alpha\) is not provable.

Let us give an example of such proofs by induction.

\begin{lemma}
 The sequents \(\alpha\leqslant \alpha\) and  \(\alpha<\suc(\alpha)\) are provable for all \(\alpha\).
\end{lemma}

\begin{proof}
  We prove \(\alpha\leqslant \alpha\) by induction on \(\alpha\). If \(\alpha = \suc(\sigma)\), we have to show
  \(\sigma_n<\suc(\sigma)\) for all \(n\), which follows from \(\sigma_n\leqslant \sigma_n\), which holds by induction.

  It follows that we have \(\alpha<\suc(\alpha)\) by the first derivation rule.
\end{proof}

Martin-Löf also proves the analogue of \cref{thOrd2} for $\Ord_2^\mathrm{ML}$. But the two statements, for $\Ord_2^\mathrm{ML}$ and for our $\Ord_2$, are independent of each other.

\subsubsection{Comparison with our system}
\label{sec:comparison}

Let us explain now why this definition does not coincide with ours by giving an example
of the form \(\alpha<\beta\) which is provable in this sequent calculus but implies \LPO\ in our system.

Let us return to \cref{exa123}: define \(\alpha = \suc( \sigma)\),
where \(\sigma_n=\und{u_n}\) with \((u_n)\) a nondecreasing sequence
of~\(0\)'s and~\(1\)'s, and \(\beta = \suc( \tau)\), where \(\tau_n=s(\sigma_n)=\und{u_n+1}\).

\begin{lemma}
The sequent \(\alpha<\beta\) is provable.
\end{lemma}

\begin{proof}
By the first rule, it is enough to show
\(\alpha<\suc( \tau),\alpha\leqslant\tau_0\). And for this we have to show
\(\sigma_n<\tau_0, \alpha<\suc(\tau)\) for all \(n\).
We fix \(n\) and we show
\( \sigma_n< \tau_0, \alpha<\suc( \tau)\).

If we do have \( \sigma_n<\tau_0=\suc( \sigma_0)\), this is fine. Note
that we can \emph{test} whether or not \(\sigma_n<\tau_0\) holds since both
\(\sigma_n\) and \(\tau_0\) are of the form~\(\und0\) or~\(\und1\) or~\(\und2\).

Otherwise, we get explicitly \(n\) such that
\( \sigma_n\geqslant\suc(\sigma_0)\) and we have then
\(\sigma_m\leqslant \sigma_n\) and so \(\sigma_m< \tau_n\) for all~\(m\).
We can prove \( \sigma_n< \tau_0,\alpha<\suc( \tau)\) by
\(\sigma_n< \tau_0, \alpha\leqslant\tau_n\)
which holds since
\(\sigma_n<\tau_0,  \sigma_m< \tau_n\)
holds for all~\(m\).
\end{proof}

  Note that we prove \(\alpha<\suc( \tau)\) by proving \(\alpha<\suc( \tau),\alpha\leqslant  \tau_0\), and
  we have to ``keep'' \(\alpha<\suc( \tau)\): maybe \(\alpha\leqslant  \tau_0\) does not hold
  (it may happen that the sequence~\((\sigma_n)\) takes the value \(\und1\) and \(\tau_0 = \und1\)).


  In \cref{exa123}, we note that \(\alpha<\beta\) implies \LPO\ in our system.
Therefore, in the set \(\Ord_2^\mathrm{ML}\) of Martin-Löf ordinals, the equality is coarser than in the set~\(\Ord_2\), though both are based on the  set  \(\ord_2\).

\addcontentsline{toc}{section}{References}

\small

\bibliographystyle{plainnat}
\markboth{References}{References}


\normalsize
\endgroup
\stopcontents[english]


\clearpage
\newpage
\thispagestyle{empty}


\clearpage
\newpage ~
\newpage

\setcounter{page}{0}
\setcounter{figure}{0}
\setcounter{page}{1} 
\setcounter{section}{0}
\setcounter{subsection}{0}
\setcounter{equation}{0}

\renewcommand\thepage{F\arabic{page}}\renewcommand\theHsection{F\arabic{section}}\renewcommand\theHfigure{F\arabic{figure}}


\rdb
\label{beginfrench}
\selectlanguage{french}

\renewcommand{\wkn}{{\tt affaiblissement}}%

\renewcommand{\crefrangeconjunction}{ à\nobreakspace}%
\renewcommand\crefrangepreconjunction{}%
\renewcommand\crefrangepostconjunction{}%
\renewcommand{\crefpairconjunction}{ et\nobreakspace}%
\renewcommand{\crefmiddleconjunction}{, }%
\renewcommand{\creflastconjunction}{ et\nobreakspace}%
\renewcommand{\crefpairgroupconjunction}{ et\nobreakspace}%
\renewcommand{\crefmiddlegroupconjunction}{, }%
\renewcommand{\creflastgroupconjunction}{ et\nobreakspace}%

\crefname{equation}{}{}
\creflabelformat{inequality}{(#2#1#3)}
\crefname{section}{section}{sections}

\theoremstyle{plain}
\newtheorem{ftheorem}{Théorème}[section]
\newtheorem{flemma}[ftheorem]{Lemme}
\newtheorem{fcorollary}[ftheorem]{Corolaire}
\newtheorem{fproposition}[ftheorem]{Proposition}
\newtheorem{fpropdef}[ftheorem]{Proposition et définition}
\newtheorem{ffact}[ftheorem]{Fait}
\newtheorem*{fpropriétésofindexors}{Propriétés pour l'ensemble des \Fords}
\newtheorem*{faxiomsfororders}{Axiomes pour les \Fords}
\theoremstyle{definition}
\newtheorem{fdefinition}[ftheorem]{Définition}
\newtheorem*{fconventions}{Conventions}
\theoremstyle{remark}
\newtheorem{fremark}[ftheorem]{Remarque}
\newtheorem*{fcomments}{Commentaire}
\newtheorem{fexample}[ftheorem]{Exemple}

\def\frenchproofname{\textsl{Démonstration}}


\newcounter{MF}
\newcommand\stMF{\stepcounter{MF}}

\newcommand{\lec}{\stMF\ifodd\value{MF}lecteur\xspace \else 
lectrice\xspace \fi}

\newcommand{\lecs}{\stMF\ifodd\value{MF}lecteurs\xspace \else 
lectrices\xspace \fi}

\newcommand{\alec}{\stMF\ifodd\value{MF}au lecteur\xspace \else%
à la lectrice\xspace \fi}

\newcommand{\dlec}{\stMF\ifodd\value{MF}du lecteur\xspace \else%
de la lectrice\xspace \fi}

\newcommand{\llec}{\stMF\ifodd\value{MF}le lecteur\xspace \else la lectrice\xspace \fi}

\newcommand{\Llec}{\stMF\ifodd\value{MF}Le lecteur\xspace \else La lectrice\xspace \fi}

\newcommand{\lui}{\ifodd\value{MF}lui\xspace \else
elle\xspace \fi}

\newcommand{\celui}{\ifodd\value{MF}celui\xspace \else
celle\xspace \fi}

\newcommand{\ceux}{\ifodd\value{MF}ceux\xspace \else
celles\xspace \fi}

\newcommand{\er}{\ifodd\value{MF}er\xspace \else
ère\xspace \fi}

\newcommand{\eux}{\ifodd\value{MF}eux\xspace \else
elles\xspace \fi}

\newcommand{\eUx}{\ifodd\value{MF}eux\xspace \else
euse\xspace \fi}

\newcommand{\leux}{\ifodd\value{MF}leux\xspace \else
leuse\xspace \fi}

\newcommand{\il}{\ifodd\value{MF}il\xspace \else
elle\xspace \fi}

\newcommand{\e}{\ifodd\value{MF} \else e\xspace \fi}

\newcommand{\n}{\ifodd\value{MF}n\xspace \else nne\xspace \fi}

\makeatletter
\newcommand{\la}{\@ifstar{\ifodd\value{MF}le\xspace\else
la\xspace\fi}{\stMF\ifodd\value{MF}le\xspace \else la\xspace \fi}}
\makeatother

\renewcommand\indexor{indexeur\xspace}
\renewcommand\indexors{indexeurs\xspace}
\renewcommand\Indexors{Indexeurss\xspace}

\newcommand \inxr {\indexor}
\newcommand \inxrs {\indexors}
\newcommand \Inxr {\indexor}

\newcommand \Ford {\(\fF\)-ordre\xspace} 
\newcommand \Fords {\(\fF\)-ordres\xspace} 

\newcommand \cara {caracté\-ristique\xspace} 
\newcommand \caras {caracté\-ristiques\xspace} 

\newcommand \dem {démonstration\xspace} 
\newcommand \dems {démonstrations\xspace} 

\newcommand \dfn {définition\xspace} 
\newcommand \dfns {définitions\xspace} 

\newcommand \egt {égalité\xspace} 
\newcommand \egts {égalités\xspace} 
\newcommand \egmt {également\xspace} 

\newcommand \elr {élémentaire\xspace} 
\newcommand \elrs {élémentaires\xspace} 

\newcommand \elt {élément\xspace} 
\newcommand \elts {éléments\xspace}

\newcommand \eqvc {équivalence\xspace}
\newcommand \eqvcs {équivalences\xspace}

\newcommand \imd {immédiat\xspace} 
\newcommand \imde {immédiate\xspace} 
\newcommand \imds {immédiats\xspace} 
\newcommand \imdes {immédiates\xspace} 

\newcommand \prt {propriété\xspace} 
\newcommand \prts {propriétés\xspace} 

\newcommand \ncr {nécessaire\xspace} 
\newcommand \ncrs {nécessaires\xspace} 
\newcommand \ncrt {nécessairement\xspace}

\newcommand \ari {arithmétique\xspace} 
\newcommand \aris {arithmétiques\xspace} 

\newcommand \cof {constructif\xspace} 
\newcommand \cofs {constructifs\xspace} 
\newcommand \cost {constructivement\xspace} 
\newcommand \cov {constructive\xspace} 
\newcommand \covs {constructives\xspace} 
 
\newcommand \coma {mathé\-ma\-tiques \covs} 
\newcommand \clama {mathé\-ma\-tiques classiques\xspace} 

\newcommand \recu {récurrence\xspace} 

\newcommand \ssi {si, et seulement si, } 

\newcommand \ML{\mathrm{ML}}

\newcommand \bleu[1]{\textcolor{blue}{#1}}

\renewcommand \ept[1]{\(#1\)-ordre\xspace}
\renewcommand \epts[1]{\(#1\)-ordres\xspace}

\FrenchFootnotes

\title{Une théorie constructive des ordinaux}

\newcommand\oge{\leavevmode\raise.3ex\hbox{$\scriptscriptstyle\langle\!\langle\,$}}
\newcommand\feg{\leavevmode\raise.3ex\hbox{$\scriptscriptstyle\,\rangle\!\rangle$}}
\renewcommand\gui[1]{\oge{#1}\feg}

\thickmuskip = 7mu plus 2mu

\pagestyle{headings}
\patchcmd{\sectionmark}{\MakeUppercase}{}{}{}

\stMF
\startcontents[french]

\author{Thierry Coquand, Henri Lombardi, Stefan Neuwirth}
\maketitle

\rdb
\label{fbeginfrench}

\begin{abstract} 
\citet{fPML} décrit des ordinaux construits de manière récursive.
Il donne une version constructivement acceptable des ordinaux calculables de Kleene.
En fait, la définition de Turing des fonctions calculables n’est pas nécessaire d’un point de vue constructif.
Nous donnons dans cet article une théorie constructive des ordinaux similaire à la théorie de Martin-Löf, mais basée uniquement sur les deux relations \gui{\(x\leq y\)} et \gui{\(x< y\)}, c'est-à-dire sans considérer les séquents dont le sens intuitif est une disjonction classique. Dans notre cadre, l’opération \gui{supremum d'une famille d'ordinaux} joue un rôle important à travers ses interactions avec les relations \gui{\(x\leq y\)} et \gui{\(x< y\)}. Cela permet d'approcher autant que possible la notion d'ordre total lorsque la propriété \gui{\(\alpha\leq \beta\) ou \(\beta\leq \alpha\)} n'est prouvable qu'en logique classique.
Notre objectif est de donner une définition formelle correspondant à l'intuition et de démontrer que nos ordinaux constructifs satisfont de manière constructive toutes les propriétés souhaitables.
Notons qu'en ajoutant la logique classique, on retrouverait les ordinaux des \clama usuelles, au prix d'une perte de la calculabilité pour la plupart des énoncés donnés sous la forme usuelle.
\end{abstract}

\noindent {\bf Mots clés:} 
nombre ordinal; mathématiques constructives. 

\smallskip \noindent {\bf MSC2020:} 03E10 03F65.

\small

\setcounter{tocdepth}{4}
\markboth{Table des matières}{Table des matières}

\printcontents[french]{}{1}{}
\normalsize

\newpage

\section{Introduction}

Ce papier est écrit dans le cadre des \coma informelles.
Nous utilisons la théorie des ensembles de Bishop enrichie par les définitions inductives généralisées (Bishop a utilisé ces sortes de constructions pour la théorie de la mesure, les ensembles boréliens et la théorie de l'intégrale de Lebesgue).

En mathématiques classiques, une définition naturelle d'un ordinal est d'être le type d'ordre d'un ensemble bien ordonné (voir par exemple  \citealp[III.2.Ex.14]{fbourbaki68}).
Néanmoins, il est plus pratique d'utiliser les ordinaux de von Neumann, pour lesquels de nombreux résultats peuvent être prouvés sans recourir au choix (voir par exemple \citealp[Chapitre~2]{fKri} et \citealp[Chapitre~II]{fdehornoy17}).

Proposons maintenant une approche constructive.
Une relation binaire \(<\) sur un ensemble \(X\) est dite \emph{bien fondée} si pour toute famille d'ensembles \((E_x)_{x\in X}\) indexée par \(X\) il est possible de construire des éléments de \(\prod_{x\in X}E_x\) par \(<\)-induction. Précisément, chaque fois que l'on donne une construction 
\(\gamma\) qui à partir d'un élément \(a\in X\) et d'un élément 
\(\varphi\in\prod_{x\in X,x<a}E_x\) construit un élément \(\gamma(a,\varphi)\in E_a\), il existe un unique \(\Phi\in \prod_{x\in X}E_x\) tel que pour tout \(a\in X \) nous avons \(\Phi(a)=\gamma(a,\Phi|_{x\in X,x<a})\).
Cette notion a une signification constructive claire.

En particulier, considérons une \prt pour les éléments de \(X\). Si la \prt est \(<\)-héréditaire, c'est-à-dire  si elle est vraie pour \(a\in X\) dès qu'elle est vraie pour tout \(x\in X\) tel que \(x<a\), alors cette \prt est vraie pour tous les éléments de~\(X\).

En mathématiques constructives, le livre \citet*[Section~I.6]{fMRR} définit la notion de relation bien fondée d'une manière différente mais équivalente et un ordinal comme un ensemble totalement ordonné pour lequel la relation d'ordre est bien fondée. Ainsi, tous les sous-ensembles de \(\NN\) sont des ordinaux même si nous ne savons pas s'ils ont un plus petit élément.

L'\citet[Section 10.3]{fhottbook}  considère les \gui{ordinaux de Grayson}  (voir \citealp*[Exercise~I.6.12]{fMRR}) dans le cadre de la théorie homotopique des types univalente; les ordinaux d'un univers donné s'avèrent former un ensemble (et non un groupoïde). Cette théorie des ordinaux diffère de la nôtre en ce qui concerne les \cref{faxsuccunaire,faxsuccunaire2} pour les \Fords dans ce qui suit.

Entre autres points de vue constructifs, il existe des descriptions d'ordinaux dénombrables construits par induction dans les ouvrages \citealt{fbrouwer26}, \citealt{fgentzen36}, \citealt{fchurch38}, \citealt{fkleene38}, \citealt{fheyting61} et \citealt[Chapitre 3]{fPML}.

Un traitement constructif des ordinaux de von Neumann basé sur la récursion transfinie est donné par \citet[Section 9.4]{faczelrathjen10}.

Brouwer propose une construction inductive basée sur l'idée que lorsque les ordinaux \(\alpha_n\) sont définis pour tout \(n\in\NN\) et sont des ensembles bien fondés totalement ordonnés, alors on peut décrire l'ordinal \(\alpha \) qui correspond intuitivement à \(\alpha_1\) suivi de \(\alpha_2\) suivi de \(\alpha_3\) suivi de\dots L'ensemble ordonné \(\alpha\) défini par Brouwer sera à nouveau un ensemble bien fondé totalement ordonné. Et si la relation d'ordre sur chaque \(\alpha_i\) est décidable, il en va de même pour \(\alpha\).

Deux ordinaux de Brouwer ne sont en général pas comparables (en logique intuitionniste): il n'existe pas de critère général permettant de décider si deux ordinaux ont le même type d'ordre, et, lorsque ce n'est pas le cas, lequel est isomorphe à un segment initial de l'autre.

L'article \citealt*{fkrausnordvallxu21} compare trois approches constructives distinctes des ordinaux constructifs, notées \textsf{Cnf}, \textsf{Brw} et \textsf{Ord}, qui sont disponibles dans le cadre de la théorie des types univalente. L'approche \textsf{Brw} est directement inspirée des ordinaux de Brouwer.

\smallskip 
Martin-Löf décrit  les ordinaux \gui{récursivement construits}. Il donne ainsi une version \cost acceptable des ordinaux dénombrables récursifs tels que définis en \clama par Kleene.
Intuitivement, un ordinal construit à la Martin-Löf est défini par induction au moyen des deux constructions suivantes:
\begin{itemize}
\item on a un ordinal minimum $\und 0$;
\item si $(\alpha_n)_{n\in\NN}$ est une suite explicite, finie ou infinie, d'ordinaux, la borne supérieure des $\alpha_n+1$ est un ordinal\footnote{D'une manière  qui nous semble peu intuitive, Martin-Löf note cette borne supérieure 
$\sup_{n\in\NN}(\alpha_n)$. Cela permet de tenir compte du cas de la suite vide d'ordinaux, qui a pour borne supérieure l'ordinal $\und 0$. Mais hormis ce cas, il s'agit bien de la borne supérieure des successeurs des $\alpha_n$. Nous préfèrerons donc la notation $\suc_{n\in\NN}(\alpha_n)$.}. 
\end{itemize}
Dire que la définition est inductive, c'est dire que tout ordinal est construit en utilisant les règles indiquées.

Dans un cadre constructif, nous pouvons abandonner les machines de Turing et remplacer la calculabilité de Turing par la calculabilité intuitive (non définie).
Dans ce cas, la principale différence entre les ordinaux de Brouwer et ceux de Martin-Löf est que les ordinaux de Martin-Löf, étant définis de manière \gui{parallèle} plutôt que de manière \gui{séquentielle}, sont plus généraux: pour toute suite d'ordinaux bien définis \((\alpha_n)\) on peut construire le supremum des successeurs des ordinaux \(\alpha_n\). Un inconvénient est qu’il n’existe aucun moyen d’associer à un ordinal de Martin-Löf un ensemble bien fondé totalement ordonné avec le même type d’ordre.
Par exemple, si les \(\alpha_n\) sont tous égaux à \(\und 0\) ou \(\und 1\), il est à priori impossible de décider si le supremum des successeurs des \(\alpha_n \) est égal à \(\und 1\) ou \(\und 2\).

\subsection*{Les ordinaux vus comme des arbres}

Martin-Löf propose de visualiser un ordinal~\(\alpha\) comme un arbre bien fondé avec des branchements finis ou dénombrables. L'ordinal~\(\alpha\) est donné avec un ensemble d'indexeurs noté~\(\In_\alpha\); dans la suite, ce sera un élément de l'ensemble \(\fF_2\) constitué de~\(\NN\) et de ses sous-ensembles finis~\(\NN_k\).
\begin{itemize}
\item L'arbre avec seulement sa racine représente \(\und 0\).
\item Si \((t_i)_{i\in \In_{\alpha}}\) est une famille d'arbres ordinaux
pour une famille d'ordinaux \((\alpha_i)_{i\in \In_{\alpha}}\), le supremum \(\alpha=\s{\alpha_i}{i\in \In_{\alpha}} \) des successeurs des \(\alpha_i\) est donné par l'arbre ordinal pour lequel il y a \(\#\In_\alpha\) branches  au-dessus de la racine et une copie de \(t_i\) est jointe à la branche indexée par \(i\in\In_\alpha\).
\end{itemize}

Considérez les arbres de la figure~\ref{ffigure}.

Si \(n\in\NN\), l'ordinal \(\und n\) peut être représenté par l'arbre avec 
\(n\)~branchements unaires successifs aux \(n\)~nœuds, de sorte qu'il a
\(n+1\)~nœuds.

Le premier ordinal infini \(\omega\) peut être représenté par l'arbre qui a un embranchement dénombrable au-dessus de la racine, les branches portant les arbres précédents (représentant \(\und n\), \(n\in\NN\)).

Son successeur, noté \(\omega+\und1\), peut être représenté par l'arbre à ramification unaire au-dessus de la racine, la branche portant l'arbre précédent.

L'ordinal \(\omega+\und2\) peut être représenté par l'arbre à ramification unaire au-dessus de la racine, la branche portant l'arbre précédent.

L'ordinal \(\omega+\omega\) peut être représenté par l'arbre qui a une ramification dénombrable au-dessus de la racine, les branches portant les arbres représentant \(\omega+\und n\), \(n\in\NN\).%
\begin{figure}[!h]
  \centering
  \includegraphics[width=\textwidth,trim=4 4 6 5
  ]{trees1200}
  \caption{Arbres ordinaux.}
  \label{ffigure}
\end{figure}

Plus formellement, un tel arbre peut être défini comme l'ensemble de ses nœuds, ou points de branchement, convenablement nommés. On peut considérer l'ensemble \(\Lst(\NN)\) de listes finies d'éléments de \(\NN\).
Soit \(n\in\NN\) et \(\ell\), \(\ell'\in\Lst(\NN)\). On note \(n\cct \ell\) la liste \([n,\ell_1,\dots,\ell_k]\), où \(\ell=[\ell_1,\dots,\ell_k]\),  \(\ell\cct n\) la liste \([\ell_1,\dots,\ell_k,n]\), et \(\ell\cct \ell'\) la concaténation des listes \(\ell\) et \(\ell'\).

On remarque que \(\Lst(\NN)\) peut être énuméré de manière naturelle\footnote{Par exemple, pour \(\ell=[\ell_1,\dots,\ell_k]\in\Lst(\NN) \), on pose \(\mu(\ell)=\sum_{i=1}^{k}(\ell_i+1)\) et on énumère les listes par \(\mu(\ell)\) croissants.} et que la notion de famille \(\NN\)-indexée dans \(\Lst(\NN)\) correspond, via une telle énumération, à la notion de base (non définie) de fonction de~\(\NN \) dans~\(\NN\).

Un arbre bien fondé avec des branchements finis ou dénombrables peut alors être décrit comme un sous-ensemble détachable~\(T\) de \(\Lst(\NN)\) qui est construit inductivement selon le processus indiqué précédemment.
\(T\) est clos par segments initiaux: si \(\ell\in\Lst(\NN)\), \(p\in\NN\) et \(\ell\cct p\in T\), alors \(\ell\in T\).
Ainsi, à chaque ordinal \(\alpha\), nous associons un arbre, défini comme un sous-ensemble approprié de \(\Lst(\NN)\), noté \(\Tree(\alpha)\).

Si \(n\in\NN\), l'ordinal \(\und n\) peut être décrit par la suite finie de \(n+1\) listes \([\,]\), \([0] \), \([0,0]\), \(\dots\)\kern.16667em, \([0,\dots ,0]\).

Le premier ordinal infini \(\omega\) peut être décrit par le sous-ensemble de \(\Lst(\NN)\) énuméré par la suite infinie \([\,]\), \([0]\), \( [1]\), \([1,0]\), \([2]\), \([2,0]\), \([2,0,0]\), \([ 3]\), \([3,0]\), \([3,0,0]\), \([3,0,0,0]\), etc.

L'ordinal \(\omega+\und1\) peut être décrit par la suite infinie \([\,]\), \([0]\), \([0,0]\), \([0,1 ]\), \([0,1,0]\), \([0,2]\), \([0,2,0]\), \([0,2,0,0]\), \([0,3]\), \([0,3,0]\), \([0,3,0,0]\), \([0,3,0,0,0 ]\), etc.

L'ordinal \(\omega+\und2\) peut être décrit par la suite infinie \([\,]\), \([0]\), \([0,0]\), \([0,0 ,0]\), \([0,0,1]\), \([0,0,1,0]\), \([0,0,2]\), \([0,0 ,2,0]\), \([0,0,2,0,0]\), \([0,0,3]\), \([0,0,3,0]\), \([0,0,3,0,0]\), \([0,0,3,0,0,0]\), etc.

\begin{sloppypar}
L'ordinal \(\omega+\omega\) peut être décrit par la suite doublement infinie \([\,]\), \([0]\), \([0,0]\), \([0, 1]\), \([0,1,0]\), \([0,2]\), \([0,2,0]\), \([0,2,0,0] \), \([0,3]\), \([0,3,0]\), \([0,3,0,0]\), \([0,3,0,0, 0]\), etc., \([1]\), \([1,0]\), \([1,0,0]\), \([1,0,1]\), \([1,0,1,0]\), \([1,0,2]\), \([1,0,2,0]\), \([1,0,2,0 ,0]\), \([1,0,3]\), \([1,0,3,0]\), \([1,0,3,0,0]\), \( [1,0,3,0,0,0]\), etc., \([2]\), \([2,0]\), \([2,0,0]\), \( [2,0,0,0]\), \([2,0,0,1]\), \([2,0,0,1,0]\), \([2,0, 0,2]\), \([2,0,0,2,0]\), \([2,0,0,2,0,0]\), \([2,0,0, 3]\), \([2,0,0,3,0]\), \([2,0,0,3,0,0]\), \([2,0,0,3, 0,0,0]\), etc., etc.
\end{sloppypar}

Ces arbres, vus comme des sous-ensembles définis par induction de $\Lst(\NN)$, forment un ensemble bien défini dans le contexte des \coma intuitives. On peut le noter $\ord_2$ (voir la \cref{fdeford}). 
Notons que d'un point de vue constructif, $\ord_2$ est un ensemble discret \ssi le principe de Markov est valide.
Cet ensemble $\ord_2$  est un \gui{ensemble des noms d'ordinaux} chez Martin-Löf. Et l'ensemble des ordinaux de Martin-Löf, \(\Ord_2^\ML\),  est un quotient de $\ord_2$ par une relation d'\eqvc correctement démontrée. L'ensemble \(\Ord_2^\mathrm{ML}\) et notre ensemble \(\Ord_2\) sont discrets \ssi le petit principe d'omniscience \LPO\ est valide.  Voir la section \ref{fsubsecOML} pour plus de précisions.

\medskip \centerline{*\ *\ *}\smallskip

Nous donnons dans cet article une théorie constructive des ordinaux similaire à la théorie de Martin-Löf, mais basée uniquement sur les deux relations \gui{\(x\leq y\)} et \gui{\(x< y\)}, c'est-à-dire sans considérer les séquents dont le sens intuitif est une disjonction classique. Dans notre cadre, l’opération \gui{supremum d'une famille d'ordinaux} joue un rôle important à travers ses interactions avec les relations \gui{\(x\leq y\)} et \gui{\(x< y\)}. Cela permet d'approcher autant que possible la notion d'ordre total lorsque la propriété \gui{\(\alpha\leq \beta\) ou \(\beta\leq \alpha\)} n'est prouvable qu'en logique classique.
Notre objectif est de donner une définition formelle correspondant à l'intuition et de démontrer que nos ordinaux constructifs satisfont de manière constructive toutes les propriétés souhaitables.
Notons qu'en ajoutant la logique classique, on retrouverait les ordinaux des \clama usuelles, au prix d'une perte de la calculabilité pour la plupart des énoncés donnés sous la forme usuelle.

\smallskip\centerline{*\ *\ *}\smallskip

La première étape dans la \cref{fsecEpto} est de décrire ces propriétés souhaitables.  

\hum{À COMPLÉTER}

\section{Ensembles presque totalement ordonnés associés à un ensemble d'\inxrs}\label{fsecEpto}
Nous définissons dans cette section la structure d'ordre (presque) total
associée à un ensemble~\(\fF\) d'\inxrs: la structure de \Ford en abrégé.

\subsection{Ensemble d'\inxrs} 
Nous avons besoin pour cela d'un ensemble \(\fF\) d'\inxrs. Un \inxr sera noté \(I\), \(J\), \(K\),  \(I'\), \(I''\), \(J'\),  \(I_a\), \(I_b\), etc. 

Un \emph{\inxr} est simplement un ensemble qui peut servir d'ensemble d'indices pour les familles que l'on va considérer. Dans la suite, un sous-ensemble finiment énuméré de \(A\) est toujours un sous-ensemble de \(A\) défini à la Bishop par une application \(\NN_k\to A\). Si \(A\) est discret, un sous-ensemble finiment énuméré de \(A\)
est un sous-ensemble détachable.

\begin{fpropriétésofindexors}
 On supposera que 
\begin{itemize} 
\item \(\NN\) et les ensembles finis \(\NN_k=\sotq{n\in\NN}{n<k}\) (\(k\geq 0\))   sont des éléments de \(\fF\);  
\item 
toute partie finiment énumérée d'un \elt de $\fF$ est isomorphe\footnote{Dans la catégorie des ensembles.} à un \elt de $\fF$;
\item si \(J\in\fF\), l'ensemble des parties finiment énumérées de $J$ est isomorphe à un \elt de~$\fF$;  
\item $\fF$ est stable par réunions disjointes indexées dans $\fF$:
on notera $I+J$ une réunion disjointe de $I$ et $J$, et $\sum_{i\in I} J_i$ une réunion disjointe de la famille $(J_i)_{i\in I}$. 
\end{itemize}
\end{fpropriétésofindexors}

Nous considérons les réunions disjointes au sens des sommes directes dans la catégorie des ensembles.
Une réunion disjointe \(J=\sum_{i\in I} J_i\) est donnée avec une famille 
 d'applications injectives  \(\iota_\ell\colon J_\ell \to J\) qui réalisent $J$ comme la somme directe des $J_i$ dans la catégorie des ensembles.

Pour les ordinaux de la seconde classe (les ordinaux dénombrables), on peut prendre pour $\fF$ l'ensemble  
\[\fbox{\(\fF_2=\sotq{\NN_k}{k\in\NN, k\geq0} \cup \so\NN\)}\]
muni d'opérations convenables pour l'ensemble des sous-ensembles finis
d'un \(I\in \fF\) et pour les réunions disjointes d'éléments de~\(\fF\) indexés par un élément de~\(\fF\).
Tout autre ensemble d'indexeurs~$\fF$ contiendra au moins l'ensemble $\fF_2$ qui sert à définir les ordinaux dénombrables.

Si $E$ est un ensemble, une \emph{famille indexée dans $\fF$ d'\elts de $E$} est simplement une famille $(x_i)_{i\in I}$ pour un $I\in\fF$, avec les $x_i\in E$. L'ensemble des familles  indexées dans $\fF$ d'\elts de~$E$ sera noté $\Fam(\fF,E)$. 

Nous utilisons des indices \gui{en bas} uniquement comme ci-dessus pour les ordinaux. Nous utiliserons des indices \gui{en haut} pour tous les autres cas.

\subsection{Les axiomes} 
Une  structure de \emph{\Ford} sur un ensemble~\((E,=)\) est  \((E,<,\leq,0_E,\sup,\suc)\), où 
\begin{itemize}
\item $<$ et $\leq$ sont des relations binaires définies sur
l'ensemble $(E,=)$;
\item  $0_E$ est un \elt de $E$; on note $E\sta=\sotq{\alpha\in E}{0_E<\alpha}$;
\item  $\sup$ est une fonction de $\Fam(\fF,E\sta)$ dans $E\sta$: à partir d'un \elt $(\alpha_i)_{i\in I}$ de $\Fam(\fF,E\sta)$, elle construit un \elt de $E\sta$ noté $\alpha=\sup_{i\in I}\alpha_i$;
\item   \(\suc\) est une fonction de  \(E\) vers \(E\sta\):  à partir d'un  élément \(\beta\in E\), elle construit un élément de \(E\sta\) noté  \(\suc(\beta)\), appelé le successeur de $\beta$\footnote{Nous disons dans ce cas que \(\suc\) est la fonction unaire \gui{successeur}. Mais plus loin nous utilisons le même symbole \(\suc\) pour une fonction infinitaire (définition et proposition \ref{ffactsupsucinfinis}).}.
\end{itemize}

\begin{fdefinition} \label{fdefibinarysup}
Pour écrire les axiomes avec des  \(\sup\) finis,  nous définissons  \(\sup(\alpha,\beta)\) pour \(\alpha,\beta\in E\) comme suit (en utilisant implicitement l'\cref{fax15}):  tout d'abord  \(\sup(0_E,\alpha)=\alpha=\sup(\alpha,0_E)\); et si \(\alpha,\beta\in E\sta\), 
\(\sup(\alpha,\beta)\) est déjà défini.
\end{fdefinition}

Ces données doivent vérifier les \prts suivantes. 

%
\begin{faxiomsfororders}
\begin{enumerate}[label=\textit{\arabic*.},ref=\textit{\arabic*}]
\item[]
\item \label[faxiom]{frefantisym} \(\alpha=\beta\) si, et seulement si, \(\alpha\leq \beta\) et \(\beta\leq \alpha\)  (réflexivité et antisymétrie);
\item \label[faxiom]{fzero}\(0_E\leq \alpha\);
\item \label[faxiom]{fax3} si \(\alpha<\alpha\) alors \(0_E=\beta\)  (irréflexivité);  
\item \label[faxiom]{faxltle} si \(\alpha<\beta\) alors \(\alpha\leq \beta\);  
\item\label[faxiom]{ftransun} si \(\alpha\leq \beta\) et \(\beta\leq \gamma\), alors \(\alpha\leq \gamma\)  (transitivité 1); 
\item\label[faxiom]{ftransdeux} si \(\alpha<\beta\) et \(\beta\leq \gamma\), alors \(\alpha<\gamma\)  (transitivité 2);
\item\label[faxiom]{ftranstrois} si \(\alpha\leq \beta\) et \(\beta< \gamma\), alors \(\alpha< \gamma\)  (transitivité 3);
\item \label[faxiom]{faxsuccunaire} \(\alpha<\suc(\beta)\) si, et seulement si, 
\(\alpha\leq \beta\)  (en utilisant l'\cref{frefantisym} on obtient \(\alpha<\suc(\alpha)\));
\item \label[faxiom]{faxsuccunaire2} \(\suc(\beta)\leq \alpha\) si, et seulement si, \(\beta<\alpha\);
\item \label[faxiom]{fsupsucfini} si \(\alpha<\gamma\) et \(\beta<\gamma\), alors \(\sup(\alpha,\beta)<\gamma\);
%
\item \label[faxiom]{fax10} si \(\alpha<\sup(\alpha,\beta)\) alors \(\alpha<\beta\);
\item  \label[faxiom]{fax11}  si \(\gamma<\alpha\) et \(\alpha\leq\sup(\beta,\gamma)\), alors \(\alpha\leq \beta\);
\item \label[faxiom]{faxsup} pour  \((\alpha_i)_{i\in I}\in \Fam(\fF,E\sta)\) et   \(\beta\in E\),  on~a
\[
  \alpha_i\leq \beta\text{ pour tout }i\in I\ \text{ si, et seulement si, }\ \sup(\alpha_i)_{i\in I}\leq \beta
\]
(propriété caractéristique de \(\sup\));
\item \label[faxiom]{fax14} si \(\gamma<\beta\) pour tout \(\gamma<\alpha\), alors \(\alpha\leq \beta\);
\item  \label[faxiom]{fax15} \(\alpha\leq 0_E\) ou   \(0_E<\alpha\).
\end{enumerate}
\end{faxiomsfororders}

La \emph{catégorie des \Fords}  est définie par ses morphismes 
\[
  \ndsp
(E,<_E,\leq_E,0_E,\sup_E,\suc_E) \lora (F,<_F,\leq_F,0_F,\sup_F,\suc_F)\text,
\]  
qui sont les fonctions de $E$ dans $F$ qui préservent la structure (au sens usuel évident).

\begin{fcomments}
  \begin{enumerate}[label=\arabic*)]
  \item Soit $\gamma\in E$ et $(\alpha_n)_{n\in \NN}$ avec  $\alpha_n=\gamma$ ou $\alpha_n=\suc(\gamma)$ pour tout $n$. Alors l'\elt $\sup_{n\in \NN}\alpha_n$ \gui{hésite} à priori entre $\suc(\gamma)$ et $\suc(\suc(\gamma))$.
 Il n'y a donc aucun espoir que la disjonction \gui{$\alpha\leq \beta$ \hbox{ou $\beta< \alpha$}}
puisse être explicite dans le cas d'\elts  
$\alpha,\beta\ne0_E$. En conséquence, on a introduit la fonction $\sup$  avec les axiomes correspondants que l'on peut réaliser de manière \cov, de façon à mieux décrire en quoi l'ordre peut être considéré comme \gui{total}.
Mais ce n'est peut-être pas optimal (il peut manquer des axiomes raisonnables, qui sont satisfaits dans l'ensemble $\Ord_2$ des ordinaux dénombrables construit dans la \cref{fConstOrd} et qui ne résultent pas de ceux donnés ici).


\item L'irréflexivité est donnée sous une forme qui, au lieu d'affirmer une négation, permet~à~\(E\) de se réduire à un singleton. Cela se produit si, et seulement si, \(0_E=\suc(0_E)\), ce qui implique \(0_E<0_E\) en utilisant l'\cref{faxsuccunaire}. 

\item L'\cref{fax15} exprime que \(\so{0_E}\) est détachable. 
Cela contraste avec le fait qu'aucun élément autre que~\(0_E\) ne définit  un singleton détachable. On a défini \(\sup\) sur \(E\sta\) plutôt que sur  \(E\) en vue de satisfaire constructivement la disjonction de l'\cref{fax15}. 

\item La propriété caractéristique de \(\sup\) montre que cette loi  est idempotente et satisfait les \prts d'associativité et commutativité généralisées.
\eoe
\end{enumerate}
\end{fcomments}

\subsection{Quelques propriétés}

\begin{fpropdef}[généralisation de la \cref{fdefibinarysup}] \label{fpropdefsupfini}\leavevmode\\*  
Pour \(\alpha^1,\dots,\alpha^r\in E\) nous posons
\[
  \sup(\alpha^1,\dots,\alpha^r)\eqdefi
\formule{0_E \hbox{ si } \alpha^1=\dots=\alpha^r=0_E\\[.5em]
\hbox{le \(\sup\) des \(\alpha^k\ne0_E\) sinon.}}
\] 
La propriété caractéristique de \(\sup\) est satisfaite:
\[
  \alpha^1\leq \beta\hbox{ et }\dots\hbox{ et }\alpha^r\leq \beta\ \text{ si, et seulement si, }\ \sup(\alpha^1,\dots,\alpha^r)\leq \beta\text.
\] 
\end{fpropdef}

\begin{ffact} \label{ffactsuciso} Soient \(\alpha,\beta\) des éléments de \(E\).
\begin{itemize}
\item \(\suc(\alpha)<\suc(\beta)\) si, et seulement si, \(\alpha<\beta\).
\item \(\suc(\alpha)\leq \suc(\beta)\) si, et seulement si, \(\alpha\leq \beta\).
\end{itemize}
\end{ffact}
\begin{proof}
Utiliser les \cref{faxsuccunaire,faxsuccunaire2}.  
\end{proof}
%

\begin{ffact} \label{ffactreverse1013}
Les \cref{fsupsucfini,fax10,fax11,fax14} sont en fait des équivalences:
\begin{itemize}
\item [\ref{fsupsucfini}.] on a \(\alpha<\gamma\) et \(\beta<\gamma\)  si, et seulement si, \(\sup(\alpha,\beta)<\gamma\); 
\item [\ref{fax10}.] on a \(\alpha<\sup(\alpha,\beta)\) si, et seulement si, \(\alpha<\beta\);
\item [\ref{fax11}.]  l'implication  si \(\gamma<\alpha\), alors \(\alpha\leq\sup(\beta,\gamma)\) a lieu si, et seulement si, \(\alpha\leq \beta\);
\item  [\ref{fax14}.]  on a \(\alpha\leq \beta\) si, et seulement si, \(\gamma<\beta\) pour tout \(\gamma<\alpha\).
\end{itemize}
\end{ffact}
\begin{proof}
Utiliser les transitivités  et la propriété caractéristique de \(\sup\).
\end{proof}
%

\begin{ffact}[la fonction successeur commute avec les \(\sup\) finis, notation comme dans la  {\cref{fpropdefsupfini}}] \label{ffactsupsuc}\leavevmode
On~a 
\(\sup(\suc(\alpha),\suc(\beta))=\suc(\sup(\alpha,\beta))\) et
plus généralement  \(\sup(\suc(\alpha^1),\allowbreak\dots,\suc(\alpha^r))=\suc(\sup(\alpha^1,\dots,\alpha^r))\).
\\
En particulier, si \(\alpha^1<\gamma\), \(\dots\)\kern.16667em, \(\alpha^r<\gamma\), alors \(\sup(\alpha^1,\dots,\alpha^r)<\gamma\).
\end{ffact}
\begin{proof} Il suffit de démontrer \(\suc(\sup(\alpha,\beta))=\sup(\suc(\alpha),\suc(\beta))\). On~a la chaine d'équivalences suivantes: \(\suc(\sup(\alpha,\beta))\leq\gamma \iff\)
\(\sup(\alpha,\beta)<\gamma \iff\)
\((\alpha<\gamma\) et \({\beta<\gamma})\iff\)
 \((\suc(\alpha)\leq\gamma\)  et \(\suc(\beta)\leq\gamma)\iff\) \(\sup(\suc(\alpha),\suc(\beta))\leq \gamma\).
\end{proof}

\begin{fpropdef}[définition d'une fonction \(\suc\) infinitaire,  sa propriété caractéristique] \label{ffactsupsucinfinis} 
 Pour tout  \((\alpha_i)_{i\in J}\in\Fam(\fF,E)\), on définit 
\(\s{\alpha_i}{i\in J}=\sup(\suc(\alpha_i))_{i\in J}\).
On a alors l'équivalence suivante:
\[\alpha_i<\beta\text{ pour tout }i\in J\ \text{ si, et seulement si, }\ \s{\alpha_i}{i\in J}\leq \beta \text. 
\]
\end{fpropdef}
\begin{proof} Utiliser les  \cref{faxsup,faxsuccunaire2}.
\end{proof}

Nous écrivons \fbox{\(F\subseteq_f I\)} pour exprimer le fait que \(F\) est une partie finiment énumérée de~\(I\). 

\begin{ffact} \label{ffactltle1} Soit \(\alpha,\beta^1,\dots,\beta^m\in E\).%
\begin{enumerate}[label=\textit{\arabic*.},ref=\textit{\arabic*}]
\item\label{ffactlte1-1} Supposons que \(\alpha=\s{\alpha_i}{i\in J}\) avec \((\alpha_i)_{i\in J}\in\Fam(\fF,E)\) et que \(\alpha_i<\sup(\beta^1,\allowbreak\dots,\beta^m)\) pour tout \(i\in J\). Alors \(\alpha\leq \sup(\beta^1,\dots,\beta^m)\).
\item\label{ffactlte1-2} Supposons que \(\beta^k=\suc ((\beta^k)_i)_{i\in J_k}\) avec \(((\beta^k)_i)_{i\in J_k}\in\Fam(\fF,E)\) pour \(k\in\lrbm\).
\\
Soit \hbox{\(F_1\subseteq_f J_1, \dots, F_m\subseteq_f J_m\)} non tous vides. Si 
\[
  \alpha\leq \sup((\beta^k)_j)_{k\in\lrbm,\;j\in F_k},
\]  
alors \(\alpha< \sup(\beta^1,\dots,\beta^m)\).
\end{enumerate}
\end{ffact}
\begin{proof}
\ref{ffactlte1-1}. C'est la \cref{ffactsupsucinfinis}.
\\
\ref{ffactlte1-2}. Supposons par exemple  \(F_1\) non vide. Alors
\(\alpha\leq \sup((\beta^1)_j)_{j \in F_1}<\beta^1\leq \sup(\beta^1,\allowbreak\dots,\allowbreak\beta^m)\). L'inégalité stricte résulte du \cref{ffactsupsuc} parce que tous les \((\beta^1)_j\) sont \(<\beta^1\) d'après la \cref{ffactsupsucinfinis}.
\end{proof}
%

\section{Construction inductive d'ensembles d'ordinaux}\label{fConstOrd}

\CAdre{.66} {Dans les sections \ref{fConstOrd} et \ref{fPropfondOrd}, l'ensemble d'indexeurs $\fF$ est fixé mais est rarement mentionné explicitement.}
 
Nous allons définir un ensemble d'ordinaux \(\Ord\) 
(plus précisément \(\Ord_\fF\))
et nous prouverons que c'est un objet initial dans la catégorie des \epts{\fF}.

Nous définissons d’abord un ensemble \(\ord\) (plus précisément \(\ord_\fF\)) de noms pour les ordinaux \(\fF\)-indexés au moyen d'une définition inductive.
La définition inductive la plus simple d'un ensemble infini est celle de \(\NN\): l'ensemble possède un élément \(0\) et une application successeur \(x\mt s(x)\colon \NN\to \NN\).
La définition inductive de \(\ord\) est très semblable à celle de \(\NN\).
Dans \(\NN\), chaque élément est soit \(0\) soit un \(s(x)\) pour un~\(x\in\NN\).
De même, dans \(\ord\),
chaque élément est soit \(\und0\) soit le \(\suc\) d'une famille \(\fF\)-indexée d'\elts de \(\ord\); on note \(\ord\sta\) l'ensemble des éléments de ce deuxième type.

\begin{fdefinition} \label{fdeford}
L'ensemble \(\ord\)  (plus précisément \(\ord_\fF\)) est défini par induction: il admet un \elt distingué \(\und0\) et il a une application 
\[\suc\colon\Fam(\fF,\ord)\to\ord\text.\]
\end{fdefinition}
\noindent N. B.: La seule contrainte dans cette définition inductive est que \(\suc\) est bien une application de~\(\Fam(\fF,\ord)\) vers \(\ord\). 

Un élément de \(\ord\)
sera appelé un \emph{[nom d'un] ordinal} dans la suite.

Quand \(\fF=\fF_2\), nous obtenons l'ensemble des noms d'ordinaux dénombrables, noté \(\ord_2\).
%
\begin{fremark} \label{fremdeford}
Tout  élément \(\alpha\in\ord\sta\) 
est donné avec: 
\begin{itemize}
\item l'\inxr utilisé dans la définition de \(\alpha\): il sera noté~\(\In_\alpha\);
\item la famille \(\chi_\ord(\alpha,i)_{i\in\In_\alpha}\) de ses \emph{sous-ordinaux définitionnels}, \emph{i.~e.}\ l'élément de \(\Fam(\fF,\ord)\) tel que \(\alpha=\s{\chi_\ord(\alpha,i)}{i\in\In_\alpha}\).
\end{itemize}
Ainsi, la  définition inductive de \(\ord\) implique l'existence 
d'une fonction \(\alpha\mt\In_\alpha\colon\ord\sta\to\fF\) et 
l'existence d'une famille dépendante \((\alpha,i)\mt\chi_\ord(\alpha,i)\)  définie pour \(\alpha\in\ord\sta\) et \(i\in\In_\alpha\). 
Pour simplifier l'exposé on commettra dans la suite
 un léger abus de notation en sous-entendant la construction de la famille dépendante $\chi_\ord$ et en notant $\alpha_i$ pour 
 $\chi_\ord(\alpha,i)$. On écrira donc selon cette convention d'écriture \fbox{\,$\alpha=\sucr\alpha\phantom{^n}\!\!$}. \eoe
 
\end{fremark}

Pour  \(\alpha^1,\dots,\alpha^r\in\ord\) on définit  \fbox{\(\suc(\alpha^1,\dots,\alpha^r)=\s{\alpha^i}{i\in\lrbr}\)}.

En particulier, si  \(\alpha\in\ord\),  son \emph{successeur immédiat}
  \hbox{\(\suc(\alpha)\)} est l'élément 
\(\beta=\s{\beta_i}{i\in\In_{\beta}}\), où \(\In_\beta=\NN_1=\so 0\) et \(\beta_0=\alpha\).
La suite \((\und m)_{m\in\NN}\)  dans \(\ord\) est définie par \recu par
  \(\und{m+1}=\suc(\und m)\). 
Puis nous définissons \fbox{\(\omega=\s{\und{n}}{n\in\NN}\)}.

Pour démontrer une propriété de \(\alpha=\s{\alpha_i}{i\in\In_\alpha}\), il suffit de démontrer la propriété pour chaque~\(\alpha_i\). 
De la même manière, on peut construire par induction une fonction dont le domaine de \dfn est \(\ord\), ou définir un prédicat sur~\(\ord\) par induction. C'est le cas par exemple dans les \cref{fdefsousord,fdefsupord} et plus généralement dans toute la suite de l'article.

\subsection{Sous-ordinaux} 
Voici une définition inductive précise.
\begin{fdefinition} \label{fdefsousord} 
Étant donné $\alpha=\sucr\alpha\in\ord\sta$, un \elt $\beta$ de $\ord$ est appelé un \emph{sous-ordinal définitionnel} de $\alpha$ lorsque $\beta=\alpha_i$ pour un $i\in \In_\alpha$. On note alors $\beta\lessdot_1 \alpha$. L'\elt~$\gamma$ est un \emph{sous-ordinal de $\alpha$} s'il est un sous-ordinal définitionnel de $\alpha$ ou un sous-ordinal d'un sous-ordinal définitionnel de $\alpha$.
On note alors $\gamma\lessdot \alpha$.
\end{fdefinition}

Ainsi $\und0$ est le seul \elt de $\ord$ qui n'a pas de sous-ordinal.

Le fait suivant prend acte du fait que la \dfn des relations $\cdot\lessdot_1\cdot$ et $\cdot\lessdot\cdot$ est une \dfn inductive correcte sur $\ord$.

\begin{ffact} \label{ffactsousord}
Les relations $\cdot\lessdot_1\cdot$ et $\cdot\lessdot\cdot$ sont bien fondées sur $\ord$.
\end{ffact}

En conséquence \gui{il n'y a pas de branches infinies dans l'arbre des sous-ordinaux d'un \elt de $\ord$}, au sens suivant.

\begin{ffact} \label{ffactSousord}
Une suite $(\alpha^j)_{j=1,2,\dots}$ dans $\ord$ où chaque $\alpha^{j+1}$ est un sous-ordinal de $\alpha^{j}$ aboutit en un nombre fini d'étapes à $\alpha^r=\und0$.
\end{ffact}

Notons que pour faire une construction (ou une démonstration) par $\lessdot_1$-induction (ou \hbox{par $\lessdot$-induction}), le cas $\und0$ doit être traité à part car il n'a pas de sous-ordinal.
Cependant, jusqu'à l'\ari des ordinaux \cpageref{fsubsec-arit-ord}, nous allons pouvoir nous passer de cette distinction de cas.

\subsection{Définition de la loi \texorpdfstring{\(\sup\)}{sup}}

\begin{fdefinition} \label{fdefsupord}~
\begin{enumerate}
\item 
La loi   $\sup\colon\Fam(\fF,\ord\sta)\to\ord\sta$  est définie comme suit.

Soit \((\alpha^j)_{j\in J}\) une famille dans \(\ord\sta\) avec \(J\in\fF\).
Si \(\alpha^j=\s{(\alpha^j)_i}{i \in I_{j}}\), alors \fbox{\(\sup(\alpha^j)_{j\in J}\)} est l'\elt \(\varepsilon=\s{\varepsilon_k}{k \in K}\), où 
\begin{itemize}
\item \(K\) est la réunion disjointe des \(I_{j}
  \);
\item \((\varepsilon_{k})_{k \in K}\) est la famille définie par  
\(\varepsilon_k=(\alpha^j)_i\) si \(\iota_j(i)=k\) 

\end{itemize}
(ici \(\iota_j\colon I_{j}
\to K\) est l'injection de \(I_{j}
\) dans la réunion disjointe des \(I_{j}
\)). 
On notera \(\sup(\alpha^j)_{j\in\lrbr}=\sup(\alpha^1,\dots,\alpha^r)\). 

\item Le \(\sup\) d'une famille finie dans \(\ord\) est défini comme suit.
\[
  \sup(\alpha^1,\dots,\alpha^r)\eqdefi
\formule{\und0 \;\hbox{ si } \alpha^1=\dots=\alpha^r=\und0
\\[.5em]
\hbox{le \(\sup\) des \(\alpha^k\in\ord\sta\) sinon.}}
\] 
\end{enumerate}
\end{fdefinition}

Notons que le point  2 est formellement inclus dans le point 1 si nous adoptons la convention \(\In_{\und0}=\NN_0\). Par contre cela ne permettrait pas de définir
des sup arbitraires de familles $\fF$-indexées   d'\elts de $\ord$.

\subsection{\texorpdfstring{Définitions de \(\leq \) et \(<\)}{Définitions de <= et  <}} 
Le principal du travail reste à faire, à savoir définir deux prédicats binaires $\leq $ et $<$ sur $\ord$ qui satisfont les \prts attendues. Plus précisément:
\begin{itemize}
\item la relation \gui{$\alpha\leq \beta$ et $\beta\leq \alpha$} doit être une relation d'\eqvc  (on note $\Ord$ l'ensemble quotient),
\item  les prédicats $\leq $ et $<$ et les fonctions $\sup$ et $\suc$ doivent passer au quotient (on ne change pas leurs noms),
\item  et la structure obtenue doit être une structure de \Ford.
\end{itemize}

\smallskip Pour faire ce travail on définit par induction deux relations asymétriques entre d'une part un \elt de $\ord$ et d'autre part une \emph{liste finie non vide (à permutation près}\footnote{On dit aussi un \emph{multiensemble} non vide.}) d'\elts de~$\ord$: 
\CAdre{.6}{\centerline{
$\alpha\leq \beta^1,\dots,\beta^m$\quad \hbox{  et  }\quad $ \alpha<\beta^1,\dots,\beta^m$\quad ($m\geq 1$).}
\vspace{-.8em}
}

\begin{fconventions}
\begin{itemize}
\item Les lettres $\alpha$, $\beta$, $\gamma$, $\varepsilon$
éventuellement munies d'exposants, d'indices ou de primes,  désignent des \elts de $\ord$.
\item Si $\alpha$ est un \elt de $\ord$ et si $F$ est une liste finie, éventuellement vide, dans $\In_\alpha$,
on note $\alpha_F$ la liste des $\alpha_i$ pour les $i$ dans $F$.
\end{itemize}
\end{fconventions}

Les   définitions inductives simultanées des  deux relations  sont les suivantes.
\begin{framed}%
\noindent Les cas particuliers pour $\und0$ sont évités en posant par convention $\In_{\und0}=\NN_0$. Soit \(m\) un entier \(\geq 1\). \smallskip

\noindent\(\alpha\leq \beta^1,\dots,\beta^m\)\quad
est défini par\quad\(\alpha_i<\beta^1,\dots,\beta^m\) pour tout \(i\in \In_\alpha\).

\noindent\(\alpha< \beta^1,\dots,\beta^m\)\quad
est défini par\quad
\begin{tabular}[t]{@{}l@{}}
  il existe \(F_1\subseteq_f \In_{\beta^1},\dots,F_m\subseteq_f \In_{\beta^m}\)\\
  non toutes vides telles que \(\alpha\leq \beta^1_{F_1}, \dots,\beta^m_{F_m}\).
\end{tabular}
\end{framed}
Le fait que ces \dfns sont correctement posées tient à ce que les \elts de $\ord$ sont définis de manière inductive et à ce que le couple des deux \dfns est inductif. 

Sans la convention concernant $\In_{\und0}=\NN_0$ le fait \ref{ffactOr00}
ci-dessous devrait faire partie de la \dfn. Cette convention est \gui{un petit miracle} qui nous permet, dans les \dems qui suivront, de ne pas à avoir à raisonner au cas par cas selon que $\alpha= \und0$ ou $\alpha\in\ord\sta$.
 
La signification de ces deux relations est \(\alpha\leq \sup(\beta^1,\dots,\beta^m)\) et \(\alpha< \sup(\beta^1,\allowbreak\dots,\beta^m)\). 

\begin{flemma} \label{flemsupetlist}
On a \(\alpha<\beta^1,\dots,\beta^m\) si, et seulement si, \(\alpha< \sup(\beta^1,\dots,\beta^m)\). 
\\
De la même manière \(\alpha\leq \beta^1,\dots,\beta^m\) si, et seulement si, \(\alpha\leq \sup(\beta^1,\dots,\beta^m)\).
\end{flemma}
\begin{proof}
Écrivons 
\[
  \begin{aligned}
    \alpha\prec \beta^1,\dots,\beta^m&\text{ pour }\alpha< \sup(\beta^1,\dots,\beta^m)\text,\\
    \alpha\preceq \beta^1,\dots,\beta^m&\text{ pour }\alpha\leq \sup(\beta^1,\dots,\beta^m)\text.
\end{aligned}
\]
Soit \(\varepsilon=\sup(\beta^1,\dots,\beta^m)\). Alors \(\alpha\prec \beta^1,\dots,\beta^m\) si, et seulement si, \(\alpha\leq\varepsilon_F\) avec \(F\) un sous-ensemble finiment énuméré non vide de la réunion disjointe~\(K\) des \(\In_{\beta^j}\) et \(\varepsilon_k=(\beta^j)_i\) si \(k\) est l'image de~\(i\) dans \(K\); en posant \(F_j=F\cap\In_{\beta^j}\), les~\(F_j\) ne sont pas tous vides et cela peut être réécrit comme \(\alpha\leq \beta^1_{F_1}, \dots,\beta^m_{F_m}\). Cela a lieu si, et seulement si, \(\alpha< \beta^1,\dots,\beta^m\).

On a \(\alpha\preceq \beta^1,\dots,\beta^m\) \ssi pour tout \(i\in \In_\alpha\), \(\alpha_i<\varepsilon\), \emph{i.~e.}\ \(\alpha_i\prec\beta^1,\dots,\beta^m\), \emph{i.~e.}\ \(\alpha_i<\beta^1,\dots,\beta^m\); cela a lieu si, et seulement si, \(\alpha\leq \beta^1,\dots,\beta^m\).
\end{proof}

La relation \(\alpha=_\Ord\beta\) est définie comme signifiant \gui{\(\alpha\leq \beta\) et \(\beta\leq \alpha\)}. 

On montrera dans la \cref{fPropfondOrd} que la relation \(\cdot=_\Ord\cdot\) est une relation d'équivalence et on définira l'ensemble \(\Ord\) comme le  quotient de \(\ord\) par cette relation.

Notons que jusqu'au \cref{fthOrd1}, le symbole \(=\) entre deux éléments de \(\ord\) est l'égalité dans \(\ord\) et n'a pas la signification de  \(=_\Ord\). Néanmoins, une fois que l'on aura montré que les
relations et les lois de $\ord$ \gui{passent au quotient} dans $\Ord$, les énoncés contenant le symbole $=$ seront \egmt valables avec le symbole $=_\Ord$.

\subsection{Ordinaux finis, ordinaux bornés}

Nous commençons avec quelques propriétés de~\(\und0\).

\begin{ffact} \label{ffactOr00}  Soit \(m\) un entier \(\geq 1\), \(\alpha,\beta^1,\dots,\beta^m \in\ord\), et \(\gamma\in\ord\sta\). On a
\begin{enumerate}[label=\textit{\arabic*.},ref=\textit{\arabic*}]
\item\label{ffactOr001}  \(\und0\leq \beta^1,\dots,\beta^m\);
\item  \(\und0< \gamma,\beta^2,\dots,\beta^m\);
\item \(\alpha<\underbrace{\und0,\dots,\und0}_{m\text{ fois}}\) est impossible.
\end{enumerate}
\end{ffact}

\begin{proof}
Conséquence \imde des \dfns.
\end{proof}

\begin{fremark} \label{frem-ax15}
L'axiome \emph{\ref{fax15}} 
sera satisfait dans $\Ord$ car tout \elt de $\ord$ est donné ou bien sous la 
forme~$\und0$ ou bien sous la forme d'un \elt~$\gamma\in\ord\sta$, de sorte qu'on a toujours $\und0<\gamma$ d'après le point~\emph{2} du \cref{ffactOr00}.
\eoe
\end{fremark}

\begin{ffact} \label{ffactnatord}
Soit \(m,n\in\NN\). Alors
\begin{enumerate}[label=\textit{\arabic*.},ref=\textit{\arabic*}]
\item \(m\leq n\) si, et seulement si, \(\und m\leq \und n\);
\item \(m< n\) si, et seulement si, \(\und m< \und n\); 
\item \(\und m\leq \und n\) et \(\und n< \und m\) sont incompatibles.
\end{enumerate}
\end{ffact}
\begin{proof}
Pour les implications directes dans \emph{1} et \emph{2}, on écrit $n=m+r$ et l'on fait une \recu sur $r$. 
Pour les implications réciproques, on a déjà vu  les cas $m=0$ et $n=0$. On vérifie ensuite que $\und{m+1}\leq \und{n+1}$ implique 
$\und{m}\leq \und{n}$, et que $\und{m+1}< \und{n+1}$ implique 
$\und{m}< \und{n}$. Cela permet de conclure par \recu sur $m$.
\\
Le point \emph{3} découle des points \emph{1} et \emph{2}. 
\end{proof}

Un $\alpha\in\ord$ est dit \emph{fini} si $\alpha=_\Ord\und m$  pour un $m\in\NN$, il est dit  \emph{borné} s'il est majoré par un ordinal fini. Les ordinaux bornés sont \gui{beaucoup plus complexes} que les ordinaux finis (voir les   \cref{fexaLLPO,fexa123}).

Dans la \cref{fin-class-math}, nous discuterons ce que les relations \(\leq\)~et~\(<\) sur l'ensemble~\(\ord_\fF\) deviennent en \clama.

\subsection{Premières conséquences} 

Le fait suivant montre que lorsqu'on sera passé au quotient, sur l'ensemble $\Ord$, la loi $\suc$ vérifiera la \prt \cara donnée dans l'\cref{faxsup}.
\begin{ffact}[\sucdef] \label{ffactsucc}
 On a $\alpha\leq \beta$ \ssi pour tout~$i\in \In_\alpha$, $\alpha_i<\beta$.
\end{ffact}
\begin{proof}
Cette \prt est tautologique: il s'agit simplement de la \dfn de \hbox{$\alpha\leq \beta$}. 
\end{proof}

De la même manière, le fait suivant montre que la loi \(\sup\)  satisfera la  propriété caractéristique donnée dans l'\cref{faxsup} quand nous montrerons qu'elle passe au quotient dans  \(\Ord\). 

\begin{ffact}[\supdef] \label{ffactsup}
Soit \((\alpha^j)_{j\in J}\) une famille dans \(\ord\sta\) avec \(J\in\fF\), \(\gamma=\sup(\alpha^j)_{j\in J}\), et \hbox{\(\beta\in\ord\)}. 
On a \(\gamma\leq \beta\) si, et seulement si, \(\alpha^j\leq \beta\) pour tout \(j\in J\).
En particulier, \(\sup(\alpha,\beta)\leq \beta \) si, et seulement si, \(\alpha\leq \beta\).
\end{ffact}

\noindent N. B.: Le résultat est \egmt vrai pour le \(\sup\) d'une famille finie dans \(\ord\).

\begin{proof} Encore une tautologie linguistique.
On a $\alpha^j=\suc_{i \in I_{j}}(\alpha^j)_i$ pour un $I_j\in\fF$.
Par \dfn de $\gamma$ et de $\leq$, l'in\egt $\gamma\leq \beta$ signifie que pour tout $j\in J$ et tout $i\in I_j$,  on a
$(\alpha^j)_i<\beta$, \emph{i.~e.}\ que pour tout $j\in J$, on a $\alpha^j\leq \beta$. 
\end{proof}
Le fait suivant montre que lorsqu'on sera passé au quotient, sur l'ensemble $\Ord$, les \cref{faxsuccunaire,faxsuccunaire2}  seront satisfaits.
\begin{ffact} \label{ffactOr01} 
\begin{enumerate}[label=\textit{\arabic*.},ref=\textit{\arabic*}]
\item\label{ffactOr01-1} \axsuccunaire.
On a  \(\alpha<\suc(\beta)\) si, et seulement si, \(\alpha\leq \beta\).
\item\label{ffactOr01-2} \axsuccunair.
On a \(\beta<\alpha\) si, et seulement si, \(\suc(\beta)\leq \alpha\).
\end{enumerate}
\end{ffact}
\begin{proof}
Rappelons que l'élément \(\gamma=\suc(\beta)\) est défini par \(\In_\gamma=\so 0\) et \(\gamma_0=\beta\). 

\noindent\ref{ffactOr01-1}. Par définition, \(\alpha<\gamma\) signifie que \(\alpha\leq \gamma_F\) pour une liste non vide \(F\subseteq_f\so0\). Cela force \(F=[0]\) et \(\gamma_F=\beta\).

\noindent\ref{ffactOr01-2}. Par définition, \(\gamma\leq \alpha\) signifie que \(\gamma_0<\alpha\), \emph{i.~e.}\ \(\beta<\alpha\).

\noindent En bref, mieux que des \eqvcs, ce sont des tautologies.
\end{proof}

La fait suivant nous permettra de raccourcir certaines démonstrations par induction.

\begin{ffact} \label{ffactOr1}~
\begin{enumerate}
%
\item [a.] On a une inégalité \(\; \alpha\leq \beta\) si, et seulement si, pour tout \(i\in \In_\alpha\), 
il existe un  \(F_i\subseteq_f \In_\beta\) non vide tel que \(\;\alpha_i\leq \beta_{F_i}\).
\item [b.] On a une inégalité \(\; \alpha< \beta\) si, et seulement si, il existe 
un \(F\subseteq_f \In_\beta\) non vide tel que  pour tout \(i\in \In_\alpha\) on~a  \(\;\alpha_i< \beta_{F}\).
\end{enumerate}
\end{ffact}
\begin{proof}
Directe d'après les définitions.
\end{proof}

Nous quittons maintenant les \dems tautologiques pour aborder les premières \dems par induction.

\penalty-2500
\begin{ffact}[affaiblissement et contraction] \label{ffactOr0}~
\begin{itemize}
%
\item \wkn. Si \(\alpha\leq \beta^1,\dots,\beta^m\), alors pour tout \(\beta\)
on a  \(\alpha\leq \beta,\beta^1,\dots,\beta^m\).
\item \ctn. Si \(\alpha\leq \beta^1,\beta^1,\beta^2,\dots,\beta^m\) alors \(\alpha\leq \beta^1,\beta^2,\dots,\beta^m\).
\item Mêmes \prts avec $<$ à la place de $\leq $.
\end{itemize}
\end{ffact}
\begin{proof}
Démonstrations par induction, \imdes d'après les \dfns.
\end{proof}

Le lemme suivant est un corolaire du lemme \ref{ffactOr1}.  Le point \emph{1} (resp. \emph{2})  impliquera que la fonction $\suc$ (resp. $\sup$) passe au quotient dans $\Ord$ (resp. $\Ord\sta$).
Le point \emph{3} impliquera que les relations $\leq $ et $=$ sont réflexives dans $\Ord$, les points \emph{5} et \emph{7} impliqueront les axiomes \emph{\ref{fax3}} et \emph{\ref{fax14}} pour $\Ord$.

\begin{flemma} \label{fcorfactOr1}~
\begin{enumerate}[label=\textit{\arabic*.},ref=\textit{\arabic*}]
%
\item\label{fcorfactOr1-succz} \succz. Soit \(\alpha,\beta\in\ord\) avec \(\In_\alpha=\In_\beta\) et \(\alpha_i\leq \beta_i\) pour tout \(i\in \In_\alpha\). Alors \(\alpha\leq \beta\).

\item\label{fcorfactOr1-supz} \supz. Soit   \(\alpha,\beta\in\ord\sta\) avec \(\In_\alpha=\In_\beta\) et \(\alpha_i\leq \beta_i\) pour tout \(i\in \In_\alpha\). Alors
\[\sup(\alpha_i)_{i\in\In_\alpha}\leq \sup(\beta_i)_{i\in\In_\beta}\text.\] 
Le résultat vaut aussi pour les sup finis dans $\ord$.
\item\label{fcorfactOr1-rfl} \rfl. Pour tout  \(\alpha\in\ord\), on~a \(\alpha\leq \alpha \).
À fortiori, \(\alpha\leq \alpha,\beta^1,\dots,\beta^m\).
\item\label{fcorfactOr1-succu} \succu. Pour tout  \(\alpha\in\ord\sta\) et tout \(i\in \In_\alpha\), on~a \(\alpha_i<\alpha\). À fortiori, \(\alpha_i< \alpha,\beta^1,\dots,\beta^m\).
\item\label{fcorfactOr1-irfl} \irfl. Pour tout  \(\alpha\in\ord\), \(\alpha<\alpha\) est impossible.
\item \label{fcorfactOr1-37} \(\alpha < \suc(\alpha)\).
\item\label{fcorfactOr1-ax14} {\tt ax14}.  Si \(\gamma<\beta\) pour tout \(\gamma<\alpha\), alors \(\alpha\leq \beta\).
\end{enumerate}
\end{flemma}
\begin{proof}

\noindent \ref{fcorfactOr1-succz}. Direct d'après le \cref{ffactOr1}\emph{a}.
On prend \(F=\so{i}\).

\noindent
\ref{fcorfactOr1-supz}. Soit   \(\gamma=\sup(\alpha_i)_{i\in\In_\alpha}\)
et  \(\epsilon=\sup(\beta_i)_{i\in\In_\beta}\). D'après le \cref{ffactOr1}\emph{a}, pour tout~\(j\in\In_{\alpha_i}\) il existe un \(F_{i,j}\subseteq_f \In_{\beta_i}\) non vide tel que \((\alpha_i)_j\leq (\beta_i)_{F_{i,j}}\); \(F_{i,j}\) est à fortiori dans la réunion disjointe des~\(\In_{\beta_i}\), de sorte que \((\alpha_i)_j<\epsilon\) 
par définition de~\(\epsilon\). Par définition, \(\alpha_i\leq\epsilon\), de sorte que d'après le \cref{ffactsup} on a \(\gamma\leq\epsilon\).

\noindent
\ref{fcorfactOr1-rfl}. Par induction: nous utilisons le \cref{ffactOr1}\emph{a}, on prend \(F=\so{i}\) et \(\alpha\leq\alpha\) se réduit à \(\alpha_i\leq \alpha_i\).

\noindent
\ref{fcorfactOr1-succu}. Par induction: nous utilisons le \cref{ffactOr1}\emph{b}, on prend \(F=\so{i}\) et \(\alpha_i<\alpha\) se réduit à \((\alpha_i)_j< \alpha_i\).

\noindent
\ref{fcorfactOr1-irfl}. Par induction: nous utilisons \cref{ffactOr1}\emph{b}, on prend \(F=\so{i}\) et \gui{\(\alpha< \alpha\) est impossible} se réduit à: \gui{\(\alpha_i< \alpha_i\) est impossible}.

\noindent
\ref{fcorfactOr1-37}.
Appliquer  \succu\ à \(\beta=\suc(\alpha)\).

\noindent
\ref{fcorfactOr1-ax14}.
Si \(\alpha=\und0\), la conclusion est claire. Si \(\alpha=\s{\alpha_i}{i\in I}\), comme \(\alpha_i<\alpha\) pour chaque \(i\in\In_\alpha\) (point~\ref{fcorfactOr1-succu}), l'hypothèse que \(\gamma<\beta\) pour tout \(\gamma<\alpha\) montre que \(\alpha_i<\beta\) pour tout \(i\in \In_\alpha\).
Nous concluons par le~\cref{ffactsucc} que \(\alpha\leq \beta\).
\end{proof}

\subsection{Ordinaux et principes d'omniscience limités}
\label{fsec:princ-omnisc}

\begin{fexample} \label{fexaLLPO}
Soit   \((v_n)_{n\in \NN}\)  une suite dans \(\so{0,1}\) qui prend au moins une fois la valeur \(1\). Le principe \LLPO\ dit que nous avons 
\[
  \exists k\in \so{0,1}\ \forall n\ \ (v_n=1 \Rightarrow \;n\equiv k\mod 2).\eqno(*)
\] 
Pour une telle suite \((v_n)_{n\in \NN}\) définissons \(\varepsilon\), \(\varepsilon^1\) et \(\varepsilon^2\in\ord\) comme suit:
\[
  \varepsilon=\s{\und{v_n}}{n\in\NN},\quad\varepsilon^1=\s{\und{v_{2m}}}{m\in\NN},\quad\varepsilon^2=\s{\und{v_{2m+1}}}{m\in\NN}.
\] 
Alors on a \(\varepsilon\leq \sup(\varepsilon^1,\varepsilon^2)\). Mais 
\(\varepsilon\leq\varepsilon^1\) donne \(k=0\) dans \((*)\) et \(\varepsilon\leq\varepsilon^2\) donne \(k=1\) dans \((*)\). Donc la disjonction \(\varepsilon\leq\varepsilon^1\) ou \(\varepsilon\leq\varepsilon^2\) n'a pas de démonstration \cov: supposer la disjonction pour une suite \((v_n)\) arbitraire implique \LLPO.
\end{fexample}

\begin{fexample} \label{fexa123}
Soit   \((u_n)_{n\in \NN}\)  une suite croissante (au sens large) dans \(\so{0,1}\). Le principe \LPO\ dit qu'une telle suite est stationnaire: 
\[
  \exists n\in\NN\ \forall m\in\NN\ u_m\leq u_n. \eqno(*)
\]   
Pour une telle suite \((u_n)_{n\in \NN}\) définissons \(\alpha\) et \(\beta\in\ord\) comme suit: 
\[
  \alpha=\s{\und{u_n}}{n\in\NN},\qquad    \beta=\s{\und{u_n+1}}{n\in\NN}.
\]  
Nous notons que l'inégalité stricte \(\alpha<\beta\) est équivalente (d'après les \cref{ffactnatord,ffactOr1} et le \cref{flemsupetlist}) à
\[
  \exists n\in\NN\ \forall m\in\NN\ u_m< u_n+1,
\] 
ce qui revient à \((*)\).
En fait, \(\alpha\) hésite entre \(1\) et \(2\), \(\beta\) hésite entre \(2\) et~\(3\), et l'inégalité \(\alpha<\beta\)
est valide si nous supposons \LPO\@. Mais affirmer \(\alpha<\beta\) pour toutes les suites \((u_n)\) implique  \LPO\ en \coma.
Ici nous voyons que l'hésitation entre  \(1\) et \(2\) pour une suite infinie donne le même sentiment qu'hésiter (dans un contexte classique) entre une suite bornée ou non bornée d'entiers naturels: ajouter $1$ aux termes d'une telle suite n'augmente sa borne supérieure que si la suite est bornée.  
\end{fexample}

\subsection{En mathématiques classiques}\label{fin-class-math}

La \cref{fpropOrdclass} montre que le principe du tiers exclu (\LEM) simplifie et/ou obscurcit dramatiquement l'étude de la structure de $\ord_\fF$ par rapport aux relations~\(<\) et~\(\leq\). 
\begin{fproposition} \label{fpropOrdclass}
Supposons \LEM\@. Alors pour $\alpha,\beta\in\ord$, on a $\alpha\leq \beta$ \hbox{ou $\beta<\alpha$}. En outre, si $\beta<\alpha$, il existe un $i\in \In_\alpha$ tel que $\beta\leq \alpha_i$.
\end{fproposition}
\begin{proof}
On prouve par induction simultanée les deux \prts suivantes. 

\smallskip \centerline{\gui{\(\alpha\leq \beta\) ou \(\beta<\alpha\)}\quad  et \quad \gui{\(\beta\leq \alpha\) ou \(\alpha<\beta\)}.}

\smallskip \noindent Par hypothèse d'induction, on a pour tout \(i\in \In_\alpha\) et tout \(j\in \In_\beta\),
\gui{\(\alpha\leq \beta_j\) ou \(\beta_j<\alpha\)}, et aussi  \gui{\(\beta\leq \alpha_i\) ou \(\alpha_i<\beta\)}. 
\\
La première disjonction implique par \LEM\ que ou bien \(\beta_j<\alpha\) pour tout \(j\in \In_\beta\) ou bien on~a un \(j\in \In_\beta\) tel que \(\alpha\leq\beta_j \). Dans le premier cas on a  \(\beta\leq \alpha\) par définition de \(\cdot\leq \cdots\). Dans le second cas on a \(\alpha<\beta\) par définition de \(\cdot< \cdots\), avec pour \(F\subseteq_f \In_\beta\) la liste \([j]\).
\\
Raisonnement symétrique pour la deuxième disjonction.
\end{proof}
\noindent N. B.: Pour les ordinaux dénombrables le principe limité d'omniscience (\LPO) suffit pour démontrer la proposition.

\begin{fcorollary} \label{fcor1propOrdclass}
Supposons \LEM\@. Tout ordinal \(\alpha\ne\und0\) est ou bien un successeur immédiat, ou bien le \(\sup\) des ordinaux \(\gamma<\alpha\).
\end{fcorollary}
\begin{proof}
Considérons \(\alpha=\s{\alpha_i}{i\in\In_\alpha}\) et comparons \(\alpha\) avec \(\sup(\alpha_i)_{i\in\In_\alpha}\). Les détails sont laissés \alec.
\end{proof}
%

\begin{fcorollary} \label{fcorpropOrdclass}
Supposons \LEM\@. Tout ordinal borné est fini.
\end{fcorollary}
\begin{proof}
Laissé \alec: utiliser le \cref{ffactOr01}.
\end{proof}
%

\section{Résultats fondamentaux}\label{fPropfondOrd}

\subsection{\texorpdfstring{\(\Ord_\fF\) est un objet initial dans la catégorie des \Fords}{Ord\_F est un objet initial dans la catégorie des  F-ordres}}
%
\begin{flemma} \label{flemsupsucc} Soient \(\alpha^1,\dots,\alpha^r\) dans \(\ord\) (\(r\geq 1\)). On~a \[\fbox{\(\sup(\alpha^j)_{j\in\lrbr}<\allowbreak\s{\alpha^j}{j\in\lrbr}\).}\] 
\end{flemma}~
\begin{proof} Montrons \emph{e.~g.}\ que \(\epsilon=\sup(\alpha,\beta)<\gamma=\suc(\alpha,\beta)\). On a \(\In_{\epsilon}=\In_\alpha+\In_\beta\), avec \(\epsilon_k=\alpha_i\) si \(\iota_1(i)=k\), et \(\epsilon_k=\beta_j\) si \(\iota_2(j)=k\). On a \(\In_\gamma=\so{1,2}\)
avec \(\gamma_1=\alpha\) et \(\gamma_2=\beta\). Nous appliquons le \cref{ffactOr1}\emph{b} avec \(F=\so{1,2}\). Pour un \(k\) arbitraire dans  \(\In_\epsilon\), on~a \(\epsilon_k<\alpha,\beta\) car \(\epsilon_k\) est \(\alpha_i\) ou \(\beta_j\) et, par \succu, on~a \(\alpha_i<\alpha\) (à fortiori \(\alpha_i<\alpha,\beta\)) et \(\beta_j<\beta\) (à fortiori \(\beta_j<\alpha,\beta\)). 
\end{proof}
Notons que la démonstration précédente repose sur le fait que les définitions de \(\leq \) et \(<\) ont été données avec des listes sur le côté droit.

\begin{flemma}[transitivités] \label{flemtrans}~
\begin{enumerate}
\item \transu. Si \(\alpha\leq \beta^1,\dots,\beta^m\) et, pour tout
\(j\in\lrbm\), \(\beta^j\leq \gamma^1,\dots,\gamma^r\), alors  
\(\alpha\leq \gamma^1,\dots,\gamma^r\).
\item \transd. Si \(\alpha< \beta^1,\dots,\beta^m\) et, pour tout
\(j\in\lrbm\), \(\beta^j\leq \gamma^1,\dots,\gamma^r\), alors  
\(\alpha< \gamma^1,\dots,\gamma^r\).
\item \transt. Si \(\alpha\leq \beta^1,\dots,\beta^m\) et, pour tout
\(j\in\lrbm\), \(\beta^j< \gamma^1,\dots,\gamma^r\), alors  
\(\alpha< \gamma^1,\dots,\gamma^r\).
\end{enumerate}
\end{flemma}
Comme cas particuliers,  les \cref{ftransun,ftransdeux,ftranstrois}  seront valides quand nous passerons au quotient dans \(\Ord\):
\begin{itemize}
\item si \(\alpha\leq \beta\) et \(\beta\leq \gamma\)
alors \(\alpha\leq \gamma\);
\item   si \(\alpha< \beta\) et \(\beta\leq \gamma\)
alors \(\alpha< \gamma\);
\item   si \(\alpha\leq \beta\) et \(\beta< \gamma\)
alors \(\alpha< \gamma\).
\end{itemize}
\begin{proof} Les trois transitivités vont être démontrées par induction simultanée.

Pour démontrer \transu, on note que l'hypothèse signifie que l'on a $\alpha_i<\beta^1,\dots,\beta^m$ pour tout $i\in\In_\alpha$. Fixons un tel $i$. On utilise  \transd\ avec cet $\alpha_i$ à la place de $\alpha$, on obtient $\alpha_i<\gamma^1,\dots,\gamma^r$. Comme c'est vrai pour tout $i\in\In_\alpha$, cela donne la conclusion souhaitée $\alpha\leq \gamma^1,\dots,\gamma^r$.

Pour démontrer \transd, on note que l'hypothèse implique que l'on a des $G_j\subseteq_f \In_{\beta^j}$ non tous vides tels que 
$\alpha\leq \beta^1_{G_1},\dots,\beta^m_{G_m}$. On a aussi, pour $j\in\lrbm$ et pour tout
$h\in\In_{\beta^j}$, $\beta^j_h<\gamma^1,\dots,\gamma^r$.
À fortiori, cela est vrai pour les $h\in G_j$. On utilise alors \transt\ avec ces $\beta^j_h$ à la place des $\beta^j$.
On obtient la conclusion souhaitée $\alpha< \gamma^1,\dots,\gamma^r$.

Pour démontrer \transt, on note que l'hypothèse implique
(utiliser l'affaiblissement) que l'on a des $F_k\subseteq_f \In_{\gamma^k}$ non tous vides tels que 
$\beta^j\leq \gamma^1_{F_1},\dots,\gamma^r_{F_r}$ pour $j\in\lrbm$.
Cette fois-ci on utilise \transu\ avec des $\gamma^k_\ell$ à la place des $\gamma^k$, on en déduit que $\alpha\leq \gamma^1_{F_1},\dots,\gamma^r_{F_r}$, ce qui implique $\alpha< \gamma^1,\dots,\gamma^r$.
\end{proof}

Le lemme suivant montre que lorsqu'on sera passé au quotient sur  $\Ord$,  l'axiome~\emph{\ref{faxltle}} sera satisfait.  

\hum{Étonnant que ce ne soit pas arrivé bien plus vite !}
\begin{flemma}[\axquattre] \label{flemltle}
  Soit   \(\alpha,\beta^1,\dots,\beta^m\in\ord\). Si \(\alpha<\beta^1,\dots,\beta^m\), alors \(\alpha\leq \beta^1,\dots,\beta^m\).
\end{flemma}
\begin{proof}  Démonstration par induction sur $\alpha$.
On a $ \alpha < \beta^1,\dots,\beta^m$ \ssi il existe  des $F_k\subseteq_f \In_{\beta^k}$ non tous vides tels que, pour chaque $i\in \In_\alpha$, on a $\alpha_i\leq  \beta^1_{F_1},\dots, \beta^m_{F_m}$. 
 Fixons un $i\in\In_\alpha$. Pour $j\in F_k$, on a $\beta^k_j<\beta^k$, et en affaiblissant \hbox{$\beta^k_j<\beta^1,\dots,\beta^m$}. Par \transt,  on obtient $\alpha_i<\beta^1,\dots,\beta^m$. Enfin, comme c'est vrai pour \hbox{tout $i\in\In_\alpha$}, on a $\alpha\leq \beta^1,\dots,\beta^m$.
\end{proof}

Le lemme suivant montre que lorsqu'on sera passé au quotient sur $\Ord$,  l'\cref{fsupsucfini} sera satisfait.
\begin{flemma}[\axsupsucfini] \label{flemaxsupsucfini}
  Si \(\alpha<\gamma\) et \(\beta<\gamma\), alors \(\sup(\alpha,\beta)<\gamma\).
\end{flemma}
\begin{proof}
Par définition, on~a  \(\suc(\alpha,\beta)\leq \gamma\). Le \cref{flemsupsucc}
donne \(\sup(\alpha,\beta)<\suc(\alpha,\beta)\). Par transitivité, on obtient \(\sup(\alpha,\beta)<\gamma\).
\end{proof}
%

\begin{flemma} \label{flemContradic}
 Soient $n>0$ et $\alpha^1,\dots,\alpha^n\in\ord$.
Il  est impossible que, pour chaque $i\in\lrbn$, on \hbox{ait  $\alpha^i < \alpha^1,\dots,\alpha^n$}.
\end{flemma}
\begin{proof} On raisonne par induction. En utilisant l'affaiblissement, l'hypothèse donne des listes finies
non toutes vides 
\[
  \text{\(F_1\subseteq_f\In_{\alpha_1}\), \(\dots\)\kern.16667em, \(F_n\subseteq_f\In_{\alpha_n}\),}
\]  
telles que
\[
  \alpha^i \leq    \alpha^1_{F_1},\dots,\alpha^n_{F_n} \hbox{  pour } i \in\lrbm.
\] 
En particulier, pour $j\in F_i$ (si $F_i$ est non vide) on a
\[
  \alpha^i_j <    \alpha^1_{F_1},\dots,\alpha^n_{F_n}.
\] 
On est ramené à l'hypothèse avec la liste non vide \(\alpha^1_{F_1},\dots,\alpha^n_{F_n}\) qui remplace la \hbox{liste $\alpha^1,\dots,\alpha^n$}.
\end{proof}
%

\begin{flemma} \label{flema<avirguleb} Soient $\alpha^1,\dots,\alpha^n,\beta^1,\dots,\beta^m\in\ord$  $(n,m\geq 1)$.
\begin{enumerate}
\item  Si   $\alpha^i < \alpha^1,\dots,\alpha^n,\beta^1,\dots,\beta^m$ pour $i\in\lrbn$,
alors  $\alpha^i < \beta^1,\dots,\beta^m$ pour chaque $i$.

\item Soient $F_1\subseteq_f\In_{\alpha_1}$, \(\dots\)\kern.16667em, $F_n\subseteq_f\In_{\alpha_n}$. Si $\alpha^i \leq  \alpha^1_{F_1},\dots,\alpha^n_{F_n},\beta^1,\dots,\beta^m$ pour $i\in\lrbn$,
alors  $\alpha^i \leq  \beta^1,\dots,\beta^m$ pour chaque $i$.
\end{enumerate}
\end{flemma}
\begin{proof} 
\emph{1}. En utilisant l'affaiblissement, l'hypothèse donne des listes finies
non toutes vides
\[
  \text{\(F_1\subseteq_f\In_{\alpha^1}\), \(\dots\)\kern.16667em, \(F_n\subseteq_f\In_{\alpha^n}\), \(G_1\subseteq_f\In_{\beta^1}\), \(\dots\)\kern.16667em, \(G_m\subseteq_f\In_{\beta^m}\),}
\]
telles que
\[
  \alpha^i \leq    \alpha^1_{F_1},\dots,\alpha^n_{F_n}, \beta^1_{G_1},\dots,\beta^m_{G_m} \hbox{ pour } i \in\lrbn.
\eqno(*)
\] 
On a alors pour  $i\in\lrbn$ et $j\in\In_{\alpha_i}$ 
\[ 
\begin{array}{ccc} 
\alpha^i_j  <  \alpha^1_{F_1},\dots,\alpha^i_{F_i},\dots,\alpha^n_{F_n}, \beta^1_{G_1},\dots,\beta^m_{G_m}\text.
\end{array}
\]
Fixons \(i\)~et~\(j\): à fortiori, avec \(F'_i=F_i\cup\so j\)
\[ 
\begin{array}{ccc} 
\alpha^i_j  <  \alpha^1_{F_1},\dots,\alpha^i_{F'_i},\dots,\alpha^n_{F_n}, \beta^1_{G_1},\dots,\beta^m_{G_m}.   
\end{array}
\]
On a aussi par affaiblissement, pour \(k\in\lrbn\) et \(\ell\in F_k\),
\[ 
\begin{array}{ccc} 
\alpha^k_\ell  <  \alpha^1_{F_1},\dots,\alpha^i_{F'_i},\dots,\alpha^n_{F_n}, \beta^1_{G_1},\dots,\beta^m_{G_m}.   
\end{array}
\]
Donc par induction
\(\alpha^i_j < \beta^1_{G_1},\dots,\beta^m_{G_m}\).
Puisque \(j\) est arbitraire, nous obtenons \(\alpha^i \leq  \beta^1_{G_1},\allowbreak\dots,\beta^m_{G_m}\).
Ceci donne la conclusion cherchée, \(\alpha^i<\beta^1,\dots,\beta^m\), si au moins une liste~\(G_k\) est non vide, Pour un~\(i\in\lrbn\) arbitraire.
Si ce n'est pas le cas, \((*)\)  donne  \(\alpha^i \leq    \alpha^1_{F_1},\dots,\alpha^n_{F_n}\) pour \(i \in\lrbn\), avec  des listes \(F_i\) non toutes vides. Par définition, cela implique \(\alpha^i <\alpha^1,\dots,\alpha^n\)  pour \(i \in\lrbn\), ce qui est impossible d'après le \cref{flemContradic}.

\noindent \emph{2}.
On a  pour  \(i\in\lrbn\) et \(j\in\In_{\alpha^i}\) 
\[ 
\begin{array}{ccc} 
\alpha^i_j  <  \alpha^1_{F_1},\dots,\alpha^i_{F_i},\dots,\alpha^n_{F_n}, \beta^1,\dots,\beta^m\text.
\end{array}
\]
Fixons \(i\)~et~\(j\): à fortiori, avec \(F'_i=F_i\cup\so j\),
\[ 
\begin{array}{ccc} 
\alpha^i_j  <  \alpha^1_{F_1},\dots,\alpha^i_{F'_i},\dots,\alpha^n_{F_n}, \beta^1,\dots,\beta^m.   
\end{array}
\]
On a aussi par affaiblissement, pour \(k\in\lrbn\) et \(\ell\in F_k\),
\[ 
\begin{array}{ccc} 
\alpha^k_\ell  <  \alpha^1_{F_1},\dots,\alpha^i_{F'_i},\dots,\alpha^n_{F_n}, \beta^1,\dots,\beta^m.   
\end{array}
\]
Le point \emph{1} donne alors
 \(\alpha^i_j  <   \beta^1,\dots,\beta^m\).   
Comme \(j\) est arbitraire, nous obtenons ce que nous voulons: \hbox{\(\alpha^i \leq  \beta^1,\dots,\beta^m\)} pour un~\(i\in\lrbn\) arbitraire. 
\end{proof}

Le résultat suivant montre que les \cref{fax10,fax11} seront valides quand on passera au quotient~\(\Ord\).

\begin{flemma} \label{flem-supleftright}~
\begin{enumerate}[label=\textit{\arabic*.},ref=\textit{\arabic*}]
%
\item \axdouze.  Si \(\alpha<\sup(\alpha,\beta)\), alors \(\alpha<\beta\);
\item \axtreize. Si \(\gamma<\alpha\) et \(\alpha\leq\sup(\beta,\gamma)\), alors \(\alpha\leq \beta\).
\end{enumerate} 
\end{flemma}
\begin{proof} \emph{1}. Supposons \(\alpha<\sup(\alpha,\beta)\).
Le \cref{flemsupetlist} donne \(\alpha<\alpha,\beta\). Le point \emph{1} du \cref{flema<avirguleb} donne \(\alpha<\beta\).
 
\noindent \emph{2}. Supposons $\gamma<\alpha$ et $\alpha\leq\sup(\beta,\gamma)$. La première hypothèse donne $\gamma\leq \alpha_F$ pour $F\subseteq_f\In_\alpha$ non vide. 
La seconde hypothèse donne \(\alpha\leq \gamma,\beta\) (d'après le \cref{flemsupetlist}). 
Par transitivité on~a   \(\alpha\leq  \alpha_F,\beta\).
 Le point~\emph{2} du \cref{flema<avirguleb} donne \(\alpha\leq \beta\).
\end{proof}
%

\begin{ftheorem} \label{fthOrd1}
Nous avons construit \(\Ord\) en tant qu'un \Ford.
\end{ftheorem}
\begin{proof} En utilisant \rfl\ et \transu, on montre d'une part que l'\egt est bien une relation d'\eqvc, et d'autre part que la relation $\leq$ passe au quotient dans $\Ord$.

\noindent 
De la même manière, \transd\ et \transt\ impliquent  que la relation $<$ passe au quotient dans~$\Ord$.

\noindent 
La loi $\sup$ passe au quotient d'après le \cref{fcorfactOr1}.

\noindent  
La loi $\suc$ passe au quotient d'après le  \cref{fcorfactOr1}, point \emph{2}.

\noindent 
Il reste à noter que les  \cref{frefantisym,fzero,fax3,faxltle,ftransun,ftransdeux,ftranstrois,faxsuccunaire,faxsuccunaire2,fsupsucfini,fax10,fax11,faxsup,fax14,fax15} des \Fords ont été démontrés précédemment. Voir, respectivement: 
le \cref{fcorfactOr1} (point~\ref{fcorfactOr1-rfl}); 
le \cref{ffactOr00} (point~\ref{ffactOr001}); 
le \cref{fcorfactOr1} (point~\ref{fcorfactOr1-irfl}); 
le \cref{flemltle}; 
le \cref{flemtrans}; 
le \cref{ffactOr01}; 
le \cref{flemaxsupsucfini}; 
le \cref{flem-supleftright}; 
le \cref{ffactsup}; 
le \cref{fcorfactOr1} (point~\ref{fcorfactOr1-ax14}); 
la \cref{frem-ax15}.
\end{proof}

Le théorème suivant généralise le \cref{ffactnatord}.
\begin{ftheorem} \label{fthOrd2}
L'ensemble \(\Ord\) n'est pas réduit à un point. Plus précisément:
\begin{itemize}
\item pour tous \(\alpha,\beta\in\Ord\),  \(\beta\leq \alpha\) et \(\alpha<\beta\) sont incompatibles;
\item  l'application \(n\mapsto \und n\colon\NN\to\Ord\) 
est injective (\(m<n\) si, et seulement si, \(\und m <\und n\));
\item pour tout \(\alpha\in\Ord\) et  \(n>m\) dans \(\NN\),  il est impossible que \(\suc^{(n)}(\alpha)=_\Ord\suc^{(m)}(\alpha)\).
\end{itemize}
\end{ftheorem}
\begin{proof} 
Le premier point résulte de \irfl\ et de \transd. Le reste suit.
\end{proof}
%

\begin{ftheorem} \label{fthOrd3}
$\Ord$ est objet initial dans la catégorie des \Fords.
\end{ftheorem}
\begin{proof}[Esquisse de démonstration]
L'idée est la suivante: la structure est \gui{purement algébrique} et pour cons\-truire~$\Ord$, on n'a rien fait d'autre qu'utiliser les axiomes.

\noindent  En effet, considérons un objet \((E,<_E,\leq_E,0_E,\sup_E,\suc_E)\) dans la catégorie. Les \elts de \(\ord\) ont leurs copies dans \(E\). Et les relations \(\cdot<\cdots\) et \(\cdot\leq \cdots\) que l'on a définies sur \(\ord\) sont satisfaites dans \(E\) d'après le \cref{ffactltle1} lorsqu'on les interprète dans $E$ via des $\sup$ finis à droite (comme cela est \ncr d'après le \cref{flemsupetlist}).
Cela implique qu'il y a un unique morphisme de $\Ord$ vers~$E$
dans la catégorie considérée.
\end{proof}
%

\subsection{Quelques propriétés supplémentaires}
%

\begin{fproposition} \label{fpropordbienfonde}
La relation binaire $<$ sur $\Ord$ est bien fondée.
\end{fproposition}
\begin{proof}
Conséquence directe du \cref{ffactsousord}.
\end{proof}
%

\begin{flemma}[formes faibles de la disjonction ``\(\alpha\leq\beta\) ou \(\beta<\alpha\)'']
  \label{flemcut} ~
  
\noindent 
Soient \(r\geq 1\) et \(\alpha,\beta^1,\dots,\beta^r,\gamma\in\ord\). 
\begin{enumerate}
\item Si \(\alpha\leq \beta\) et \(\beta<\alpha,\gamma^1,\dots,\gamma^r\), alors  \(\beta<\gamma^1,\dots,\gamma^r\).
\item Si \(\beta< \alpha\) et \(\alpha\leq \beta,\gamma^1,\dots,\gamma^r\), alors  \(\alpha\leq \gamma^1,\dots,\gamma^r\). 
\end{enumerate}
\end{flemma}
\begin{proof}
On pose $\beta=\sup(\beta^1,\dots,\beta^r)$ et l'on est ramené à des \prts déjà démontrées (en utilisant le \cref{flemsupetlist}).
\end{proof}
%

\begin{fdefinition} \label{fdefiordfiltrant}
Un \elt $\beta=\sucr\beta\in\ord$ est dit \emph{filtrant} si pour tout  $F\subseteq_f\In_\beta$, il existe $j\in \In_\beta$ tel que $\sup(\beta_i)_{i\in F}\leq\beta_j$.
\end{fdefinition}

\begin{flemma} \label{ffiltrant}
Pour tout $\alpha\in\ord$, il existe un $\beta\in\ord$ tel que
$\alpha=_\Ord  \beta$ et $\beta$ est filtrant.
\end{flemma}
\begin{proof}
Si $\alpha=\suc(\alpha_i)_{i\in J}$, on note $K$ l'ensemble des parties finiment énumérées de $J$, et pour $F\subseteq_f J$ on note $\beta_F=\sup(\alpha_j)_{j\in F}$. Enfin $\beta=\suc(\beta_F)_{F\in K}$.
\end{proof}

\hum{ 
\begin{fremark} \label{fthOrdComplet} 
(Complétude) ??? 
 On aimerait bien un résultat du genre suivant.
\emph{Pour \(\beta\in\Ord\) on note \(\dar \beta=\sotq{\alpha\in\Ord}{\alpha\leq \beta}\). 
Soit \(\beta\in\Ord\) et \((\gamma^{(\alpha)})_{\alpha\in\dar\beta}\) 
une famille dans \(\Ord\) indexée par~\(\dar\beta\): alors, cette famille admet une borne supérieure  dans \(\Ord\).}
Cela serait  peut-être pratique pour la suite, mais semble hors d'atteinte: il faudrait parcourir tous les noms d'ordinaux \(\leq \beta\), mais ceci ne semble pas former un ensemble. En tout cas, même si c'est un ensemble, il semble que l'on n'a aucun accès raisonnable à la relation d'ordre.
\end{fremark}
}

\subsection{Arithmétique \elr des ordinaux} \label{fsubsec-arit-ord}

\subsubsection*{Addition (séquentielle)}
L'addition  \(\alpha+\beta\) (\(\alpha\) suivi de \(\beta\): l'addition n'est pas  commutative) est définie par induction sur~\(\beta\): 
\[
  \alpha+\und0=\alpha\quad \hbox{et}\quad 
\alpha+\beta=\s{\alpha+\beta_j}{j\in\In_\beta} \;\hbox{ si }\;\beta=\s{\beta_j}{j\in\In_\beta} \in\ord\sta.
\]  

Cette formule pour \(\alpha+\beta\) ne fonctionne que pour le cas \(\In_\beta\neq \NN_0\) (elle donnerait \(\alpha+\und0=\und0\)).
%
Nous avons aussi \(\alpha+\beta=\sup((\alpha+\beta_j)+\und1)_{j\in\In_\beta}\) si \(\In_\beta\neq \NN_0\).

Les propriétés suivantes se démontrent par induction:
\begin{itemize}
\item si \(\alpha\leq \alpha'\) et \(\beta\leq \beta'\), alors 
\(\alpha+\beta\leq \alpha'+\beta'\);
\item \((\alpha+\beta)+\gamma=\alpha+(\beta+\gamma)\);
\item \(\alpha+\und0=\und0+\alpha=\alpha\); 
\item  \(\alpha+\beta\leq \alpha+\gamma\) si, et seulement si, \(\beta\leq\gamma\); 
\item  \(\alpha+\beta< \alpha+\gamma\) si, et seulement si, \(\beta< \gamma\);
\item  \(\alpha=\und1+\alpha\) si, et seulement si, \(\omega\leq  \alpha\);
\item si \(\alpha\leq \gamma\), alors il y a un \(\beta\) tel que \(\gamma=\alpha+\beta\);
\item si \(\alpha< \gamma\), alors il y a un \(\beta\ne\und0\) tel que \(\gamma=\alpha+\beta\).
\end{itemize}

\subsubsection*{Somme séquentielle}
Soit \(J\in\fF\) muni d'une relation d'ordre  $\prec$ bien fondée,
possédant un \elt minimum~$0_J$ détachable. 
Soit $(\beta^j)_{j\in J}$ un \elt de $\Fam(J,\Ord)$. La \gui{somme indexée $\prec$-séquentielle} $\sum_{j\prec \ell}\beta^j$ est définie par induction sur~$\ell$ \hbox{dans $(J,\prec)$}: 
\[
  \som_{j\prec 0_J}\beta^j=0_J\quad \hbox{et}\quad 
\som_{j\prec \ell}\beta^j=
\sup \left(\big(\som_{j\prec k}\beta^j\big)+
\beta^{k}\right)_{k\prec\ell} \hbox{ si }0_J\prec\ell.
\] 
On démontre par induction sur $J$ que si l'on a deux familles $(\beta^j)_{j\in J}$ 
et $(\gamma^j)_{j\in J}$ avec $\beta^j\leq\gamma^j $ pour tout $j\in J$, alors pour tout $\ell\in J$, $\som_{j\prec \ell}\beta^j\leq \som_{j\prec \ell}\gamma^j$.  
Cette construction passe donc au quotient~\(\Ord\).

\begin{fremark} \label{fremsomindexée}  
 Cette construction permet de définir une fonction $\ord_2^\mathrm{Br}\to\ord_2$, où $\ord_2^\mathrm{Br}$ est l'ensemble des noms des ordinaux de Brouwer. Voir \cite{ftroelstra69} et \cite{fbrouwer18,fbrouwer26}. Troelstra traite uniquement les ordinaux de Brouwer dénombrables.
\end{fremark}

\subsubsection*{Multiplication}

On définit  $\alpha\cdot\beta$ par induction sur $\beta\in\ord$: 
\[
  \alpha\cdot\und0=\und0\quad \hbox{et}\quad 
\alpha\cdot\beta=\sup (\alpha\cdot\beta_j +\alpha)_{j\in\In_\beta} \;\hbox{ si }\;\beta=\s{\beta_j}{j\in\In_\beta} \in\ord\sta.
\]  

Les propriétés suivantes se démontrent par induction:
\begin{itemize}
\item  si \(\alpha\leq \alpha'\) et \(\beta\leq \beta'\), alors
\(\alpha\cdot\beta\leq \alpha'\cdot\beta'\);
\item \((\alpha\cdot\beta)\cdot\gamma=\alpha\cdot(\beta\cdot\gamma)\);
\item \(\alpha\cdot\und1=\und1\cdot\alpha=\alpha\); 
\item  
\(\alpha\cdot(\beta+\gamma)=(\alpha\cdot\beta)+(\alpha\cdot\gamma)\);
\item si \(\und1\leq \alpha\), alors \(\alpha\cdot\beta\leq \alpha\cdot\gamma\) si, et seulement si, \(\beta\leq\gamma\); 
\item
si \(\und1\leq \alpha\), alors  \(\alpha\cdot\beta< \alpha\cdot\gamma\) si, et seulement si, \(\beta< \gamma\).
%
\end{itemize}

\subsubsection*{Exponentiation}
On définit  $\alpha^\beta$ par induction sur $\beta\in\ord$: 
\[
  \alpha^{\und0}=\und1\quad \hbox{et}\quad 
\alpha^\beta=\sup ( \alpha^{\beta_j} \cdot \alpha)_{j\in\In_\beta} \;\hbox{ si }\;\beta=\s{\beta_j}{j\in\In_\beta} \in\ord\sta.
\]  

\subsubsection*{Ackermann}

On peut  continuer cette \gui{arithmétique \elr} à  la Ackermann  comme dans \citealt{ffinsler51} en définissant par induction un ordinal $\Acko(\alpha,\beta,\gamma)$ obtenu \gui{en itérant $\gamma$ fois la fonction précédente, initialisée à $\alpha$}, c'est-à-dire de manière plus précise
\[ 
\begin{array}{rcl} 
\Acko(\alpha,\beta,\und0)  & =  & \alpha+\beta  \\[.3em] 
\Acko(\alpha,\und0,\gamma)  & =  &  \alpha \qquad \hbox{ si }\;\gamma\in\ord\sta\\[.3em] 
\Acko(\alpha,\beta,\gamma)  & =  & \sup \bigl(\sup
  (\Acko(\Acko(\alpha,\beta_j,\gamma),\alpha,\gamma_k))_{j\in\In_\beta}\bigr)_{k\in\In_\gamma} \\[.3em]&& \qquad\quad \hbox{ si }\;\beta=\s{\beta_j}{j\in\In_\beta}\;\hbox{ et }\;\gamma=\s{\gamma_k}{k\in\In_\gamma}.
 \end{array}
\]

En particulier, \(\varepsilon_0=\Acko({\omega,\omega,\und4})\).




\section{Ordinaux dénombrables}

\subsection{Premiers pas}
Comme indiqué précédemment, les ordinaux de la seconde classe (les ordinaux dénom\-brables), sont définis en prenant l'ensemble d'indexeurs   
\[\fF_2=\sotq{\NN_k}{k\in\NN, k\geq0} \cup \so\NN\]
muni d'opérations convenables pour l'ensemble des sous-ensembles finis
d'un \(I\in \fF\) et pour les réunions disjointes d'éléments de~\(\fF\) indexés par un élément de~\(\fF\).
On écrit \(\ord_2\) et \(\Ord_2\) pour \(\ord_{\fF_2}\) et \(\Ord_{\fF_2}\).
Donc \(\Ord_2\) est l'ensemble des ordinaux de la seconde classe
tandis que $\ord_2$ est un ensemble de noms pour les \elts de $\Ord_2$.

\begin{flemma} \label{flemsucccroissant}
Tout ordinal dénombrable est le $\suc$ d'une suite croissante d'ordinaux dénombrables.
\end{flemma}
\begin{proof}
C'est le \cref{ffiltrant}.
\end{proof}
%

\begin{fproposition} \label{fpropOrdLPO1} 
Supposons \LPO\@. 
Alors, pour \(\alpha,\beta\in\Ord\), on a  \(\alpha\leq \beta\) ou \(\beta<\alpha\).
\end{fproposition}
\begin{proof} C'est comme la \cref{fpropOrdclass}, pour le cas dénombrable.
\end{proof}
%

%
%

\subsection{Comparaison avec les ordinaux de Martin-Löf}\label{fsubsecOML}

Nous présentons ici une variante de la théorie des ordinaux du livre \emph{Notes on Constructive Mathematics} \citep[Chapter 3]{fPML}. 
Nous disons \gui{variante} car la théorie de Martin-Löf est présentée dans le cadre des mathématiques récursives à la Markov, alors que nous nous situons dans la logique intuitionniste avec les définitions inductives généralisées,
comme dans le travail \citealt{fheyting61} (le fait que ce cadre fournit un traitement plus élégant que celui des mathématiques récursives est souligné dans le rapport de \citet{fkreisel63} sur ce travail).

\subsubsection{Le système formel de Martin-Löf}

Dans ce système, les ordinaux sont décrits de manière inductive: si l'on a une suite finie ou infinie
d'ordinaux \(\sigma = \sigma_0,\dots,\sigma_n,\dots\) (peut-être vide), alors \(\suc(\sigma)\) est un ordinal.

La sémantique classique de cette opération est la suivante: à la suite d'ordinaux \((\sigma_n)\) on associe le supremum de la suite des successeurs des~\(\sigma_n\).

En particulier, \(\und 0\) est défini comme \(\suc(\sigma)\), où \(\sigma\) est la suite vide.

Nous écrivons simplement \(\suc(\alpha)\) pour \(\suc(\sigma)\), où \(\sigma\) est la suite à un élément \(\sigma_0 = \alpha\).

En \coma, l'ensemble de ces ordinaux est un exemple d'ensemble non discret.

Comme décrit dans l'introduction, à chaque ordinal \(\alpha\) nous associons, par induction sur~\(\alpha\), un arbre \(\Tree(\alpha)\):
\(\Tree(\alpha)\) contient toujours la suite vide, et \(\Tree(\suc(\sigma))\) contient \(n\cct \ell\) si
\(\ell\) est dans \(\Tree(\sigma_n)\).

Cet ensemble \(\Tree(\alpha)\) ne contient aucune branche infinie: si \(f\) est une fonction numérique,
nous pouvons toujours trouver un \(n\) tel que \([f(0),\dots,f(n-1)]\) n'est pas dans \(\Tree(\alpha)\). Ceci est démonté directement par induction sur \(\alpha\). Autrement dit, l'arbre \(\Tree(\alpha)\) est bien fondé.

L'ensemble de ces arbres, $\ord_2$, est l'ensemble des noms d'ordinaux, aussi bien chez Martin-Löf que dans notre approche.

Le fait que l'on trouve de cette manière tous les arbres bien fondés
est le contenu du théorème de la barre. Ce théorème de Brouwer n'est valide ni dans la théorie des ensembles de Bishop, ni dans la théorie des types dépendants. Cela résulte du fait que ces deux systèmes ont une interprétation en mathématiques récursives, où le théorème de la barre est faux, comme démontré dans un exemple dû à Kleene \citep[voir][]{fkleenevesley65}. 

Par \dfn, une formule atomique est une formule de la forme \(\alpha<\beta\) ou \(\alpha\leqslant \beta\);
et un séquent est un ensemble fini de formules atomiques. 

Nous définissons maintenant par induction la phrase \gui{le séquent \(\Gamma\) est valide}. La formulation est très élégante!
\[
  \frac{\Gamma,\alpha\leqslant \sigma_n}{\Gamma,\alpha<\suc(\sigma)} ~~~~~~~~
\frac{\cdots\ \Gamma,\sigma_n<\beta\ \cdots}{\Gamma,\suc(\sigma)\leqslant \beta}
\] 
Notez qu'il y a une démonstration directe de  \(\und0\leqslant \beta\) par la deuxième règle, appliquée avec un ensemble vide de prémisses.

La signification intuitive du séquent est la disjonction classique des formules atomiques qu'il contient.

Martin-Löf définit alors une relation d'\eqvc $\alpha=_\ML \beta$ sur $\ord_2$ comme exprimant le fait que les séquents $\alpha\leq \beta$ et $\beta\leq \alpha$ sont valides. L'ensemble des ordinaux de Martin-Löf, noté $\Ord_2^\ML$, est le quotient de $\ord_2$ par cette relation d'\eqvc.

Martin-Löf  prouve le séquent \(\alpha<\beta,\beta\leqslant \alpha\) par induction sur~\(\beta\) et~\(\alpha\).
Il démontre aussi par induction sur \(\alpha\) que la règle suivante est admissible:
\[
  \frac{\Gamma,\alpha<\alpha}{\Gamma}\text,
\] 
ce qui implique en particulier  que \(\alpha<\alpha\) n'est pas démontrable.

Donnons un exemple de telles \dems par induction.

\begin{flemma}
Pour tout \(\alpha\) les séquents \(\alpha\leqslant \alpha\) et  \(\alpha<\suc(\alpha)\) sont valides.
\end{flemma}

\begin{proof}
On démontre \(\alpha\leqslant \alpha\) par induction sur \(\alpha\). Si \(\alpha = \suc(\sigma)\), on doit montrer
  \(\sigma_n<\suc(\sigma)\) pour tout \(n\), ce qui résulte de \(\sigma_n\leqslant \sigma_n\), qui se démontre par induction.

Par suite on~a \(\alpha<\suc(\alpha)\) en utilisant la première règle.
\end{proof}

Martin-Löf peut \egmt démontrer l'analogue du \cref{fthOrd2} pour $\Ord_2^\ML$. Mais les deux énoncés pour $\Ord_2^\ML$ et pour notre $\Ord_2$ sont indépendants l'un de l'autre.

\subsubsection{Comparaison avec notre système}
\label{fsec:comparison}

Expliquons maintenant pourquoi cette \dfn ne coïncide pas avec la nôtre. Pour cela nous donnons un exemple de la forme \(\alpha<\beta\) qui est démontrable dans ce calcul des séquents, mais qui implique  \LPO\ dans notre système.

On reprend l'\cref{fexa123}: on définit \(\alpha = \suc( \sigma)\),
où \(\sigma_n=\und{u_n}\) avec \((u_n)\) une suite croissante (au sens large)
de~\(0\) et de~\(1\), et \(\beta = \suc( \tau)\), où \(\tau_n=s(\sigma_n)=\und{u_n+1}\).

\begin{flemma}
Le séquent \(\alpha<\beta\) est valide.
\end{flemma}

\begin{proof}
D'après la première règle, il suffit de valider le séquent
\(\alpha<\suc( \tau),\alpha\leqslant\tau_0\). Pour cela on doit démontrer pour tout \(n\) le séquent
\(\sigma_n<\tau_0, \alpha<\suc(\tau)\).
Fixons \(n\).

\noindent Si l'on a \( \sigma_n<\tau_0=\suc( \sigma_0)\), c'est OK. Notez
que nous pouvons tester si \ \(\sigma_n<\tau_0\) est valide ou pas car  
\(\sigma_n\) et \(\tau_0\) sont tous deux de la forme~\(\und0\) ou~\(\und1\) ou~\(\und2\).

\noindent Sinon, nous avons explicitement un \(n\) tel que
\( \sigma_n\geqslant\suc(\sigma_0)\) et on~a alors
\(\sigma_m\leqslant \sigma_n\) et donc \(\sigma_m< \tau_n\) pour tout~\(m\).
On démontre \( \sigma_n< \tau_0,\alpha<\suc( \tau)\) comme conséquence de
\(\sigma_n< \tau_0, \alpha\leqslant\tau_n\)
qui est valide parce que
\(\sigma_n<\tau_0,  \sigma_m< \tau_n\)
est valide pour tout~\(m\).
\end{proof}

  Notez que nous démontrons \(\alpha<\suc( \tau)\) en démontrant \(\alpha<\suc( \tau),\alpha\leqslant  \tau_0\), et
  l'on doit ``garder'' \(\alpha<\suc( \tau)\): peut-être \(\alpha\leqslant  \tau_0\) n'est pas valide
  (il se peut que la suite~\((\sigma_n)\) prenne la valeur~\(\und1\) et que \(\tau_0 = \und1\)).

Dans l'\cref{fexa123}, on a vu que \(\alpha<\beta\) implique \LPO\ dans notre système.
Par conséquent, dans l'ensemble \(\Ord_2^\mathrm{ML}\) des ordinaux de  Martin-Löf, l'égalité est plus grossière que dans l'ensemble~\(\Ord_2\)  (les deux sont des quotients de  \(\ord_2\)).

\addcontentsline{toc}{section}{Références}

\bibliographystyle{plainnat-fr}
\markboth{Références}{Références}
\small


\begin{thebibliography}{18}
\providecommand{\natexlab}[1]{#1}
\providecommand{\url}[1]{\texttt{#1}}
\expandafter\ifx\csname urlstyle\endcsname\relax
  \providecommand{\doi}[1]{doi: #1}\else
  \providecommand{\doi}{doi: \begingroup \urlstyle{rm}\Url}\fi

\bibitem[Aczel and Rathjen(2010)]{aczelrathjen10}
Peter Aczel and Michael Rathjen.
\newblock {CST} book draft.
\newblock \url{http://www1.maths.leeds.ac.uk/~rathjen/book.pdf}, 2010.

\bibitem[Bourbaki(1968)]{bourbaki68}
Nicolas Bourbaki.
\newblock \emph{Elements of mathematics: theory of sets}.
\newblock Hermann, Paris and Addison-Wesley, Reading, 1968.
\newblock Translated from the French.

\bibitem[{Brouwer}(1918)]{brouwer18}
L.~E.~J. {Brouwer}.
\newblock {Begründung der Mengenlehre unabhängig vom logischen Satz vom
  ausgeschlossenen Dritten. Erster Teil: allgemeine Mengenlehre.}
\newblock \emph{{Verh. Nederl. Akad. Wetensch. Afd. Natuurk. Sect. 1}},
  12\penalty0 (5):\penalty0 3--43, 1918.

\bibitem[{Brouwer}(1926)]{brouwer26}
L.~E.~J. {Brouwer}.
\newblock {Zur Begründung der intuitionistischen Mathematik. III.}
\newblock \emph{{Math. Ann.}}, 96:\penalty0 451--488, 1926.
\newblock \url{http://eudml.org/doc/159181}.

\bibitem[Church(1938)]{church38}
Alonzo Church.
\newblock The constructive second number class.
\newblock \emph{Bull. Amer. Math. Soc.}, 44:\penalty0 224--232,
  1938.
\newblock \doi{10.1090/S0002-9904-1938-06720-1}.

\bibitem[Dehornoy(2017)]{dehornoy17}
Patrick Dehornoy.
\newblock \emph{La théorie des ensembles: introduction à une théorie
  de l'infini et des grands cardinaux}.
\newblock Tableau Noir, 106. Calvage et Mounet, Paris, 2017.

\bibitem[Finsler(1951)]{finsler51}
Paul Finsler.
\newblock Eine transfinite {F}olge arithmetischer {O}perationen.
\newblock \emph{Comment. Math. Helv.}, 25:\penalty0 75--90, 1951.
\newblock \url{http://eudml.org/doc/139019}.

\bibitem[Gentzen(1936)]{gentzen36}
Gerhard Gentzen.
\newblock {Die Widerspruchsfreiheit der reinen Zahlentheorie}.
\newblock \emph{{Math. Ann.}}, 112:\penalty0 493--565, 1936.
\newblock \url{http://eudml.org/doc/159839}.
\newblock Translation by M. Szabo: The consistency of elementary number
  theory, in \citealt{gen1969}, pages 132--201.

\bibitem[Heyting(1961)]{heyting61}
Arend Heyting.
\newblock Infinitistic methods from a finitist point of view.
\newblock In \emph{Infinitistic methods: proceedings of the symposium on
  foundations of mathematics, {W}arsaw, 2--9 {S}eptember 1959}, pages 185--192.
  Pergamon, Oxford and Pa\'nstwowe Wydawnictwo Naukowe, Warsaw, 1961.

\bibitem[Kleene(1938)]{kleene38}
Stephen~Cole Kleene.
\newblock On notation for ordinal numbers.
\newblock \emph{J. Symb. Log.}, 3:\penalty0 150--155, 1938.
\newblock \url{http://www.jstor.org/stable/2267778}.

\bibitem[Kleene and Vesley(1965)]{kleenevesley65}
Stephen~Cole Kleene and Richard~Eugene Vesley.
\newblock \emph{The foundations of intuitionistic mathematics, especially in
  relation to recursive functions}.
\newblock Studies in Logic and the Foundations of Mathematics. North-Holland, Amsterdam, 1965.

\bibitem[Kraus et~al.(2021)Kraus, Nordvall~Forsberg, and Xu]{krausnordvallxu21}
Nicolai Kraus, Fredrik Nordvall~Forsberg, and Chuangjie Xu.
\newblock Connecting constructive notions of ordinals in homotopy type theory.
\newblock In \emph{46th {I}nternational {S}ymposium on {M}athematical
  {F}oundations of {C}omputer {S}cience (MFCS 2021)}, Leibniz International Proceedings in Informatics (LIPIcs), 202, pages 70:1--70:16. Schloss Dagstuhl -- Leibniz-Zentrum für Informatik, 2021.
\newblock \doi{10.4230/LIPIcs.MFCS.2021.70}.
\newblock \href{http://arxiv.org/abs/2104.02549}{arXiv:\allowbreak2104.02549}
  contains an appendix with proofs.

\bibitem[Kreisel(1963)]{kreisel63}
Georg Kreisel.
\newblock Review of \citealt{heyting61}.
\newblock \emph{Math. Rev.}, 26, 1963.
\newblock \#2363 (MR0144822),
  \url{http://mathscinet.ams.org/mathscinet-getitem?mr=144822}.

\bibitem[Krivine(1998)]{Kri}
Jean-Louis Krivine.
\newblock \emph{Théorie des ensembles}.
\newblock Cassini, Paris, 1998.

\bibitem[Martin-Löf(1970)]{PML}
Per Martin-Löf.
\newblock \emph{Notes on constructive mathematics}.
\newblock Almqvist \& Wiksell, Stockholm, 1970.

\bibitem[Mines et~al.(1988)Mines, Richman, and Ruitenburg]{MRR}
Ray Mines, Fred Richman, and Wim Ruitenburg.
\newblock \emph{A course in constructive algebra}.
\newblock Universitext. Springer, New York, 1988.
\newblock \doi{10.1007/978-1-4419-8640-5}.

\bibitem[Szabo(1969)]{gen1969}
Manfred~E. Szabo, editor.
\newblock \emph{The collected papers of {G}erhard {G}entzen}.
\newblock Studies in Logic and the Foundations of Mathematics.
 North-Holland, Amsterdam, 1969.

\bibitem[Troelstra(1969)]{troelstra69}
Anne S. Troelstra.
\newblock \emph{Principles of intuitionism: lectures presented at the summer
  conference on Intuitionism and Proof theory (1968) at SUNY at Buffalo, N.Y.}
\newblock Lecture Notes in Mathematics, vol. 95. Springer, Berlin, 1969.

\bibitem[{Univalent Foundations Program}(2013)]{hottbook}
{Univalent Foundations Program}.
\newblock \emph{Homotopy type theory: univalent foundations of mathematics}.
\newblock \url{http://homotopytypetheory.org/book}, Institute for Advanced
  Study, 2013.

\end{thebibliography}

\begin{thebibliography}{20}
\expandafter\ifx\csname natexlab\endcsname\relax\def\natexlab#1{#1}\fi
\expandafter\ifx\csname fonteauteurs\endcsname\relax
\def\fonteauteurs{\scshape}\fi
\expandafter\ifx\csname url\endcsname\relax
  \def\url#1{{\tt #1}}%
    \message{You should include the url package}\fi

\bibitem[Aczel et Rathjen(2010)]{faczelrathjen10}
Peter \bgroup\fonteauteurs\bgroup Aczel\egroup\egroup{} et Michael
  \bgroup\fonteauteurs\bgroup Rathjen\egroup\egroup{} :
\newblock {CST} book draft.
\newblock \url{http://www1.maths.leeds.ac.uk/~rathjen/book.pdf}, 2010.

\bibitem[Bourbaki(1970)]{fbourbaki68}
Nicolas \bgroup\fonteauteurs\bgroup Bourbaki\egroup\egroup{} :
\newblock {\em Éléments de mathématique: théorie des ensembles}.
\newblock Hermann, Paris, 1970.
\newblock Nouvelle édition.

\bibitem[{Brouwer}(1918)]{fbrouwer18}
L.~E.~J. \bgroup\fonteauteurs\bgroup {Brouwer}\egroup\egroup{} :
\newblock {Begründung der Mengenlehre unabhängig vom logischen Satz vom
  ausgeschlossenen Dritten. Erster Teil: allgemeine Mengenlehre.}
\newblock {\em {Verh. Nederl. Akad. Wetensch. Afd. Natuurk. Sect. 1}},
  12\penalty0 (5)\string:\penalty500\relax 3-43, 1918.

\bibitem[{Brouwer}(1926)]{fbrouwer26}
L.~E.~J. \bgroup\fonteauteurs\bgroup {Brouwer}\egroup\egroup{} :
\newblock {Zur Begründung der intuitionistischen Mathematik. III.}
\newblock {\em {Math. Ann.}}, 96\string:\penalty500\relax 451-488, 1926.
\newblock \url{http://eudml.org/doc/159181}.

\bibitem[Church(1938)]{fchurch38}
Alonzo \bgroup\fonteauteurs\bgroup Church\egroup\egroup{} :
\newblock The constructive second number class.
\newblock {\em Bull. Amer. Math. Soc.}, 44\string:\penalty500\relax 224-232, 1938. \doi{10.1090/S0002-9904-1938-06720-1}.

\bibitem[Dehornoy(2017)]{fdehornoy17}
Patrick \bgroup\fonteauteurs\bgroup Dehornoy\egroup\egroup{} :
\newblock {\em La théorie des ensembles: introduction à une théorie
  de l'infini et des grands cardinaux}.
\newblock Tableau Noir, 106. Calvage et Mounet, Paris, 2017.

\bibitem[Finsler(1951)]{ffinsler51}
Paul \bgroup\fonteauteurs\bgroup Finsler\egroup\egroup{} :
\newblock Eine transfinite {F}olge arithmetischer {O}perationen.
\newblock {\em Comment. Math. Helv.}, 25\string:\penalty500\relax 75-90, 1951. \url{http://eudml.org/doc/139019}.

\bibitem[Gentzen(1936)]{fgentzen36}
Gerhard \bgroup\fonteauteurs\bgroup Gentzen\egroup\egroup{} :
\newblock {Die Widerspruchsfreiheit der reinen Zahlentheorie}.
\newblock {\em {Math. Ann.}}, 112\string:\penalty500\relax 493-565, 1936.
\newblock \url{http://eudml.org/doc/159839}.
\newblock Traduction par M. Szabo: The consistency of elementary number
  theory, in \citealt{fgen1969}, pages 132-201.

\bibitem[Heyting(1961)]{fheyting61}
Arend \bgroup\fonteauteurs\bgroup Heyting\egroup\egroup{} :
\newblock Infinitistic methods from a finitist point of view.
\newblock \emph{In} {\em Infinitistic methods: proceedings of the symposium on
  foundations of mathematics, {W}arsaw, 2--9 {S}eptember 1959}, pages 185-192.
  Pergamon, Oxford et Pa\'nstwowe Wydawnictwo Naukowe, Varsovie,
  1961.

\bibitem[Kleene(1938)]{fkleene38}
Stephen~Cole \bgroup\fonteauteurs\bgroup Kleene\egroup\egroup{} :
\newblock On notation for ordinal numbers.
\newblock {\em J. Symb. Log.}, 3\string:\penalty500\relax 150-155, 1938.
\newblock \url{http://www.jstor.org/stable/2267778}.

\bibitem[Kleene et Vesley(1965)]{fkleenevesley65}
Stephen~Cole \bgroup\fonteauteurs\bgroup Kleene\egroup\egroup{} et
  Richard~Eugene \bgroup\fonteauteurs\bgroup Vesley\egroup\egroup{} :
\newblock {\em The foundations of intuitionistic mathematics, especially in
  relation to recursive functions}.
\newblock Studies in Logic and the Foundations of Mathematics. North-Holland, Amsterdam, 1965.

\bibitem[Kraus \emph{et~al.}(2021)Kraus, Nordvall~Forsberg et
  Xu]{fkrausnordvallxu21}
Nicolai \bgroup\fonteauteurs\bgroup Kraus\egroup\egroup{}, Fredrik
  \bgroup\fonteauteurs\bgroup Nordvall~Forsberg\egroup\egroup{} et Chuangjie
  \bgroup\fonteauteurs\bgroup Xu\egroup\egroup{} :
\newblock Connecting constructive notions of ordinals in homotopy type theory.
\newblock \emph{In} {\em 46th {I}nternational {S}ymposium on {M}athematical
  {F}oundations of {C}omputer {S}cience (MFCS 2021)}, Leibniz International Proceedings in Informatics (LIPIcs), 202, pages 70\string:\penalty500\relax1-70\string:\penalty500\relax16. Schloss Dagstuhl -- Leibniz-Zentrum für Informatik, 2021. \doi{10.4230/LIPIcs.MFCS.2021.70}.
\newblock \href{http://arxiv.org/abs/2104.02549}{arXiv:\allowbreak2104.02549}
  contient un appendice avec des démonstrations.

\bibitem[Kreisel(1963)]{fkreisel63}
Georg \bgroup\fonteauteurs\bgroup Kreisel\egroup\egroup{} :
\newblock Recension de \citealt{fheyting61}.
\newblock {\em Math. Rev.}, 26, 1963.
\newblock \#2363 (MR0144822),
  \url{http://mathscinet.ams.org/mathscinet-getitem?mr=144822}.

\bibitem[Krivine(1998)]{fKri}
Jean-Louis \bgroup\fonteauteurs\bgroup Krivine\egroup\egroup{} :
\newblock {\em Théorie des ensembles}.
\newblock Cassini, Paris, 1998.

\bibitem[Martin-Löf(1970)]{fPML}
Per \bgroup\fonteauteurs\bgroup Martin-Löf\egroup\egroup{} :
\newblock {\em Notes on constructive mathematics}.
\newblock Almqvist \& Wiksell, Stockholm, 1970.

\bibitem[Mines \emph{et~al.}(1988)Mines, Richman et Ruitenburg]{fMRR}
Ray \bgroup\fonteauteurs\bgroup Mines\egroup\egroup{}, Fred
  \bgroup\fonteauteurs\bgroup Richman\egroup\egroup{} et Wim
  \bgroup\fonteauteurs\bgroup Ruitenburg\egroup\egroup{} :
\newblock {\em A course in constructive algebra}.
\newblock Universitext. Springer, New York, 1988.

\bibitem[Szabo(1969)]{fgen1969}
Manfred~E. \bgroup\fonteauteurs\bgroup Szabo\egroup\egroup{}, éditeur.
\newblock {\em The collected papers of {G}erhard {G}entzen}.
\newblock Studies in Logic and the Foundations of Mathematics.
  North-Holland, Amsterdam, 1969.

\bibitem[Troelstra(1969)]{ftroelstra69}
Anne~S. \bgroup\fonteauteurs\bgroup Troelstra\egroup\egroup{} :
\newblock {\em Principles of intuitionism: lectures presented at the summer
  conference on intuitionism and proof theory (1968) at SUNY at Buffalo, N.Y.}
\newblock Lecture Notes in Mathematics, 95. Springer, Berlin, 1969.

\bibitem[{Univalent Foundations Program}(2013)]{fhottbook}
\bgroup\fonteauteurs\bgroup {Univalent Foundations Program}\egroup\egroup{} :
\newblock {\em Homotopy type theory: univalent foundations of mathematics}.
\newblock \url{http://homotopytypetheory.org/book}, Institute for Advanced
  Study, 2013.

\end{thebibliography}

\end{document}